\documentclass[11pt]{article}

\usepackage[a4paper,margin=1in]{geometry}

\usepackage{amsmath,amssymb,amsthm,mathtools,bbm}
\usepackage{enumitem}
\usepackage{hyperref}
\usepackage{cite}
\usepackage{xcolor}
\usepackage{mdframed}
\usepackage{booktabs}
\usepackage{longtable}
\usepackage{array}
\usepackage{microtype}
\usepackage{tikz}
\usetikzlibrary{arrows.meta, positioning, calc, matrix, fit, decorations.pathreplacing}
\usepackage{subfiles,subcaption}

\usepackage{graphicx}
\usepackage[percent]{overpic}
\usepackage{subcaption}
\newcommand{\R}{\mathbb{R}}
\newcommand{\E}{\mathbb{E}}
\newcommand{\Pp}{\mathbb{P}}
\newcommand{\1}{\mathbbm{1}}
\newcommand{\abs}[1]{\left|#1\right|}

\newcommand{\tauv}{\tau_{\mathrm v}}
\newcommand{\tauc}{\tau_{\mathrm c}}
\newcommand{\Vth}{V_{\mathrm{th}}}
\newcommand{\Vr}{V_{\mathrm r}}
\newcommand{\Ith}{I_{\mathrm{th}}}

\newcommand{\essinf}{\operatorname*{ess\,inf}}
\newtheorem{theorem}{Theorem}[section]
\newtheorem{lemma}[theorem]{Lemma}
\newtheorem{proposition}[theorem]{Proposition}
\newtheorem{corollary}[theorem]{Corollary}

\theoremstyle{definition}
\newtheorem{definition}[theorem]{Definition}
\newtheorem{assumption}[theorem]{Assumption}

\theoremstyle{remark}
\newtheorem{remark}[theorem]{Remark}

\newtheorem*{theorem*}{Theorem}

\newmdenv[
linewidth=0.8pt,
roundcorner=4pt,
innerleftmargin=8pt,
innerrightmargin=8pt,
innertopmargin=6pt,
innerbottommargin=6pt,
skipabove=\baselineskip,
skipbelow=\baselineskip
]{resultbox}

\newmdenv[
linewidth=0.6pt,
roundcorner=4pt,
innerleftmargin=8pt,
innerrightmargin=8pt,
innertopmargin=6pt,
innerbottommargin=6pt,
skipabove=0.75\baselineskip,
skipbelow=0.75\baselineskip,
backgroundcolor=gray!7
]{intuitionbox}

\usepackage{booktabs}
\title{Numerical analysis for leaky-integrate-fire networks under Euler--Maruyama}
\author{Xu'an Dou, Frank Chen, Kevin K Lin, Zhuo-Cheng Xiao$^*$}
\date{\today \\$^*$ zx555@nyu.edu}

\begin{document}
\maketitle
\begin{abstract}
Leaky integrate-and-fire (LIF) networks are standard reduced models for spike-based neural dynamics and a natural substrate for neuromorphic computation. We study time-driven Euler--Maruyama simulation of current-based LIF networks with exponentially decaying synapses and instantaneous resets. Because diffusion acts through the synaptic current rather than directly through the voltage, numerical error is concentrated at threshold events: it is governed by spike-time perturbations and by grid-induced spike-count mismatch.

For layered feedforward networks, under the stated density, rate, regularity, and one-step boundary-layer assumptions, we derive finite-horizon strong and weak error bounds. On matched spike histories up to the observation horizon, the strong analysis combines a conditional single-spike hitting-time comparison with direct averaging of the resulting synaptic-impact kernel against the boundary flux of crossing speeds. This yields matched-trajectory mean-square strong error of order $h$ up to polylogarithmic factors. Horizon spike-count mismatch is then controlled separately by local rate, dense-spike, and spike-time-tail bounds. As for weak errors, an averaged backward-semigroup argument---conditional on a backward transmission problem and on one-step rate/strip/factorial-moment controls for the numerical marginals---gives weak order $1$ for smooth spike-map-compatible observables, with constants explicit in rate and weight bounds. Furthermore, with the derivation of a Lyapunov exponent formula coupling the stationary threshold flux to the reset saltation factor, we outline deterministic and noisy recurrent extensions, including loop-truncated matched strong bounds governed by synaptic cycles and recurrent averaged weak bounds governed by stepwise spike moments. These results separate single-trial spike-train fidelity from observable-level accuracy and clarify which notion of numerical accuracy is relevant for mechanistic spiking models and for spike-based computation.
\end{abstract}

\tableofcontents

\section{Introduction}\label{sec:intro}

\subsection{Motivation: why spiking resets change numerical analysis}

Leaky integrate-and-fire (LIF) networks are among the simplest spiking models that still retain threshold crossing, reset, and synaptic filtering \cite{GerstnerEtAl2014,BretteReview2007}. They therefore play a useful role in two communities. In computational neuroscience, they are standard reduced models for fluctuation-driven spiking and network dynamics \cite{GerstnerEtAl2014,BretteReview2007}. In neuromorphic engineering and spike-based AI, the same event-driven structure makes them natural building blocks for energy-efficient temporal processing and online computation \cite{Maass1997,Mead1990,IndiveriEtAl2011,IndiveriLiu2015,Roy2019,Schuman2022}.

For such systems, numerical accuracy is not a single notion. When the scientific target is a single-trial spike train, a synchrony pattern, a spike-timing-dependent plasticity update, or a causal spike ordering, pathwise accuracy is essential. When the target is a smooth readout, a firing-rate statistic, or an averaged loss functional, weak accuracy is the more relevant criterion \cite{BretteGuigon2003,NeftciMostafaZenke2019,Eshraghian2023}. A useful numerical theory for spiking networks should therefore distinguish strong and weak error rather than treat them as interchangeable.

The biggest challenge is that LIF dynamics is hybrid, which distinguishes them from standard numerical analysis for stochastic differential equations (SDE). Between spikes, the state evolves smoothly. At a spike, the voltage resets discontinuously and synaptic currents jump. In the current-based model studied here, diffusion enters through the synaptic current but not directly through the voltage. Each spike is therefore a threshold hit for a deterministically advected voltage with random crossing speed $A=I-\Ith$. When $A$ is small, a small state perturbation can produce a large spike-time shift. The main numerical difficulty is thus concentrated at threshold events rather than in the bulk subthreshold evolution. Standard SDE convergence theory, usually formulated for globally Lipschitz coefficients or for
jumps at prescribed times, does not directly capture missed spikes, delayed spikes, or reset-induced propagation of timing error; see the simulation literature
\cite{HanselMatoMeunier1998,ShelleyTao2001,MorrisonDiesmann2007,Hanuschkin2010}.

The paper studies time-driven Euler--Maruyama simulation for current-based LIF networks with exponentially decaying synapses.  We first analyze layered feedforward networks, where the propagation of spike-time perturbations can be tracked explicitly across depth, and then discuss deterministic and stochastic recurrent extensions.  In the feedforward strong analysis, the key object is a matched-trajectory spike-time impact functional: we first compare one exact and one numerical spike through a local hitting-time problem, then average the resulting impact kernel against the boundary flux of crossing speeds.  Terminal spike-count mismatch is treated separately.  For weak error, the key object is instead an averaged one-step defect along the numerical chain: transmission at threshold removes the leading jump defect, and the remaining boundary-strip contribution is small only after averaging over the step-start law.  Our goal is to clarify which parts of the classical Euler--Maruyama picture survive in the presence of threshold/reset events and which parts must be replaced by event-based estimates.

\subsection{Related work}
\paragraph{Spiking models in neuroscience and neuromorphic computation.}
Integrate-and-fire models have long served as reduced descriptions of neural excitability,
fluctuation-driven spiking, and synaptic integration \cite{GerstnerEtAl2014}.
They also sit at the core of neuromorphic engineering, from Mead's original event-based vision
\cite{Mead1990} to silicon neuron circuits, large spiking chips, and modern spike-based AI
\cite{IndiveriEtAl2011,Merolla2014,Davies2018,Roy2019,Schuman2022}.
This breadth of use makes numerical fidelity a shared issue: the same simulator may be used to test
neural hypotheses, benchmark hardware mappings, or train and analyze spike-based models
\cite{NeftciMostafaZenke2019,Eshraghian2023}.

\paragraph{Simulation algorithms and numerical pitfalls.}
Time-driven simulation of integrate-and-fire networks is classical, and the practical numerical issues are well known.
Early analyses and practical remedies include \cite{HanselMatoMeunier1998,ShelleyTao2001}, as well as
spike-time interpolation and exact subthreshold integration schemes
\cite{MorrisonDiesmann2007,Hanuschkin2010}.
A broad survey is given in \cite{BretteReview2007}.
Most of this literature is algorithmic or software-oriented. By contrast, rigorous finite-horizon strong and weak error estimates for trajectory simulation of finite spiking networks remain limited.

\paragraph{Mean-field and Fokker--Planck discretizations.}
A complementary line of work studies mean-field and Fokker--Planck descriptions of noisy LIF models and develops structure-preserving schemes at the PDE level \cite{HuangLiuXuZhou2021}. Those results are important for distributional evolution and firing-rate statistics, but they do not directly address pathwise spike-train fidelity for finite networks.

\paragraph{Exit times and boundary effects.} At the single-neuron level, the spike time is a first-passage time.  For diffusion-driven exits, Euler monitoring on a grid typically yields a $1/4$ strong order and requires boundary corrections for weak order one \cite{Gobet2000,BouchardMenozzi2009,BouchardGeissGobet2017}.  Our setting is different: diffusion acts on the current rather than directly on the voltage, and the singular quantities are the crossing speed at threshold and the associated local spike-time sensitivity. This leads naturally to a conditional hitting-time analysis, direct boundary-flux averaging of the one-spike impact kernel, and a separate estimate for horizon spike-count mismatch, rather than to the standard nondegenerate exit-time theory.  For weak error, the boundary contribution is handled through an averaged strip estimate along the numerical chain rather than through a pointwise operator expansion at the boundary.  The resulting feedforward matched-trajectory strong bound has mean-square order $h$ up to polylogarithmic factors.

\paragraph{Dynamical stability and Lyapunov exponents in integrate-and-fire models.}
The paper also interacts with work on dynamical stability and reliability in spiking systems. For one-dimensional integrate-and-fire models, the reset contributes a saltation factor to infinitesimal perturbations \cite{Brette2004}. For common-input reliability and spike-based temporal computation, Lyapunov exponents therefore provide a natural description of long-time amplification or contraction \cite{BretteGuigon2003,TeramaeTanaka2004}. We use this viewpoint to connect spike-time sensitivity to numerical error growth.

\subsection{Contributions and interpretation}
From a neuroscience viewpoint, the strong results concern single-trial spike-train fidelity,
whereas the weak results concern averaged or smooth observables such as firing-rate readouts.
From a neuromorphic-computation viewpoint, the same distinction separates event-stream accuracy from
correctness of network-level functionals.
With this interpretation in mind, the main messages are as follows (proved in
Sections~\ref{sec:model}--\ref{sec:recurrent_weak}):
\begin{enumerate}[label=(\roman*),leftmargin=2.2em]
\item \textbf{Feedforward strong error through matched spike-time analysis.}
For layered feedforward current-based LIF networks, we first work on spike histories whose exact and numerical spike counts match up to the observation horizon. A conditional one-spike hitting-time comparison, combined with direct averaging of the resulting impact kernel against the threshold flux, yields a matched-trajectory recursion for the expected strong error and mean-square order $h$ up to polylogarithmic factors. Terminal spike-count mismatch is estimated separately by local rate, dense-spike, and spike-time-tail bounds.

\item \textbf{Feedforward weak order one for smooth observables.}
For spike-map-compatible observables, and under the backward-regularity plus averaged one-step rate/strip/factorial-moment assumptions stated in Section~\ref{sec:weak}, we derive a weak order-$1$ bound by combining a backward semigroup argument with an averaged boundary-strip estimate for spike-map decisions along the numerical chain.

\item \textbf{A stability viewpoint via saltation and flux.}
For a single current-based LIF neuron driven by a stationary input, we derive a Lyapunov-exponent formula that combines continuous-time contraction with the reset saltation factor. In feedforward networks this leads to a max-type propagation rule for layerwise exponents and clarifies how dynamical stability interacts with long-time error amplification.

\item \textbf{Deterministic and recurrent extensions.}
We formulate deterministic recurrent and noisy recurrent extensions that identify hybrid Lyapunov growth, recurrent gain, rate moments, effective causal depth, and boundary densities as the natural quantities governing long-time amplification. In the noisy recurrent setting, the matched strong analysis again keeps crossing speeds inside a boundary-flux averaged impact term. These extensions show what must be controlled in a genuinely recurrent theory, even when the strongest recurrent bounds require additional assumptions.
\end{enumerate}

\subsection{Notation}\label{subsec:notation}
Table~\ref{tab:notation} collects the main notation used throughout the manuscript.
Each symbol is defined at its first occurrence in the text; the table is intended as a quick reference.

\small
\setlength{\LTpre}{0pt}
\setlength{\LTpost}{0pt}
\begin{longtable}{>{\raggedright\arraybackslash}p{0.22\textwidth} >{\raggedright\arraybackslash}p{0.73\textwidth}}
\caption{Notation used throughout the manuscript.}\label{tab:notation}\\
\toprule
Symbol & Meaning \\
\midrule
\endfirsthead
\toprule
Symbol & Meaning \\
\midrule
\endhead
\midrule
\multicolumn{2}{r}{\emph{(continued on next page)}}\\
\endfoot
\bottomrule
\endlastfoot
$L$ & Feedforward depth (number of layers).\\
$n_\ell$ & Number of neurons in layer $\ell$; $N=\sum_{\ell=1}^L n_\ell$.\\
$\ell$ & Layer index ($1\le \ell\le L$); $i,j$ denote neuron indices within a layer.\\
$\mathcal{J}$ & Index set of neurons: $\mathcal{J}=\{(\ell,j):1\le \ell\le L,\ 1\le j\le n_\ell\}$.\\
$\tauv,\tauc$ & Voltage and synaptic-current time constants.\\
$\Vth,\Vr$ & Voltage threshold and reset level ($\Vr<\Vth$).\\
$\Ith$ & Threshold current: $\Ith=(\Vth-\Vr)/\tauv$.\\
$v_{\ell,j}(t)$ & Voltage of neuron $(\ell,j)$ at time $t$.\\
$I_{\ell,j}(t)$ & Synaptic current of neuron $(\ell,j)$ at time $t$.\\
$b_{\ell,j}(t)$ & Deterministic input (drift) in the current equation.\\
$\sigma_{\ell,j}$ & Brownian noise amplitude in the current equation.\\
$B_{\ell,j}(t)$ & Independent standard Brownian motions.\\
$W^{\ell-1,\ell}_{ji}$ & Feedforward weight from presynaptic neuron $(\ell-1,i)$ to postsynaptic neuron $(\ell,j)$.\\
$s^k_{\ell,j}$ & $k$th exact spike time of neuron $(\ell,j)$ (upward threshold hit).\\
$N_{\ell,j}(t)$ & Exact spike count up to time $t$: $N_{\ell,j}(t)=\sum_k\mathbf 1_{\{s^k_{\ell,j}\le t\}}$.\\
$A^k_{\ell,j}$ & Threshold crossing speed at $s^k_{\ell,j}$: $A^k_{\ell,j}=I_{\ell,j}((s^k_{\ell,j})^-)-\Ith$.\\
$h$ & Time step; $t_m=mh$ are grid times.\\
$\Delta B_{\ell,j,m}$ & Brownian increment $B_{\ell,j}(t_{m+1})-B_{\ell,j}(t_m)\sim\mathcal N(0,h)$.\\
$v^h_{\ell,j}(t_m),\,I^h_{\ell,j}(t_m)$ & Numerical approximations at grid times (Euler / Euler--Maruyama).\\
$\hat s^k_{\ell,j}$ & $k$th numerical spike time (grid-based detection).\\
$N^h_{\ell,j}(t)$ & Numerical spike count up to time $t$.\\
$\Delta N^h_{\ell,j}(m)$ & Numerical spikes of neuron $(\ell,j)$ in $[t_m,t_{m+1})$ (typically $0$ or $1$).\\
$\varepsilon^k_{\ell,j}$ & Spike-time error (STE): $\varepsilon^k_{\ell,j}=\hat s^k_{\ell,j}-s^k_{\ell,j}$.\\
$\alpha_0$ & Fixed upper cutoff in the boundary-depletion assumption near zero crossing speed.\\
$B^{\mathrm{mis}}_{\ell,j}(h,T)$ & Horizon mismatch event: $N^h_{\ell,j}(T)\neq N_{\ell,j}(T)$.\\
$\mathcal G^{\mathrm{match}}(h,T)$ & Network-level matched event:
$\bigcap_{\ell=1}^L\bigcap_{j=1}^{n_\ell}\{N^h_{\ell,j}(T)=N_{\ell,j}(T)\}$.\\
$\mathcal M_{\ell,j}(T)$ & Matched-trajectory STE impact for neuron $(\ell,j)$:
$\E[\sum_{k:s^k_{\ell,j}\le T}\psi(\abs{\varepsilon^k_{\ell,j}})\mathbf 1_{\mathcal G^{\mathrm{match}}(h,T)}]$.\\
$\mathcal M_\ell(T)$ & Layerwise matched-trajectory STE impact envelope:
$\max_{1\le j\le n_\ell}\mathcal M_{\ell,j}(T)$.\\
$R_\ell(h,T)$ & Layerwise local-gap scale entering the conditional single-spike impact kernel.\\
$C^{\mathrm{loc}}_{\ell,j}(T)$ & Subthreshold EM strong-error coefficient for neuron $(\ell,j)$ on $[0,T]$.\\
$C_{\max}$ & A (model-dependent) uniform bound used to control the strong error on the mismatch set.\\
$\Lambda_{\ell,j}(T)$ & Spike-load bound used in feedforward/recurrent strong propagation estimates.\\
$\rho_{\max,\ell,j}(T,\alpha_0)$ & Uniform bound on the joint density at threshold on $I\in[\Ith,\Ith+\alpha_0]$; used in boundary-flux and boundary-layer estimates.\\
$\varrho^V_{\ell,j}(v,t)$ & Voltage marginal density of $v_{\ell,j}(t)$.\\
$\rho^V_{\max,\ell,j}(T)$ & Uniform bound on $\varrho^V_{\ell,j}$ in a unit strip below $\Vth$ (used for false positives and weak error).\\
$\mu_{\ell,j}(t),\,\sigma_{\mathrm{eff},\ell,j}(t)$ & Effective drift and diffusion in an OU surrogate for $I_{\ell,j}(t)$.\\[2pt]
$\mathsf{D}$ & Interior state space for weak error: $\mathsf{D}=\{x=(v,I):v_{\ell,j}<\Vth\ \forall(\ell,j)\}$.\\
$\Sigma_p$ & Threshold surface for neuron $p$: $\Sigma_p=\{x\in\overline{\mathsf{D}}:v_p=\Vth\}$.\\
$\mathcal{R}_p$ & Spike map (reset of $v_p$ and synaptic jump(s) in postsynaptic currents).\\
$X(t)$ & Exact Markov process on $\mathsf{D}$ with resets/jumps; $P_{s,t}$ is its semigroup.\\
$Q_h$ & One-step EM Markov operator (grid spike-map composition).\\
$u(t,x)$ & Backward value function $u(t,x)=(P_{t,T}\Phi)(x)$ for weak error analysis.\\
$R_\ell(T)$ & Averaged exact one-step layer rate bound used in weak error: if $\Delta \widetilde N_\ell^{(m)}$ denotes layer-$\ell$ spikes of the exact step started from $X_m^h$, then $\E[\Delta \widetilde N_\ell^{(m)}]\le R_\ell(T)h$.\\
$R_{\mathrm{ff},2}(T)$ & Feedforward one-step second factorial-moment bound: $\E[\Delta \widetilde N_{\mathrm{tot}}^{(m)}(\Delta \widetilde N_{\mathrm{tot}}^{(m)}-1)]\le R_{\mathrm{ff},2}(T)h^2$.\\
$C_{\eta,p}(T)$ & One-step boundary-layer remainder constant in the averaged feedforward weak theorem: $\E[\eta_{p,m}^2]\le C_{\eta,p}(T)h^3$.\\
$\|W^{\ell,\ell+1}\|_1$ & Column-sum norm of a feedforward weight matrix (used in weak-error constants).\\[2pt]
$\lambda$ & Lyapunov exponent for common-noise synchronization/decoupling.\\
$\rho_\infty(\Vth,I)$ & Stationary boundary density (Fokker--Planck flux at threshold) used in $\lambda$.\\
$S(I)$ & Saltation (reset) factor at a spike with current $I$ (hybrid variational jump).\\[2pt]
$N_{\mathrm r}$ & Number of neurons in a recurrent network.\\
$W$ & Recurrent weight matrix; $\|W\|_1$ is its column-sum norm.\\
$L^\ast$ & Length of the shortest directed synaptic cycle (take $L^\ast=\infty$ if the graph is acyclic).\\
$K_{\mathrm{eff}}(T)$ & Effective causal depth up to time $T$ (maximal length of a time-ordered synaptic influence chain on the matched set).\\
$H(T,h)$ & Recurrent matched-impact gain matrix with entries defined in \eqref{eq:H_rec_def}.\\
$R_{\mathrm{net}}(T)$ & Averaged recurrent exact one-step rate bound: $\E[\Delta \widetilde N_{\mathrm{tot}}^{(m)}]\le R_{\mathrm{net}}(T)h$.\\
$R_{\mathrm{net},2}(T)$ & Averaged recurrent one-step second factorial-moment bound:
$\E[\Delta \widetilde N_{\mathrm{tot}}^{(m)}(\Delta \widetilde N_{\mathrm{tot}}^{(m)}-1)]\le R_{\mathrm{net},2}(T)h^2$.\\
\end{longtable}
\normalsize

\subsection{Organization}
Section~\ref{sec:model} introduces the model and EM discretization and summarizes the main results.
Section~\ref{sec:strong} contains the STE-based strong analysis, first on matched trajectories through
boundary-flux averaging and then for horizon spike-count mismatch.
Section~\ref{sec:weak} provides the averaged weak order-$1$ theorem with explicit constants.
Section~\ref{sec:stability} derives Lyapunov exponents and their feedforward propagation, and
discusses implications for long-time strong error.
Sections~\ref{sec:recurrent_strong}--\ref{sec:recurrent_weak} extend strong/weak bounds to recurrent
networks under deterministic transversality and/or rate \& density controls.
The manuscript ends with a brief conclusion.

\section{Feedforward current-based LIF network and EM scheme}\label{sec:model}

\subsection{Architecture, state variables, and spike trains}
Fix a depth $L\ge 1$. Layer $\ell$ contains $n_\ell$ neurons.
For neuron $(\ell,j)$ let $v_{\ell,j}(t)$ be voltage and $I_{\ell,j}(t)$ be synaptic current.
Spike times $\{s^k_{\ell,j}\}_{k\ge 1}$ are successive threshold hits (we use the left limit because $v$ is reset at spikes):
\[
s^1_{\ell,j}:=\inf\{t>0:\ v_{\ell,j}(t^-)=\Vth\},
\quad
s^{k+1}_{\ell,j}:=\inf\{t>s^k_{\ell,j}:\ v_{\ell,j}(t^-)=\Vth\}.
\]
The counting process is $N_{\ell,j}(t)=\sum_{k\ge 1}\1_{\{s^k_{\ell,j}\le t\}}$.


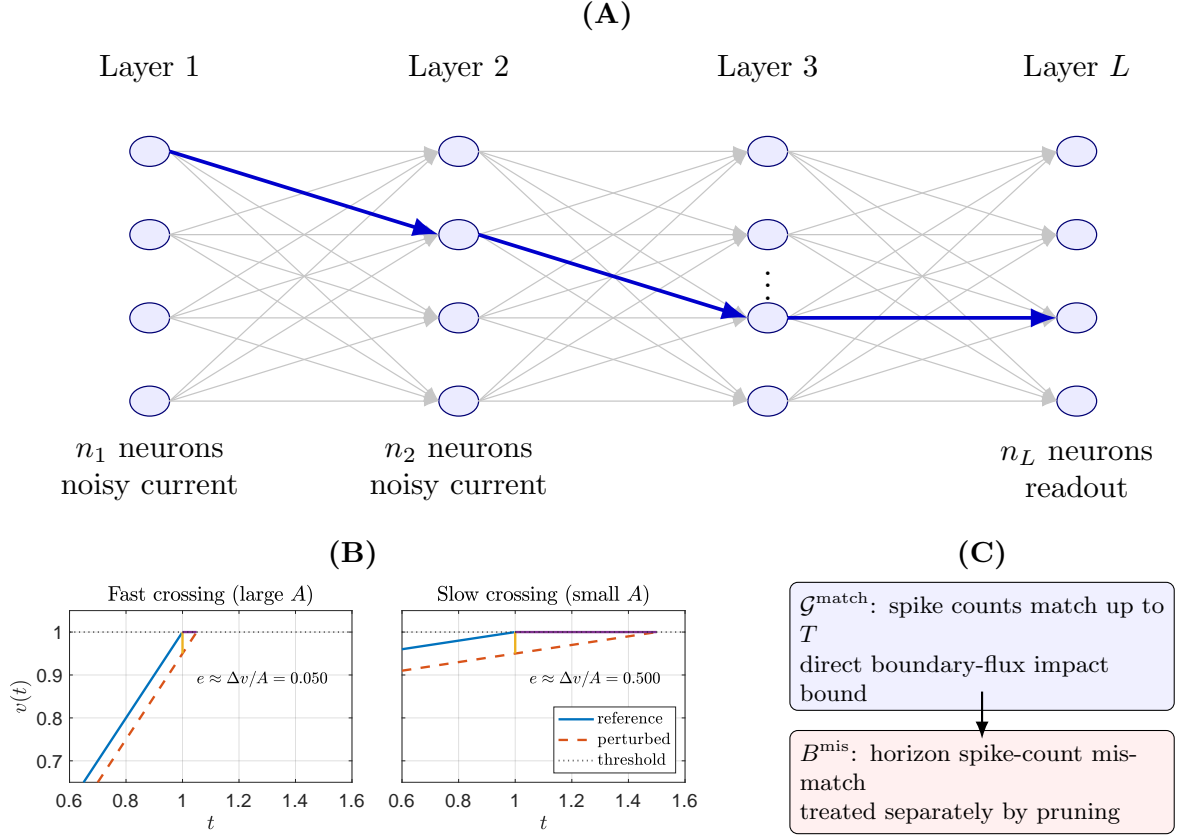
\begin{figure}[t]
\centering

\begin{minipage}{0.96\textwidth}
\centering
\textbf{(A)}\\[0.35em]
\resizebox{0.96\linewidth}{!}{%
\begin{tikzpicture}[x=1.25cm,y=0.95cm,>=Latex]
  \foreach \x/\lab in {0/{Layer 1},2.8/{Layer 2},5.6/{Layer 3},8.4/{Layer $L$}}{
    \node[font=\small] at (\x,2.5) {\lab};
  }
  \foreach \x in {0,2.8,5.6,8.4}{
    \foreach \y in {1.5,0.5,-0.5,-1.5}{
      \filldraw[fill=blue!8,draw=blue!45!black] (\x,\y) circle (0.18);
    }
  }
  \node at (5.6,0) {$\vdots$};
  \foreach \ya in {1.5,0.5,-0.5,-1.5}{
    \foreach \yb in {1.5,0.5,-0.5,-1.5}{
      \draw[gray!45,thin,->] (0.18,\ya) -- (2.62,\yb);
      \draw[gray!45,thin,->] (2.98,\ya) -- (5.42,\yb);
      \draw[gray!45,thin,->] (5.78,\ya) -- (8.22,\yb);
    }
  }
  \draw[blue!80!black,very thick,->] (0.18,1.5) -- (2.62,0.5);
  \draw[blue!80!black,very thick,->] (2.98,0.5) -- (5.42,-0.5);
  \draw[blue!80!black,very thick,->] (5.78,-0.5) -- (8.22,-0.5);
  \node[align=center,font=\small] at (0,-2.35) {$n_1$ neurons\\noisy current};
  \node[align=center,font=\small] at (2.8,-2.35) {$n_2$ neurons\\noisy current};
  \node[align=center,font=\small] at (8.4,-2.35) {$n_L$ neurons\\readout};
\end{tikzpicture}%
}
\end{minipage}

\vspace{0.9em}

\begin{minipage}[t]{0.58\textwidth}
\centering
\textbf{(B)}\\[0.35em]
\includegraphics[width=\linewidth]{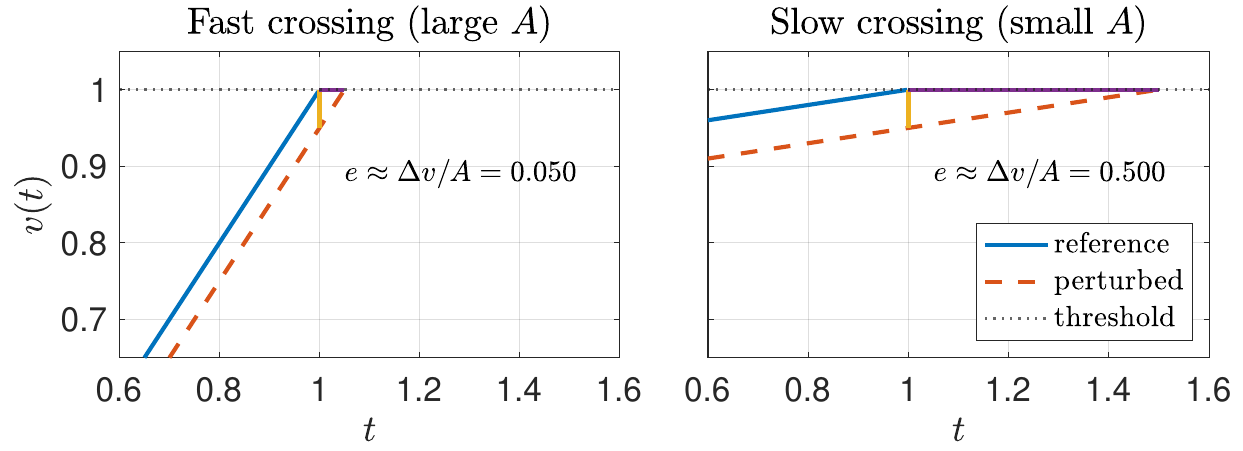}
\end{minipage}\hfill
\begin{minipage}[t]{0.38\textwidth}
\centering
\textbf{(C)}\\[0.35em]
\begin{tikzpicture}[x=1cm,y=1cm,every node/.style={font=\footnotesize},>=Latex]
  \node[draw,rounded corners=3pt,fill=blue!6,
        text width=4.8cm,align=left,minimum height=0.95cm] (good) at (0,1.05)
  {$\mathcal G^{\mathrm{match}}$: spike counts match up to $T$ \\ direct boundary-flux impact bound};

  \node[draw,rounded corners=3pt,fill=red!6,
        text width=4.8cm,align=left,minimum height=0.95cm] (bad) at (0,-0.75)
  {$B^{\mathrm{mis}}$: horizon spike-count mismatch \\ treated separately by pruning};

  \draw[thick,->] (0,0.45) -- (0,-0.18);


\end{tikzpicture}
\end{minipage}

\caption{A feedforward network architecture and strong error proof strategy. \textbf{(A)} Feedforward network architecture with noisy currents in every layer, feedforward weights $W^{\ell-1,\ell}$, and exponentially decaying synaptic kernel. \textbf{(B)} Illustration of fast and slow threshold crossings. Blue: true solution (reference); red: numerical solution (perturbed); yellow: pathwise error ($\Delta v$); purple: spike time error ($e_m$). \textbf{(C)} Section~\ref{sec:strong} separates the feedforward strong analysis into a matched-trajectory term, handled directly through boundary-flux averaging of the STE impact kernel, and a separate horizon-mismatch term.}
\label{fig:composite-overview}
\end{figure}

\subsection{Continuous-time dynamics with reset}
Fix $\tauv,\tauc>0$ and reset level $\Vr<\Vth$.
Between spikes of neuron $(\ell,j)$:
\begin{align}
dv_{\ell,j}(t) &=
\Bigl(-\tfrac{1}{\tauv}(v_{\ell,j}(t)-\Vr)+I_{\ell,j}(t)\Bigr)\,dt,
\qquad v_{\ell,j}(t)<\Vth,
\label{eq:dv}\\
dI_{\ell,j}(t) &=
-\tfrac{1}{\tauc}\bigl(I_{\ell,j}(t)-b_{\ell,j}(t)\bigr)\,dt
+\tfrac{\sigma_{\ell,j}}{\tauc}\,dB_{\ell,j}(t)
+\sum_{i=1}^{n_{\ell-1}} W^{\ell-1,\ell}_{ji}\,dN_{\ell-1,i}(t),
\qquad \ell\ge 2,
\label{eq:dI}
\end{align}
and for $\ell=1$ the sum over $i$ is replaced by external input only (equivalently take
$W^{0,1}\equiv 0$).
Here $\{B_{\ell,j}\}$ are independent standard Brownian motions, and $b_{\ell,j}(t)$ are bounded
deterministic inputs.

When $v_{\ell,j}$ hits $\Vth$ from below at time $s^k_{\ell,j}$, it resets instantly:
\begin{equation}\label{eq:reset}
v_{\ell,j}\bigl((s^k_{\ell,j})^-\bigr)=\Vth \quad\Rightarrow\quad
v_{\ell,j}\bigl((s^k_{\ell,j})^+\bigr)=\Vr,
\qquad
I_{\ell,j} \text{ is continuous at } s^k_{\ell,j}.
\end{equation}

\begin{remark}[Threshold current and crossing speed]\label{rem:Ith_A}
Define the threshold current
\[
\Ith := \frac{\Vth-\Vr}{\tauv}.
\]
At a spike time $s=s^k_{\ell,j}$, the instantaneous upward crossing speed equals
\begin{equation}\label{eq:A_def}
A^k_{\ell,j} := \dot v_{\ell,j}(s^-)= -\tfrac{1}{\tauv}(\Vth-\Vr)+I_{\ell,j}(s^-)=I_{\ell,j}(s^-)-\Ith.
\end{equation}
Necessarily $A^k_{\ell,j}>0$ at a genuine upward crossing, but it can be arbitrarily small in
fluctuation-driven regimes.
\end{remark}

\subsection{Euler--Maruyama scheme with grid-based threshold/reset}
Fix $h>0$ and grid $t_m:=mh$.
Let $\Delta B_{\ell,j,m}=B_{\ell,j}(t_{m+1})-B_{\ell,j}(t_m)\sim \mathcal{N}(0,h)$.

\paragraph{Current update.}
\begin{equation}\label{eq:EM_I}
I^h_{\ell,j}(t_{m+1})
=
I^h_{\ell,j}(t_m)
-\frac{h}{\tauc}\bigl(I^h_{\ell,j}(t_m)-b_{\ell,j}(t_m)\bigr)
+\frac{\sigma_{\ell,j}}{\tauc}\Delta B_{\ell,j,m}
+\sum_{i=1}^{n_{\ell-1}} W^{\ell-1,\ell}_{ji}\,\Delta N^h_{\ell-1,i}(m),
\end{equation}
where $\Delta N^h_{\ell-1,i}(m)\in\{0,1\}$ indicates whether the numerical presynaptic neuron
$(\ell-1,i)$ fired in $[t_m,t_{m+1})$.

\paragraph{Voltage update and spike detection.}
\begin{equation}\label{eq:EM_v}
v^h_{\ell,j}(t_{m+1})
=
v^h_{\ell,j}(t_m)
+h\Bigl(-\tfrac{1}{\tauv}(v^h_{\ell,j}(t_m)-\Vr)+I^h_{\ell,j}(t_m)\Bigr).
\end{equation}
If $v^h_{\ell,j}(t_{m+1})\ge \Vth$, we register a numerical spike at $\hat s^k_{\ell,j}:=t_{m+1}$
and reset $v^h_{\ell,j}(t_{m+1})\leftarrow \Vr$.
Let $N_{\ell,j}^h(t)$ be the corresponding numerical spike count.

\paragraph{Continuous-time interpolation.}
When we write $(v^h_{\ell,j}(t),I^h_{\ell,j}(t))$ for $t$ between grid points, we mean the standard
continuous-time EM interpolation on each step (coefficients frozen at $t_m$ and driven by the same
Brownian path); this is the version used in sup-norm bounds and in the spike-time sensitivity
estimate below.

\subsection{Spike-time error (STE)}\label{subsec:STE_def}
When exact and numerical spike sequences can be matched in order up to time $T$, define
\begin{equation}\label{eq:STE}
\varepsilon^k_{\ell,j}:=\hat s^k_{\ell,j}-s^k_{\ell,j},\qquad s^k_{\ell,j}\le T.
\end{equation}

\subsection{Main results}\label{subsec:main_results}
For readability, we summarize the main results here; full assumptions, constants,
and proofs appear in Sections~\ref{sec:strong}--\ref{sec:recurrent_weak}.

\begin{resultbox}
\begin{theorem*}[Matched-trajectory strong bound plus a separate mismatch-probability estimate]
Fix $T>0$ and define the network-level matched event
\[
\mathcal G^{\mathrm{match}}(h,T)
:=
\bigcap_{\ell=1}^L\bigcap_{j=1}^{n_\ell}
\{N^h_{\ell,j}(T)=N_{\ell,j}(T)\}.
\]
Let $X_L(T)$ be a bounded network output functional satisfying \eqref{eq:output_direct_lip} on
matched spike histories.  Under the spike-load assumptions of
Section~\ref{sec:strong},
\[
\E\Bigl[
\abs{X_L^h(T)-X_L(T)}^2\,\mathbf 1_{\mathcal G^{\mathrm{match}}(h,T)}
\Bigr]
\le
C_X\,n_L\,K_L(T,W)\,h\,(1+\abs{\log h})^{L}.
\]
Moreover, Lemma~\ref{lem:mis_prob} yields
\[
\Pp\bigl(B^{\mathrm{mis}}_{\ell,j}(h,T)\bigr)
\lesssim
r^{\max}_{\ell,j}(T)\sqrt{h|\log h|}
+
R^{(2)}_{\ell,j}(T)\,h|\log h|
+
C^{\mathrm{err}}_{\ell,j}(T)\,h^{c_0}.
\]
Thus the feedforward strong analysis consists of a matched-trajectory bound plus a separate
terminal mismatch term; small crossing speeds are absorbed into the matched term through the
boundary-flux integral rather than being declared a bad set.
\end{theorem*}
\end{resultbox}

\begin{resultbox}
\begin{theorem*}[Weak error (averaged order one)]
Let $T=Mh$ and $\Phi\in C_b^4(\mathsf{D})$ be spike-map compatible at the terminal time $T$.
Under the backward transmission assumption and the averaged one-step rate/strip/factorial-moment
package of Section~\ref{sec:weak},
\[
\abs{\E[\Phi(X(T))]-\E[\Phi(X^h(T))]}
\ \le\
T\,h\,C_{\mathrm{weak}}(T,L,W),
\]
with an explicit constant $C_{\mathrm{weak}}(T,L,W)$ given in \eqref{eq:Cweak_def}.
In particular, the weak proof is averaged along the numerical chain rather than a pointwise operator
estimate near threshold.  Under uniform rate and weight bounds, $C_{\mathrm{weak}}$ grows at most
linearly in the depth $L$.
\end{theorem*}
\end{resultbox}

\begin{resultbox}
\begin{theorem*}[Lyapunov exponents (saltation/flux and feedforward propagation)]
In a stationary regime with boundary flux
$J_\infty(I)=(I-\Ith)\rho_\infty(\Vth,I)\1_{\{I>\Ith\}}$,
the top Lyapunov exponent of a single neuron under common-noise coupling is
\[
\lambda
=
-\frac{1}{\tauv}
+
\int_{\R} J_\infty(I)\,\log\!\Bigl(\frac{I}{I-\Ith}\Bigr)\,dI,
\]
and the logarithmic singularity at $I\downarrow \Ith$ is integrable.
For feedforward networks, conditional on matched spike trains, layerwise exponents satisfy the
max-rule
$\lambda_\ell=\max\{\lambda_\ell^{\mathrm{diag}},\lambda_{\ell-1}\}$ for $\ell\ge2$.
\end{theorem*}
\end{resultbox}

\begin{resultbox}
\begin{theorem*}[Recurrent extensions (deterministic and noisy)]
For deterministic recurrent networks in a mean-driven (uniformly transversal) regime, the hybrid
fundamental matrix admits the explicit product form \eqref{eq:Phi_hyb_product} and the top hybrid
Lyapunov exponent satisfies the explicit upper bound in
Theorem~\ref{thm:lambda_hyb_upper}.
Moreover, for the time-driven Euler scheme with grid reset, deterministic global discretization
error and STE grow at most like $e^{\lambda_{\mathrm{hyb}}T}$ (Theorem~\ref{thm:det_error_lambda}),
so $\lambda_{\mathrm{hyb}}<0$ yields uniform-in-$T$ $O(h)$ accuracy while $\lambda_{\mathrm{hyb}}>0$
implies exponential-in-$T$ error amplification at rate at most $\lambda_{\mathrm{hyb}}$.

For noisy recurrent networks, under rate moment bounds ($R_{\mathrm{net}},R_{\mathrm{net},2}$) and
boundary density controls ($\rho_{\max},\rho^V_{\max}$), we obtain matched-trajectory strong bounds
in which small crossing speeds are handled directly through boundary-flux averaging of the
single-spike impact kernel; pruning is needed only for horizon spike-count mismatch.  Recurrence
enters through loop-truncated causal propagation governed by the shortest directed cycle length of
the synaptic graph together with spike-load and rate constants
(Theorem~\ref{thm:recurrent_goodset_explicit}).  Under the recurrent averaged one-step
rate/strip/factorial-moment package of Section~\ref{sec:recurrent_weak}, weak error remains order~$1$
with constants depending explicitly on $R_{\mathrm{net}},R_{\mathrm{net},2}$, $\|W\|_1$, and the
boundary-strip remainder constants (Theorem~\ref{thm:recurrent_weak_global}).
\end{theorem*}
\end{resultbox}

\section{Strong error via spike-time analysis}\label{sec:strong}

This section analyzes the strong error through spike times.  The central point is that in the simulation of a feedforward spiking network, if there is no spike from upstream layers, then neurons are essentially isolated. Therefore, one needs to carefully study how local Euler-Maruyama error gives rise to spike-time error (STE), and the \emph{impact} of STE on downstream synaptic currents.  

One important caveat is that the numbers of simulated spikes should match the real trajectory so one can pair them together, otherwise there is no STE to speak of. If we ignore the probability of mismatch of spiking numbers,  we can therefore proceed in two
steps. 
\begin{itemize}
    \item First, for a single matched spike we prove a conditional tail bound for the spike-time error by viewing it as a local hitting-time perturbation of the threshold event.
    \item Second, we integrate the corresponding spike-time \emph{impact} directly against the boundary flux distribution of crossing speeds.  This leads to a direct matched-trajectory recursion for the expected strong error, without any good/bad split inside the matched regime.
\end{itemize}

Spike mismatch in the given simulation horizon is treated separately in
Section~\ref{subsec:mismatch_count}.  These trajectories are not included in the final matched
strong error estimate; we will argue that, in practice, they are removed from the strong-error statistic.

\subsection{Warm-up: strong error for single neurons with noisy voltage}\label{subsec:warmup_vnoise}
We use one warm-up example to illustrate the importance of isolating spike mismatch. 
This warm-up isolates the mechanism already visible in the earlier single-neuron analysis:
away from threshold, Euler--Maruyama behaves as for a standard Ornstein--Uhlenbeck step,
while the reset can create an $\mathcal O(1)$ discrepancy only when the exact and numerical
paths disagree about the \emph{number of spikes} in a time step. 

Consider a single LIF neuron with \emph{voltage diffusion}
\begin{equation}\label{eq:warmup_vSDE}
dv(t)=b(v(t))\,dt+\sigma_v\,dB(t),
\qquad
b(v):=-\frac{1}{\tauv}(v-\Vr)+\bar I,
\qquad
v(t)<\Vth,
\end{equation}
with instantaneous reset
\[
v(s^-)=\Vth \quad\Longrightarrow\quad v(s^+)=\Vr.
\]
Let $N(t)$ be the exact spike count, and fix $T=Mh$.

The Euler--Maruyama approximation on the grid $t_m=mh$ is
\begin{equation}\label{eq:warmup_EM}
v^h(t_{m+1})
=
v^h(t_m)+h\,b\!\bigl(v^h(t_m)\bigr)+\sigma_v\Delta B_m,
\qquad
\Delta B_m:=B(t_{m+1})-B(t_m),
\end{equation}
followed by the grid reset rule
\[
v^h(t_{m+1})\ge \Vth
\quad\Longrightarrow\quad
v^h(t_{m+1})\leftarrow \Vr.
\]
Let $N^h(t)$ be the corresponding numerical spike count and define the grid-point error
\[
e_m:=v^h(t_m)-v(t_m),
\qquad
0\le m\le M.
\]

For each step write
\[
\Delta N_m:=N(t_{m+1})-N(t_m),
\qquad
\Delta N_m^h:=N^h(t_{m+1})-N^h(t_m)\in\{0,1\},
\]
and define the stepwise spike-count mismatch event
\begin{equation}\label{eq:warmup_Bmis}
B^{\mathrm{mis}}(h,T)
:=
\bigcup_{m=0}^{M-1}\{\Delta N_m\neq \Delta N_m^h\},
\qquad
G^{\mathrm{mis}}(h,T):=\bigl(B^{\mathrm{mis}}(h,T)\bigr)^c.
\end{equation}
Since the numerical scheme can generate at most one spike per step, on $G^{\mathrm{mis}}(h,T)$
every step contains either no spike for either trajectory or exactly one spike for both.

\begin{theorem}[Warm-up: matched-step strong bound and pruning]\label{thm:warmup_pruning}
Fix $T=Mh>0$ and assume
\begin{equation}\label{eq:warmup_moment_ass}
\sup_{0\le t\le T}\E|v(t)|^4<\infty,
\qquad
\max_{0\le m\le M}\E|v^h(t_m)|^4<\infty,
\qquad
\E[N(T)]<\infty.
\end{equation}
Then there exists a constant $C(T)$, independent of $h$, such that
\begin{equation}\label{eq:warmup_goodset_bound}
\max_{0\le m\le M}\E\bigl[|e_m|^2\,\1_{G^{\mathrm{mis}}(h,T)}\bigr]
\le C(T)\,h.
\end{equation}

If, in addition, there exists a deterministic constant $C_{\max}(T)$ such that
\[
|e_m|\le C_{\max}(T)
\qquad\text{for all }0\le m\le M\text{ on }B^{\mathrm{mis}}(h,T),
\]
then
\begin{equation}\label{eq:warmup_strong_decomp}
\max_{0\le m\le M}\E|e_m|^2
\le
C(T)\,h + C_{\max}(T)^2\,\Pp\!\bigl(B^{\mathrm{mis}}(h,T)\bigr).
\end{equation}
Consequently, any estimate $\Pp(B^{\mathrm{mis}}(h,T))=O(h^\beta)$ yields
\[
\max_{0\le m\le M}\E|e_m|^2 = O(h)+O(h^\beta).
\]
\end{theorem}

\begin{proof}
For $m=0,\dots,M$, define the good event up to time $t_m$ by
\[
G_m:=\bigcap_{r=0}^{m-1}\{\Delta N_r=\Delta N_r^h\},
\qquad
G_0:=\Omega.
\]
Then $G_M=G^{\mathrm{mis}}(h,T)$ and $G_{m+1}\subset G_m$.

Set
\[
a_m:=\E\bigl[|e_m|^2\,\1_{G_m}\bigr].
\]
We derive a recursion for $a_m$.

\medskip
\noindent\textbf{Step 1: no spike in the step.}
On $G_{m+1}\cap\{\Delta N_m=0\}$ we also have $\Delta N_m^h=0$, so no reset occurs for
either trajectory on $[t_m,t_{m+1}]$.  The exact solution satisfies
\[
v(t_{m+1})
=
v(t_m)+\int_{t_m}^{t_{m+1}} b(v(s))\,ds+\sigma_v\Delta B_m,
\]
whereas the numerical step is
\[
v^h(t_{m+1})
=
v^h(t_m)+h\,b\!\bigl(v^h(t_m)\bigr)+\sigma_v\Delta B_m.
\]
Subtracting and using the linearity of $b$ gives
\[
e_{m+1}
=
\Bigl(1-\frac{h}{\tauv}\Bigr)e_m-r_m,
\qquad
r_m:=\int_{t_m}^{t_{m+1}}\bigl(b(v(s))-b(v(t_m))\bigr)\,ds.
\]
Since $b$ is Lipschitz,
\[
|r_m|^2
\le
C h\int_{t_m}^{t_{m+1}} |v(s)-v(t_m)|^2\,ds.
\]
By the standard moment estimate for an additive-noise OU step and
\eqref{eq:warmup_moment_ass},
\[
\sup_{s\in[t_m,t_{m+1}]}\E\bigl[|v(s)-v(t_m)|^2\bigr]\le C(T)\,h,
\]
hence
\[
\E|r_m|^2\le C(T)\,h^3.
\]
Using $2ab\le h a^2+h^{-1}b^2$, we obtain
\[
|e_{m+1}|^2
\le
(1+C h)|e_m|^2 + C h^{-1}|r_m|^2,
\]
and therefore
\begin{equation}\label{eq:warmup_nospike_rec}
\E\bigl[|e_{m+1}|^2\,\1_{G_{m+1}\cap\{\Delta N_m=0\}}\bigr]
\le
(1+C h)\,a_m + C(T)\,h^2.
\end{equation}

\medskip
\noindent\textbf{Step 2: exactly one exact spike and one numerical spike in the step.}
On $G_{m+1}\cap\{\Delta N_m=1\}$ we have $\Delta N_m^h=1$ as well.  Let
$s_m\in(t_m,t_{m+1}]$ be the exact spike time in this step.  The numerical scheme resets at
the grid point, so
\[
v^h(t_{m+1})=\Vr.
\]
Since the exact solution resets at $s_m$ and then evolves again until $t_{m+1}$,
\[
v(t_{m+1})-\Vr
=
\int_{s_m}^{t_{m+1}} b(v(s))\,ds
+
\sigma_v\bigl(B(t_{m+1})-B(s_m)\bigr),
\]
and hence
\[
e_{m+1}
=
-\int_{s_m}^{t_{m+1}} b(v(s))\,ds
-\sigma_v\bigl(B(t_{m+1})-B(s_m)\bigr).
\]
Using $(x+y)^2\le 2x^2+2y^2$, then Cauchy--Schwarz for the drift term, and finally
\eqref{eq:warmup_moment_ass}, we obtain
\begin{align*}
\E\bigl[|e_{m+1}|^2\,\1_{G_{m+1}\cap\{\Delta N_m=1\}}\bigr]
&\le
2\,\E\!\left[\left|\int_{s_m}^{t_{m+1}} b(v(s))\,ds\right|^2 \1_{\{\Delta N_m=1\}}\right]
+
2\sigma_v^2\,\E\!\left[(t_{m+1}-s_m)\,\1_{\{\Delta N_m=1\}}\right] \\
&\le
C h \int_{t_m}^{t_{m+1}} \E\!\left[(1+|v(s)|^2)\,\1_{\{\Delta N_m=1\}}\right] ds
+
2\sigma_v^2 h\,\Pp(\Delta N_m=1) \\
&\le
C(T)\,h^2 + C(T)\,h\,\Pp(\Delta N_m=1).
\end{align*}

\medskip
\noindent\textbf{Step 3: recursion on the good set.}
Combining the two cases and using $G_{m+1}\subset G_m$ yields
\[
a_{m+1}
\le
(1+C h)a_m + C(T)\,h^2 + C(T)\,h\,\Pp(\Delta N_m=1).
\]
Since $a_0=0$, discrete Gronwall gives
\[
a_m
\le
C(T)\left(
h + h\sum_{r=0}^{m-1}\Pp(\Delta N_r=1)
\right),
\qquad 0\le m\le M.
\]
Finally,
\[
\sum_{r=0}^{M-1}\Pp(\Delta N_r=1)
=
\E\!\left[\sum_{r=0}^{M-1}\1_{\{\Delta N_r=1\}}\right]
\le
\E[N(T)]<\infty,
\]
so \eqref{eq:warmup_goodset_bound} follows.

If $|e_m|\le C_{\max}(T)$ on $B^{\mathrm{mis}}(h,T)$, then
\[
\E|e_m|^2
=
\E\bigl[|e_m|^2\,\1_{G^{\mathrm{mis}}(h,T)}\bigr]
+
\E\bigl[|e_m|^2\,\1_{B^{\mathrm{mis}}(h,T)}\bigr]
\le
C(T)\,h + C_{\max}(T)^2\,\Pp\!\bigl(B^{\mathrm{mis}}(h,T)\bigr),
\]
uniformly in $m$, which proves \eqref{eq:warmup_strong_decomp}.
\end{proof}

\begin{remark}[How this warm-up relates to the current-based analysis]
On matched steps, ordinary OU consistency remains intact, and the only extra cost caused by
the reset is the short post-spike excursion of the exact solution over an interval of length at
most $h$.  Because the expected number of spikes on $[0,T]$ is finite, these matched spike
steps still contribute only $O(h)$ to the finite-horizon mean-square error.

On the other hand, for voltage diffusion, the mismatch itself is the
classical grid-monitoring exit-time error: the exact path can hit $\Vth$ and return below it
before the next grid point.  Standard boundary-layer estimates for Euler schemes then give
\[
\Pp\!\bigl(B^{\mathrm{mis}}(h,T)\bigr)\lesssim T\,h^{1/2},
\]
see, for example, \cite{BouchardMenozzi2009,BouchardGeissGobet2017}.  Under a deterministic
post-mismatch bound, this recovers mean-square order at least $h^{1/2}$, i.e. root-mean-square
order at least $h^{1/4}$.
\end{remark}

\medskip
From now on, we return to the analysis of current-noise LIF neurons in
Eq.~\eqref{eq:dv} and \eqref{eq:dI}, where the voltage has no Brownian diffusion.  The
Brownian recrossing mechanism therefore disappears, and the dominant strong obstruction becomes
a nearly tangential crossing, quantified by the small speed $A=I-\Ith$ and corresponding large STE.  This is also why the
main theory does \emph{not} use the stepwise mismatch event \eqref{eq:warmup_Bmis}: pruning
after the first transient stepwise disagreement is too pessimistic.  Starting from Section~\ref{subsec:single_strong} we simply assume
$N^h_{\ell,j}(T) = N_{\ell,j}(T)$, and defer the analysis of horizon mismatch to Section~\ref{subsec:mismatch_count}.


\subsection{Single-spike hitting-time comparison and conditional STE tail}\label{subsec:single_strong}

We now replace the deterministic catch-up map by a genuine first-passage problem.
To keep the notation readable, throughout this subsection and the next two subsections we suppress
the neuron and spike indices unless they are needed explicitly.

Fix one matched spike pair.  Let $s$ be the exact spike time, let $\tilde s$ be the threshold time
of the continuous-time EM interpolation before the grid reset is applied, and let $\hat s$ be the
grid-registered numerical spike time.  Thus
\begin{equation}\label{eq:grid_lag}
0\le \hat s-\tilde s\le h.
\end{equation}
The spike-time error is
\begin{equation}\label{eq:STE_local}
\varepsilon:=\hat s-s.
\end{equation}

At the exact spike time the threshold-crossing speed is
\begin{equation}\label{eq:A_def_strong}
A
:=
\dot v(s^-)
=
-\frac{1}{\tauv}(\Vth-\Vr)+I(s^-)
=
I(s^-)-\Ith.
\end{equation}

\paragraph{Local pre-reset comparison.}
Write $r:=s\wedge \tilde s$.  Let $v^{\sharp}$ and $v^{h,\sharp}$ be the exact and numerical
\emph{pre-reset continuations} near $r$: both evolve by the subthreshold dynamics and neither reset
is applied at $s$ or $\tilde s$.  At the earlier threshold time $r$, exactly one of the two
pre-reset continuations is at threshold; the other one is the lagging continuation.  Define the
local voltage gap
\begin{equation}\label{eq:Dk_def}
D:=\Vth-v^{\mathrm{lag}}(r)\in[0,\infty).
\end{equation}

\begin{assumption}[Local isolated catch-up window]\label{ass:excursion_window}
Take $\delta_\ast(h,T)\in(0,1]$ as the largest catch-up time for the matched spike pair.  For the local comparison around each matched spike we assume:
\begin{itemize}
    \item[(i)] on $[r,r+\delta_\ast(h,T)]$ no presynaptic spike arrives at the neuron;
    \item[(ii)] on that window the deterministic input is frozen at its crossing value, so the
    current excess $J:=I-\Ith$ is an OU process started at $A$ and mean-reverting to the same value
    $A$;
    \item[(iii)] the numerical spike used in the strong-error statistic is the grid-registered one,
    hence \eqref{eq:grid_lag} always holds.
\end{itemize}
For layer $1$ there are no presynaptic jumps, and only the local freezing of the deterministic input
is used.  In the feedforward propagation step below we apply the same isolated-window model
conditionally to deeper layers.
\end{assumption}

The first local input needed by the hitting-time analysis is the law of the local EM voltage gap.

\begin{proposition}[First-layer local EM gap estimate]\label{ass:local_EM_tail}
Assume that each deterministic input $b_{\ell,j}$ is bounded and Lipschitz in time.  Then for the
first layer there exist constants
\[
C_1^{\mathrm{loc}}(T)>0,\qquad
C_1^{\mathrm{sg}}(T)>0,\qquad
c_1^{\mathrm{sg}}(T)>0
\]
such that, for every neuron $(1,j)$, every matched spike index $k$ with $s^k_{1,j}\le T$, and every
$r\in(0,1]$,
\begin{equation}\label{eq:D_tail_direct}
\Pp\!\Bigl(
D^k_{1,j}>r
\,\Big|\,
\mathcal F^{k-}_{1,j},A^k_{1,j}
\Bigr)
\le
C_1^{\mathrm{sg}}(T)
\exp\!\Bigl(
-c_1^{\mathrm{sg}}(T)\,\frac{r^2}{h}
\Bigr),
\end{equation}
and
\begin{equation}\label{eq:D_second_moment_direct}
\E\!\Bigl[
(D^k_{1,j})^2
\,\Big|\,
\mathcal F^{k-}_{1,j},A^k_{1,j}
\Bigr]
\le
C_1^{\mathrm{loc}}(T)\,h.
\end{equation}
In fact the proof gives the sharper tail
\[
\Pp\!\Bigl(
D^k_{1,j}>r
\,\Big|\,
\mathcal F^{k-}_{1,j},A^k_{1,j}
\Bigr)
\le
\widetilde C_1(T)
\exp\!\Bigl(
-\widetilde c_1(T)\,\frac{r^2}{h^2}
\Bigr).
\]
\end{proposition}

\begin{proof}
Fix one first-layer neuron and one local isolated comparison window.  Since there is no presynaptic
jump term in layer $1$, the exact current solves
\[
dI(t)= -\frac1{\tauc}\bigl(I(t)-b(t)\bigr)\,dt+\frac{\sigma}{\tauc}\,dB(t),
\]
whereas the continuous-time EM interpolation on each grid step uses the same Brownian path with the
drift frozen at the left endpoint of that step.  Because the coefficients are linear and the noise is
additive, the exact-minus-EM current error on any fixed finite window can be written as
\[
I(t)-I^h(t)=\beta_I(t)+\Gamma_I(t),
\]
where $\Gamma_I$ is a centered Gaussian process measurable with respect to the Brownian increments on
that window and $\beta_I$ is a deterministic bias coming only from freezing the Lipschitz input
$b(\cdot)$.  Standard one-step variation-of-constants formulas therefore yield
\begin{equation}\label{eq:eI_Gauss_decomp}
\sup_{0\le u\le \delta_\ast(h,T)}
\abs{\beta_I(r+u)}\le C(T)\,h,
\qquad
\sup_{0\le u\le \delta_\ast(h,T)}
\operatorname{Var}\!\Bigl(
\Gamma_I(r+u)\,\Big|\,\mathcal F^{k-}_{1,j},A^k_{1,j}
\Bigr)
\le C(T)\,h^2.
\end{equation}
The voltage error is the stable linear convolution of the current error:
\[
v^{\sharp}(r+u)-v^{h,\sharp}(r+u)
=
e^{-u/\tauv}\bigl(v^{\sharp}(r)-v^{h,\sharp}(r)\bigr)
+
\int_0^u e^{-(u-q)/\tauv}\bigl(I(r+q)-I^h(r+q)\bigr)\,dq.
\]
In the first layer the entrance discrepancy on the isolated comparison window is exactly the local
continuous-time EM defect, so the same decomposition gives
\begin{equation}\label{eq:ev_Gauss_bounds}
\sup_{0\le u\le \delta_\ast(h,T)}
\abs{
v^{\sharp}(r+u)-v^{h,\sharp}(r+u)
}
\le
C(T)\,h
+
\sup_{0\le u\le \delta_\ast(h,T)}\abs{\Gamma_v(u)},
\end{equation}
for another centered Gaussian process $\Gamma_v$ satisfying
\[
\sup_{0\le u\le \delta_\ast(h,T)}
\operatorname{Var}\!\Bigl(
\Gamma_v(u)\,\Big|\,\mathcal F^{k-}_{1,j},A^k_{1,j}
\Bigr)
\le C(T)\,h^2.
\]
At the earlier threshold time $r$ the lagging continuation is below threshold by exactly the local
voltage error, so
\[
D^k_{1,j}\le
C(T)\,h+\sup_{0\le u\le \delta_\ast(h,T)}\abs{\Gamma_v(u)}.
\]
Borell--TIS therefore yields the sharper sub-Gaussian tail at scale $h$:
\[
\Pp\!\Bigl(
D^k_{1,j}>C(T)h+x
\,\Big|\,
\mathcal F^{k-}_{1,j},A^k_{1,j}
\Bigr)
\le
\widetilde C_1(T)\exp\!\Bigl(-\widetilde c_1(T)\,\frac{x^2}{h^2}\Bigr).
\]
Since $0<h\le1$, this implies the weaker scale-$\sqrt h$ form \eqref{eq:D_tail_direct} after
enlarging the constants.  Integrating \eqref{eq:D_tail_direct} gives
\eqref{eq:D_second_moment_direct}.
\end{proof}

\begin{remark}[Transfer to deeper feedforward layers]\label{rem:first_layer_transfer}
The proof above is purely local: on the isolated comparison window the current equation is linear,
the Brownian noise is shared, and the local voltage gap is a stable linear functional of the local
current defect.  In deeper feedforward layers we use exactly the same local comparison, but the
variance proxy is no longer $C_1^{\mathrm{loc}}(T)h$; it is replaced by the layerwise deterministic
scale $R_\ell(h,T)$ defined in \eqref{eq:Rell_direct}, which contains both the local current-EM
defect and the inherited synaptic timing error from layer $\ell-1$.
\end{remark}

We now turn to the local catch-up problem itself.

\begin{proposition}[Frozen-OU random catch-up law]\label{lem:STE_sensitivity}\label{ass:excursion_law}
Assume Assumption~\ref{ass:excursion_window}.  Condition on one local comparison window, on a fixed
crossing speed $A=a>0$, and on a fixed local voltage gap $D=d\ge0$.  Then the interpolation-level
spike lag $\abs{\tilde s-s}$ equals the first-passage time
\[
\tau_{d,a}:=\inf\{u\ge0:\ R_u=0\}
\]
of the one-dimensional threshold-gap process
\begin{equation}\label{eq:Theta_def_const}
dR_u
=
\Bigl(J_u-\frac1{\tauv}R_u\Bigr)\,du,
\qquad
R_0=-d,
\qquad
dJ_u
=
-\frac1{\tauc}(J_u-a)\,du+\frac{\sigma}{\tauc}\,dW_u,
\qquad
J_0=a.
\end{equation}
Consequently
\begin{equation}\label{eq:STE_sensitivity}
\abs{\varepsilon}\le h+\tau_{D,A}.
\end{equation}

Moreover, when $d$ is small
\begin{equation}\label{eq:catchup_moments}
m_1 = \E_{d,a}[\tau] = \frac{d}{a} + O(d^2) \qquad m_2 = \E_{d,a}[\tau^2] = \frac{d^2}{a^2} + O(d^3) \\
\end{equation}
\end{proposition}

\begin{proof}
The exact survival function of the random catch-up time is a backward first-passage solution.
Indeed, with state variables $(r,j)$ and generator
\[
\mathcal L
=
\Bigl(j-\frac r{\tauv}\Bigr)\partial_r
-
\frac{j-a}{\tauc}\,\partial_j
+
\frac{\sigma^2}{2\tauc^2}\partial_{jj},
\]
the function
\[
u(r,j,t):=\Pp_{(r,j)}(\tau_0>t),
\qquad \tau_0:=\inf\{u\ge0:R_u=0\},
\]
solves
\[
\partial_t u=\mathcal L u,\qquad u(0,j,t)=0,\qquad u(r,j,0)=1,\qquad r<0.
\]
Standard Dynkin formulas give
\begin{align*}
    \mathcal L m_1 = -1, &\qquad m_1(0,j) = 0 \\
    \mathcal L m_2 = -2m_1, &\qquad m_2(0,j) = 0
\end{align*}
Since the boundary data are identically zero in $j$, we get 
\[
\partial_jm_q(0,j) = \partial_{jj}m_q(0,j) = 0, \qquad q = 1,2.
\]
When evaluating $\mathcal L m_1 = -1$ at $r = 0$, since the $j$-derivatives vanish here, we get $j\partial_r m_1(0,j) = -1$, leading to $m_1(-d,a) = d/a + O(d^2)$. Similarly $\partial_r m_2(0,j) = 0$, leading to 
\[
j\partial_{rr}m_2(0,j) = -2\partial_{r}m_1(0,j) =  \frac2j.
\]
Hence $m_2(-d,a) = d^2/a^2 + O(d^3)$.
\end{proof}

We can now consider the impact of spike time error to the postsynaptic neuron.

\paragraph{Conditional impact kernel at fixed crossing speed.}
The exact $L^1_t$ misalignment of one postsynaptic exponential synaptic kernel shifted by a spike-time
error $\varepsilon$ is
\[
2\tauc\bigl(1-e^{-|\varepsilon|/\tauc}\bigr),
\]
so we again define
\begin{equation}\label{eq:psi_def}
\psi(x):=\bigl(1-e^{-x/\tauc}\bigr)^2,\qquad x\ge0.
\end{equation}
Then
\begin{equation}\label{eq:psi_basic}
0\le \psi(x)\le 1,
\qquad
\psi(x)\le \frac{x^2}{\tauc^2}.
\end{equation}

\begin{lemma}[Conditional impact bound at fixed crossing speed]\label{thm:kernel_bound_strong}
For every neuron $(\ell,j)$, every matched spike index $k$, and every realization of the
crossing speed $A^k_{\ell,j}=a$, there exists a constant
$C_\ell^{\mathrm{hit}}(T)<\infty$ such that
\begin{equation}\label{eq:single_impact_bound}
\E\!\Bigl[
\psi\!\bigl(\abs{\varepsilon^k_{\ell,j}}\bigr)
\,\Big|\,
\mathcal F^{k-}_{\ell,j},A^k_{\ell,j}=a
\Bigr]
\le
C_\ell^{\mathrm{hit}}(T)\,
\min\!\Bigl\{1,\frac{R_\ell(h,T)^2}{a^2}\Bigr\},
\qquad a>0.
\end{equation}
where $R_\ell(h,T)$ is the local EM gap scale defined in
\eqref{eq:R1_direct}--\eqref{eq:Rell_direct}. The bound on the RHS of Eq.~\ref{eq:single_impact_bound} is defined as $\Gamma_\ell(a;h,T)$.

\end{lemma}

\begin{proof}
From \eqref{eq:STE_sensitivity} and \eqref{eq:psi_basic},
\[
\psi\!\bigl(\abs{\varepsilon}\bigr)
\le
\min\!\Bigl\{1,\frac{(h+\tau_{D,A})^2}{\tauc^2}\Bigr\}
\lesssim
\min\!\Bigl\{1,h^2+\tau_{D,A}^2\Bigr\}.
\]
Condition on $A^k_{\ell,j}=a$.  The local-gap estimate at layer $\ell$ gives
$\E[D^2\mid \mathcal F_{\ell,j}^{k-},A^k_{\ell,j}=a]\lesssim R_\ell(h,T)^2$, while
\eqref{eq:catchup_moments} implies $\E[\tau_{D,a}^2\mid D,A=a]\lesssim D^2/a^2$ for small $D$.
Taking conditional expectations and using $\psi\le1$ yields
\[
\E\!\Bigl[
\psi\!\bigl(\abs{\varepsilon^k_{\ell,j}}\bigr)\,\Big|\,\mathcal F^{k-}_{\ell,j},A^k_{\ell,j}=a
\Bigr]
\lesssim
\min\!\Bigl\{1,\frac{R_\ell(h,T)^2}{a^2}\Bigr\},
\]
which is exactly \eqref{eq:single_impact_bound}.
\end{proof}

\subsection{Boundary flux of crossing speeds and direct expected spike impact}\label{subsec:flux}

We next average the deterministic conditional kernel $\Gamma_\ell(a;h,T)$ against the distribution
of crossing speeds at threshold.

\paragraph{Boundary flux of crossing speeds.}
For every neuron $(\ell,j)$, let $\rho_{\ell,j}(v,I,t)$ denote the subthreshold density of
$(v_{\ell,j}(t),I_{\ell,j}(t))$.  The outgoing threshold flux is
\begin{equation}\label{eq:flux}
J_{\ell,j}(I,t)
=
(I-\Ith)\,\rho_{\ell,j}(\Vth,I,t)\,\mathbf 1_{\{I>\Ith\}}.
\end{equation}
Equivalently, with crossing speed $a=I-\Ith>0$,
\begin{equation}\label{eq:flux_a}
J_{\ell,j}(\Ith+a,t)=a\,\rho_{\ell,j}(\Vth,\Ith+a,t),\qquad a>0.
\end{equation}

\begin{lemma}[Flux identity for observables of the crossing speed]\label{lem:flux_identity}
Let $\varphi:[0,\infty)\to[0,\infty)$ be measurable. Then
\begin{equation}\label{eq:flux_identity}
\E\Bigg[
\sum_{k:\,s^k_{\ell,j}\le T}\varphi(A^k_{\ell,j})
\Bigg]
=
\int_0^T\int_0^\infty
\varphi(a)\,a\,\rho_{\ell,j}(\Vth,\Ith+a,t)\,da\,dt.
\end{equation}
\end{lemma}

\begin{proof}
This is the standard flux representation of threshold crossings for the absorbed kinetic
Fokker--Planck equation.
\end{proof}

\begin{assumption}[Boundary depletion and spike-count bounds]\label{ass:boundary_depletion}
For each layer $\ell$ there exist constants
\[
\gamma_\ell\ge0,\qquad
C_\ell^\rho(T)<\infty,\qquad
\bar N_\ell(T)<\infty,\qquad
\mathcal N_\ell(T)<\infty
\]
such that, for every neuron $(\ell,j)$,
\begin{align}
\rho_{\ell,j}(\Vth,\Ith+a,t)
&\le C_\ell^\rho(T)\,a^{\gamma_\ell},
\qquad 0<a\le \alpha_0,\ \ 0\le t\le T,
\label{eq:boundary_depletion}\\[0.3em]
\E[N_{\ell,j}(T)]&\le \bar N_\ell(T),
\label{eq:expected_count_bound}\\[0.3em]
N_{\ell,j}(T)&\le \mathcal N_\ell(T)
\qquad\text{on the matched trajectories retained in the strong-error statistic.}
\label{eq:pathwise_count_bound}
\end{align}
\end{assumption}

\begin{remark}[Interpretation of $\gamma_\ell$]
    Any nonnegative $\gamma_\ell$ means that the the probability density for small crossing speed $a$ scales with $a$. When $\rho_{\ell,j}(\Vth,\Ith+a,t)$ doesn't vanish at $a=0$, we have $\gamma_\ell = 0$. On the other hand, in a mean-driven regime we have crossing speed $a\geq \alpha_{\min}>0$, then $\gamma_\ell = \infty$.
\end{remark}

To facilitate the discussion below, we define a STE impact constant:
\begin{equation}\label{eq:Xi_gamma}
\Xi_\gamma(r)
:=
\begin{cases}
1+\abs{\log r}, & \gamma=0,\\[0.3em]
1, & \gamma>0,
\end{cases}
\qquad 0<r\le 1.
\end{equation}

Throughout the matched-trajectory part of Section~\ref{sec:strong}, we write
\begin{equation}\label{eq:Gmatch_feedforward_def}
\mathcal G^{\mathrm{match}}(h,T)
:=
\bigcap_{\ell=1}^L\bigcap_{j=1}^{n_\ell}
\{N^h_{\ell,j}(T)=N_{\ell,j}(T)\}.
\end{equation}
All spike-time quantities in Sections~\ref{subsec:flux}--\ref{subsec:strong_explicit_TL} are
understood on this event; equivalently, the indicator
$\mathbf 1_{\mathcal G^{\mathrm{match}}(h,T)}$ is included in the corresponding expectations.

\begin{theorem}[Bound on STE impact]\label{thm:direct_flux_moment}
Assume Assumption~\ref{ass:boundary_depletion} and Lemma~\ref{thm:kernel_bound_strong}.  Then, for every
neuron $(\ell,j)$,
\begin{equation}\label{eq:direct_flux_moment}
\mathcal M_{\ell,j}(T)
:=
\E\Bigg[
\sum_{k:\,s^k_{\ell,j}\le T}\psi\!\bigl(\abs{\varepsilon^k_{\ell,j}}\bigr)\,
\mathbf 1_{\mathcal G^{\mathrm{match}}(h,T)}
\Bigg]
\le
\int_0^T\int_0^\infty
\Gamma_\ell(a;h,T)\,
a\,\rho_{\ell,j}(\Vth,\Ith+a,t)\,da\,dt.
\end{equation}
Defining
\begin{equation}\label{eq:Mlayer_def}
\mathcal M_\ell(T):=\max_{1\le j\le n_\ell}\mathcal M_{\ell,j}(T),
\end{equation}
we obtain
\begin{equation}\label{eq:direct_flux_moment_layer}
\mathcal M_\ell(T)
\le
\max_{1\le j\le n_\ell}
\int_0^T\int_0^\infty
\Gamma_\ell(a;h,T)\,
a\,\rho_{\ell,j}(\Vth,\Ith+a,t)\,da\,dt.
\end{equation}
Moreover, there exists a bounded $C_\ell^{\mathrm{dir}}(T)$ s.t., 
\begin{equation}\label{eq:log_inv_moment}
\mathcal M_\ell(T)
\le
C_\ell^{\mathrm{dir}}(T)\,
R_\ell(h,T)^2\,
\Xi_{\gamma_\ell}\!\bigl(R_\ell(h,T)\bigr).
\end{equation}

\end{theorem}

\begin{proof}
Apply Lemma~\ref{thm:kernel_bound_strong} inside the spike sum:
\[
\mathcal M_{\ell,j}(T)
=
\E\Bigg[
\sum_{k:\,s^k_{\ell,j}\le T}
\E\Bigl[\psi\!\bigl(\abs{\varepsilon^k_{\ell,j}}\bigr)\,\Big|\,\mathcal F^{k-}_{\ell,j},A^k_{\ell,j}\Bigr]
\Bigg]
\le
\E\Bigg[
\sum_{k:\,s^k_{\ell,j}\le T}
\Gamma_\ell\!\bigl(A^k_{\ell,j};h,T\bigr)
\Bigg].
\]
Applying Lemma~\ref{lem:flux_identity} with $\varphi(a)=\Gamma_\ell(a;h,T)$ yields \eqref{eq:direct_flux_moment_layer}.

To obtain \eqref{eq:log_inv_moment} we split the
flux integral at $a=R_\ell(h,T)$ and $a=\alpha_0$.  The
small-speed contribution on $(0,R_\ell]$ is $O(R_\ell(h,T)^2)$, the intermediate range
$[R_\ell,\alpha_0]$ gives the logarithm when $\gamma_\ell=0$ and a bounded contribution when
$\gamma_\ell>0$, and the large-speed tail $(\alpha_0,\infty)$ is $O(R_\ell(h,T)^2)$ by the
expected count bound \eqref{eq:expected_count_bound}.
\end{proof}

\subsection{Propagation across feedforward layers}\label{subsec:propagate}

The next two lemmas are the deterministic propagation mechanism.  They are unchanged in spirit from
the previous version; the propagated quantity is still the accumulated spike-impact
$\psi(\abs{\varepsilon})$, but the local one-spike input is now the deterministic kernel
$\Gamma_\ell(a;h,T)$ rather than the earlier hard-wired inverse-speed truncation.

\subsubsection{Synaptic kernel and the misalignment window}

\begin{lemma}[Misalignment window: $L^\infty$ is not Lipschitz, but $L^1$ is]\label{lem:kernel_shift}
Let
\[
K(t;s):=e^{-(t-s)/\tauc}\mathbf 1_{\{t\ge s\}},
\qquad
\hat s=s+\varepsilon.
\]
Then
\begin{align}
\sup_{t\in[0,T]}\abs{K(t;s)-K(t;\hat s)}
&=
\mathbf 1_{\{\varepsilon\neq0\}},
\label{eq:kernel_Linfty}\\[1mm]
\int_0^T\abs{K(t;s)-K(t;\hat s)}\,dt
&=
2\tauc\bigl(1-e^{-\abs{\varepsilon}/\tauc}\bigr)
=
2\tauc\sqrt{\psi(\abs{\varepsilon})}.
\label{eq:kernel_L1}
\end{align}
\end{lemma}

\begin{proof}
Direct computation.
\end{proof}

\begin{lemma}[Postsynaptic current sensitivity in $L^1_t$]\label{lem:I_from_STE_L1}
Fix $\ell\ge2$ and neuron $(\ell,j)$. On matched spike histories,
\begin{align}
\int_0^T \abs{I_{\ell,j}(t)-I^h_{\ell,j}(t)}\,dt
&\le
T\,C^{I,\mathrm{loc}}_{\ell,j}(T)\,h^{1/2}
\notag\\
&\quad
+
2\tauc\sum_{i=1}^{n_{\ell-1}}\abs{W^{\ell-1,\ell}_{ji}}
\sum_{k:\,s^k_{\ell-1,i}\le T}
\sqrt{\psi\!\bigl(\abs{\varepsilon^k_{\ell-1,i}}\bigr)}.
\label{eq:Ierr_from_STE_L1}
\end{align}
Consequently,
\begin{align}
\E\Bigg[
\Bigl(\int_0^T \abs{I_{\ell,j}(t)-I^h_{\ell,j}(t)}\,dt\Bigr)^2
\Bigg]
&\le
2T^2\bigl(C^{I,\mathrm{loc}}_{\ell,j}(T)\bigr)^2 h
\notag\\
&\quad
+
8\tauc^2\,n_{\ell-1}\,\mathcal N_{\ell-1}(T)\,
\|W^{\ell-1,\ell}\|_\infty^2\,
\mathcal M_{\ell-1}(T).
\label{eq:Ierr_direct_L2}
\end{align}
\end{lemma}

\begin{proof}
Subtract exact and numerical kernel sums and apply \eqref{eq:kernel_L1} to every matched
presynaptic spike to obtain \eqref{eq:Ierr_from_STE_L1}.  Then square, use $(x+y)^2\le2x^2+2y^2$,
and apply \eqref{eq:pathwise_count_bound}.
\end{proof}

\begin{lemma}[Voltage $L^1_t$ stability]\label{lem:v_from_I}
On an interval without resets, if $v$ and $\tilde v$ solve
\[
\dot v = -\tfrac1{\tauv}(v-\Vr)+I(t),\qquad
\dot{\tilde v} = -\tfrac1{\tauv}(\tilde v-\Vr)+\tilde I(t)
\]
with the same initial condition, then
\begin{equation}\label{eq:v_from_I}
\sup_{t\in[t_0,t_1]}\abs{v(t)-\tilde v(t)}
\le
\int_{t_0}^{t_1}\abs{I(s)-\tilde I(s)}\,ds.
\end{equation}
\end{lemma}

\begin{proof}
Variation of constants.
\end{proof}

\paragraph{Layerwise local-gap scale.}
Set
\begin{equation}\label{eq:R1_direct}
R_1(h,T)^2:=C^{\mathrm{loc}}_1(T)\,h.
\end{equation}
For $\ell\ge2$, define
\begin{equation}\label{eq:Rell_direct}
R_\ell(h,T)^2
:=
C^{\mathrm{loc}}_\ell(T)\,h
+
8\tauc^2\,n_{\ell-1}\,\mathcal N_{\ell-1}(T)\,
\|W^{\ell-1,\ell}\|_\infty^2\,
\mathcal M_{\ell-1}(T).
\end{equation}
This is the deterministic variance proxy entering the conditional local-gap law at layer $\ell$.

\begin{theorem}[Exact matched-trajectory recursion in kernel form]\label{thm:direct_strong_recursion}
Assume Assumption~\ref{ass:excursion_window}, the conditional layerwise version of the local-gap tail
bound \eqref{eq:D_tail_direct} with deterministic scale $R_\ell(h,T)$, and
Assumption~\ref{ass:boundary_depletion}.  Then, for every layer $\ell=1,\dots,L$,
\begin{equation}\label{eq:STE_recursion}
\mathcal M_\ell(T)
\le
\max_{1\le j\le n_\ell}
\int_0^T\int_0^\infty
\Gamma_\ell(a;h,T)\,
a\,\rho_{\ell,j}(\Vth,\Ith+a,t)\,da\,dt.
\end{equation}
\end{theorem}

\begin{proof}
For $\ell=1$, this is exactly Theorem~\ref{thm:direct_flux_moment} with
$R_1(h,T)^2=C_1^{\mathrm{loc}}(T)h$.

For $\ell\ge2$, Lemma~\ref{lem:I_from_STE_L1} and Lemma~\ref{lem:v_from_I} show that on every
isolated comparison window the local voltage gap entering is
controlled by the deterministic scale $R_\ell(h,T)$.  Therefore the same one-spike kernel
$\Gamma_\ell(a;h,T)$ is valid at layer $\ell$.  Theorem~\ref{thm:direct_flux_moment} then gives
\eqref{eq:STE_recursion}.
\end{proof}

\begin{corollary}[Explicit depth dependence under the matched-impact kernel bound]\label{cor:direct_depth}
Assume the hypotheses of Theorem~\ref{thm:direct_strong_recursion}.  Let
\[
q_\ell:=\sum_{m=1}^{\ell}\mathbf 1_{\{\gamma_m=0\}},
\qquad
B_\ell(T,W):=
8\tauc^2\,n_{\ell-1}\,\mathcal N_{\ell-1}(T)\,
\|W^{\ell-1,\ell}\|_\infty^2
\quad (\ell\ge2).
\]
Then there exist constants $K_\ell(T,W)<\infty$ such that for all sufficiently small $h$,
\begin{equation}\label{eq:DeltaL_explicit}
\mathcal M_\ell(T)
\le
K_\ell(T,W)\,h\,\bigl(1+\abs{\log h}\bigr)^{q_\ell},
\qquad \ell=1,\dots,L.
\end{equation}
In particular, if $\gamma_1=\cdots=\gamma_\ell=0$, then
\begin{equation}\label{eq:DeltaL_explicit_alllog}
\mathcal M_\ell(T)
\le
K_\ell(T,W)\,h\,\bigl(1+\abs{\log h}\bigr)^{\ell}.
\end{equation}
The constants can be chosen recursively by
\begin{align}
K_1(T,W)&:=2\,C_1^{\mathrm{dir}}(T)\,C_1^{\mathrm{loc}}(T),
\label{eq:K1_rec}\\[0.3em]
K_\ell(T,W)
&:=
2\,C_\ell^{\mathrm{dir}}(T)
\Bigl(
C_\ell^{\mathrm{loc}}(T)+B_\ell(T,W)\,K_{\ell-1}(T,W)
\Bigr),
\qquad \ell\ge2,
\label{eq:Kell_rec}
\end{align}
where $C_\ell^{\mathrm{dir}}(T)$ is the constant from Theorem~\ref{thm:direct_flux_moment}.
\end{corollary}

\begin{proof}
The claim follows by induction on $\ell$.  For $\ell=1$, combine
\eqref{eq:R1_direct} with \eqref{eq:log_inv_moment}.  Assuming the bound at level $\ell-1$,
insert \eqref{eq:DeltaL_explicit} into \eqref{eq:Rell_direct}, then apply
\eqref{eq:log_inv_moment} once more.  The recursive constants \eqref{eq:K1_rec}--\eqref{eq:Kell_rec}
are chosen precisely so that the induction closes.
\end{proof}

\subsection{Matched-trajectory strong error for output observables}\label{subsec:strong_explicit_TL}

\begin{theorem}[Matched-trajectory strong error in exact kernel form]\label{thm:direct_strong_output}
Let $X_L(T)$ be an output functional depending on the layer-$L$ spike train and assume that on
matched spike histories,
\begin{equation}\label{eq:output_direct_lip}
\abs{X_L^h(T)-X_L(T)}^2
\le
C_X
\sum_{j=1}^{n_L}\sum_{k:\,s^k_{L,j}\le T}
\psi\!\bigl(\abs{\varepsilon^k_{L,j}}\bigr).
\end{equation}
Then
\begin{equation}\label{eq:direct_output_bound}
\E\Bigl[
\abs{X_L^h(T)-X_L(T)}^2\,\mathbf 1_{\mathcal G^{\mathrm{match}}(h,T)}
\Bigr]
\le
C_X\,n_L\,\mathcal M_L(T).
\end{equation}
Consequently, under the hypotheses of Theorem~\ref{thm:direct_strong_recursion},
\begin{equation}\label{eq:direct_output_bound_exact}
\E\Bigl[
\abs{X_L^h(T)-X_L(T)}^2\,\mathbf 1_{\mathcal G^{\mathrm{match}}(h,T)}
\Bigr]
\le
C_X\,n_L\,
\max_{1\le j\le n_L}
\int_0^T\int_0^\infty
\Gamma_L(a;h,T)\,
a\,\rho_{L,j}(\Vth,\Ith+a,t)\,da\,dt.
\end{equation}
According to Corollary~\ref{cor:direct_depth}, we obtain
\begin{equation}\label{eq:strong_explicit_new}
\E\Bigl[
\abs{X_L^h(T)-X_L(T)}^2\,\mathbf 1_{\mathcal G^{\mathrm{match}}(h,T)}
\Bigr]
\le
C_X\,n_L\,K_L(T,W)\,h\,(1+\abs{\log h})^{q_L}.
\end{equation}
In particular, if $\gamma_1=\cdots=\gamma_L=0$, then
\begin{equation}\label{eq:strong_explicit_alllog}
\E\Bigl[
\abs{X_L^h(T)-X_L(T)}^2\,\mathbf 1_{\mathcal G^{\mathrm{match}}(h,T)}
\Bigr]
\le
C_X\,n_L\,K_L(T,W)\,h\,(1+\abs{\log h})^L.
\end{equation}
\end{theorem}

\begin{proof}
The first inequality is immediate from \eqref{eq:output_direct_lip} and the definition of
$\mathcal M_L(T)$.  The exact-kernel bound \eqref{eq:direct_output_bound_exact} follows from
Theorem~\ref{thm:direct_strong_recursion}.  The explicit all-log estimate follows from
Corollary~\ref{cor:direct_depth}.
\end{proof}

\begin{figure}[bthp]
\centering
\begin{subfigure}[t]{0.495\textwidth}
  \centering
  \begin{overpic}[width=\linewidth]{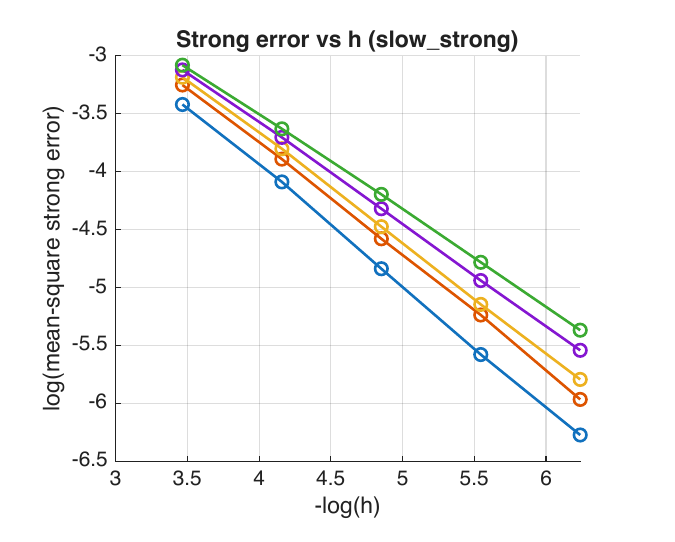}
    \put(2,76){\large\bfseries A}
  \end{overpic}
\end{subfigure}\hfill
\begin{subfigure}[t]{0.495\textwidth}
  \centering
  \begin{overpic}[width=\linewidth]{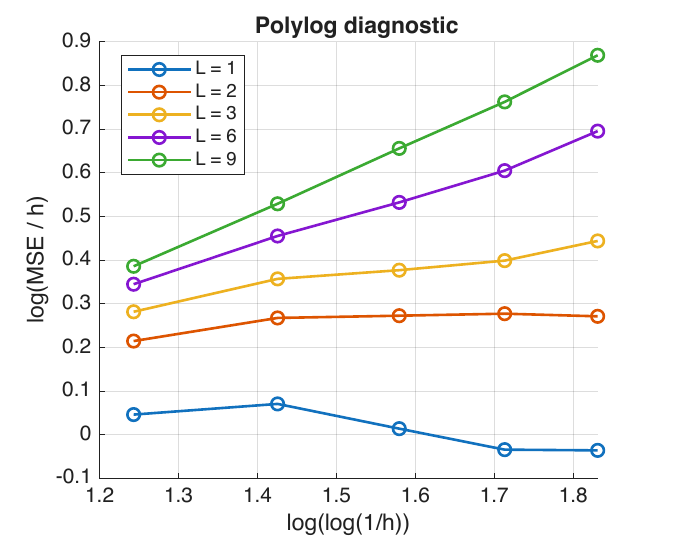}
    \put(2,76){\large\bfseries B}
  \end{overpic}
\end{subfigure}
\caption{Strong-error illustration. We simulate a feedforward network with nine layers with 24 LIF neurons in each layer. More details are given in Appendix. A: mean-square strong error versus $h$. B: polylogarithmic diagnostic for $\mathrm{MSE}/h$.}
\label{fig:strong_illustration}
\end{figure}


\subsection{Spike-count mismatch bad set}\label{subsec:mismatch_count}

The matched-trajectory theorem above does not include trajectories on which the exact and numerical
spike counts disagree at the observation horizon.  We now record the corresponding mismatch bound.

\paragraph{Horizon mismatch.}
For each neuron $(\ell,j)$ define
\begin{equation}\label{eq:Bmis_def}
B^{\mathrm{mis}}_{\ell,j}(h,T):=\{N^h_{\ell,j}(T)\neq N_{\ell,j}(T)\}.
\end{equation}

\paragraph{Boundary-layer and dense-spike auxiliaries.}
Fix $\delta\in(0,1]$ and define
\begin{align}
B^{\partial}_{\ell,j}(\delta,T)
&:=\{N_{\ell,j}(T+\delta)-N_{\ell,j}(T-\delta)\ge 1\},
\label{eq:Bbd_def}\\
B^{(2)}_{\ell,j}(\delta,T)
&:=\{\exists k:\ s^{k+1}_{\ell,j}-s^k_{\ell,j}\le 2\delta,\ \ s^k_{\ell,j}\le T+\delta\},
\label{eq:Bpair_def}\\
B^{\mathrm{err}}_{\ell,j}(h,\delta,T)
&:=\left\{\max_{k:\,s^k_{\ell,j}\le T+\delta}\abs{\varepsilon^k_{\ell,j}}>\delta\right\}.
\label{eq:Berr_def}
\end{align}

\begin{lemma}[Count mismatch decomposition]\label{lem:mis_decomp}
For every $\delta>0$,
\[
B^{\mathrm{mis}}_{\ell,j}(h,T)
\subseteq
B^{\partial}_{\ell,j}(\delta,T)\cup
B^{(2)}_{\ell,j}(\delta,T)\cup
B^{\mathrm{err}}_{\ell,j}(h,\delta,T).
\]
\end{lemma}

\begin{proof}
If there is no exact spike within
$\delta$ of $T$, no two exact spikes are separated by at most $2\delta$ up to time $T+\delta$,
and every matched spike-time shift is at most $\delta$, then index pairing is preserved and
\[
s^k_{\ell,j}\le T\quad\Longleftrightarrow\quad \hat s^k_{\ell,j}\le T
\]
for every $k$. Hence $N^h_{\ell,j}(T)=N_{\ell,j}(T)$.
\end{proof}

\begin{assumption}[Local rate and dense-spike bounds]\label{ass:mismatch_rate}
For each neuron $(\ell,j)$ there exist constants
$r^{\max}_{\ell,j}(T),R^{(2)}_{\ell,j}(T)<\infty$ such that
\begin{align}
\E[N_{\ell,j}(T+\delta)-N_{\ell,j}(T-\delta)]
&\le 2\,r^{\max}_{\ell,j}(T)\,\delta,
\qquad 0<\delta\le1,
\label{eq:rate_local_bound}\\
\Pp\bigl(B^{(2)}_{\ell,j}(\delta,T)\bigr)
&\le 4\,R^{(2)}_{\ell,j}(T)\,\delta^2,
\qquad 0<\delta\le1.
\label{eq:pair_bound}
\end{align}
In addition, we will use that spike-time errors have Gaussian-type concentration at scale $\sqrt{h}$.
We record this as an assumption tailored to the current-driven LIF setting (where all driving noises are Gaussian):
there exist constants $C_{\mathrm{sg}},c_{\mathrm{sg}}>0$ such that for all $\delta\in(0,1]$,
\begin{equation}\label{eq:STE_subG}
\sup_{k\ge 1}\Pp\big(|\varepsilon^k_{\ell,j}|>\delta,\ s^k_{\ell,j}\le T+1\big)
\ \le\ C_{\mathrm{sg}}\exp\!\Big(-c_{\mathrm{sg}}\,\frac{\delta^2}{h}\Big).
\end{equation}
\end{assumption}

\begin{lemma}[Spike-count mismatch probability]\label{lem:mis_prob}
Assume Assumptions~\ref{ass:boundary_depletion} and \ref{ass:mismatch_rate}.  Then, for every
$\delta\in(0,1]$,
\begin{equation}\label{eq:mis_prob_new}
\Pp\bigl(B^{\mathrm{mis}}_{\ell,j}(h,T)\bigr)
\le
2\,r^{\max}_{\ell,j}(T)\,\delta
+
4\,R^{(2)}_{\ell,j}(T)\,\delta^2
+
C_{\mathrm{sg}}\exp\!\Big(-c_{\mathrm{sg}}\,\frac{\delta^2}{R^2_\ell(h,T)}\Big).
\end{equation}
In particular, choosing $\delta=\delta_h:=\kappa\sqrt{h\log(1/h)}$ yields
\begin{equation}\label{eq:mis_prob_choice}
\Pp\big(B^{\mathrm{mis}}_{\ell,j}(h,T)\big)
\ \lesssim\
r^{\max}_{\ell,j}(T)\,\sqrt{h\log(1/h)}
+\ R^{(2)}_{\ell,j}(T)\,h\log(1/h)
+\ C^{\mathrm{err}}_{\ell,j}(T)\,h^{c_0},
\qquad c_0:=c_{\mathrm{sg}}\kappa^2.
\end{equation}
\end{lemma}

\begin{remark}
Section~\ref{sec:strong} proves two complementary statements.
\begin{enumerate}
    \item On the matched event \eqref{eq:Gmatch_feedforward_def}, it yields the explicit
    matched-trajectory strong bound \eqref{eq:strong_explicit_new}, and in particular the
    all-log bound \eqref{eq:strong_explicit_alllog}.
    \item Separately, it gives the horizon spike-count mismatch estimate
    \eqref{eq:mis_prob_choice}, which is of order $O(\sqrt{h\log(1/h)})$ under the stated
    rate, burst, and spike-time-tail assumptions.
\end{enumerate}
\end{remark}

\section{Weak error: an averaged order-$1$ theorem with explicit dependence on depth, time, and weights}\label{sec:weak}

Weak error cannot be controlled by pruning: missed or extra spikes must be handled inside the expectation.
The key technical point is that, near threshold, one should \emph{not} seek a pointwise operator bound of order $h^2$.
Indeed, if a deterministic state $x$ lies at distance $O(h)$ below threshold, then the exact one-step dynamics and the grid spike map may make different spike/no-spike decisions with $O(1)$ probability.
What is small is the \emph{averaged} one-step defect, because the jump discrepancy vanishes on the threshold surface by transmission and is therefore proportional to the distance from threshold.
Accordingly, Section~\ref{sec:weak} proves an \emph{averaged} weak-order theorem along the numerical chain.

When discussing weak errors, we will sometimes use neuron index $p = (\ell,j)$ in a completely interchangeable fashion.

\subsection{State space, spike maps, and the exact Markov semigroup}
Let $\mathcal{J}:=\{p=(\ell,j):\,1\le \ell\le L,\ 1\le j\le n_\ell\}$ and $N:=\sum_{\ell=1}^L n_\ell$.
Write the network state as $x=(v,I)\in\R^{2N}$ with coordinates $(v_{p},I_{p})_{(p)\in\mathcal{J}}$.
Define the interior (subthreshold) domain
\[
\mathsf{D}:=\{x=(v,I)\in\R^{2N}:\ v_{p}<\Vth\ \text{for all }(p)\in\mathcal{J}\}.
\]
For each neuron $p\in\mathcal{J}$ define the threshold surface
\[
\Sigma_p:=\{x\in\overline{\mathsf{D}}:\ v_{p}=\Vth\}.
\]

\paragraph{Spike map (reset + feedforward synaptic jump).}
When neuron $p=(\ell,i)$ spikes, it resets its own voltage and (if $\ell\le L-1$) produces an instantaneous
current jump in layer $\ell+1$. Define $\mathcal{R}_p:\overline{\mathsf{D}}\to\mathsf{D}$ by
\begin{enumerate}[label=(\roman*),leftmargin=2.2em]
\item $(\mathcal{R}_p x)_{v_{\ell,i}}:=\Vr$,
\item if $\ell\le L-1$, then for each $j\in\{1,\dots,n_{\ell+1}\}$,
\[
(\mathcal{R}_p x)_{I_{\ell+1,j}}:=I_{\ell+1,j}+W^{\ell,\ell+1}_{j i},
\]
\item all other coordinates are unchanged.
\end{enumerate}

\begin{lemma}[Commutativity of feedforward spike maps]\label{lem:feedforward_spikemap_commute}
For the current-based feedforward model, $\mathcal{R}_p\mathcal{R}_q=\mathcal{R}_q\mathcal{R}_p$ for all
$p,q\in\mathcal J$.
Hence the product $\prod_{p\in \mathcal S^{\mathrm{grid}}(x)}\mathcal R_p$ is independent of the order.
\end{lemma}

\begin{proof}
Each spike map resets one voltage coordinate and adds a fixed current column to the next layer.
Voltage resets on different coordinates commute, current additions commute because they are additive,
and no spike map modifies the current coordinate of the neuron whose voltage it resets.
Therefore the maps commute pairwise.
\end{proof}

\paragraph{Exact process and semigroup.}
Let $X(t)=(V(t),I(t))$ be the exact network process:
between spikes it solves \eqref{eq:dv}--\eqref{eq:dI}, and when $X(t^-)\in\Sigma_p$ it jumps to
$X(t^+)=\mathcal{R}_p(X(t^-))$.
For bounded measurable $\Phi:\mathsf{D}\to\R$, define
\[
(P_{s,t}\Phi)(x):=\E\big[\Phi(X(t))\,\big|\,X(s)=x\big],\qquad 0\le s\le t.
\]

\subsection{EM one-step operator}
Fix $h>0$ and grid times $t_m:=mh$.
Given $x\in\mathsf{D}$ at time $t_m$, define the EM pre-update
\[
v^{\mathrm{pre}}_{p}=v_{p}+h\Bigl(-\tfrac{1}{\tauv}(v_{p}-\Vr)+I_{p}\Bigr),
\qquad
I^{\mathrm{pre}}_{p}=I_{p}-\frac{h}{\tauc}(I_{p}-b_{p}(t_m))+\frac{\sigma_{p}}{\tauc}\Delta B_{p,m},
\]
with independent $\Delta B_{p,m}\sim\mathcal{N}(0,h)$.
Define the grid spike set
\[
\mathcal{S}^{\mathrm{grid}}(x):=\{p\in\mathcal{J}:\ v^{\mathrm{pre}}_{p}\ge \Vth\}.
\]
Apply all spike maps at the grid endpoint:
\[
x^{\mathrm{post}}
=
\Bigl(\prod_{p\in\mathcal{S}^{\mathrm{grid}}(x)}\mathcal{R}_p\Bigr)(v^{\mathrm{pre}},I^{\mathrm{pre}}).
\]
Define the EM one-step operator
\[
(Q_h\Psi)(x):=\E\big[\Psi(x^{\mathrm{post}})\big].
\]

\subsection{Backward value function and transmission at threshold}
Fix $T=Mh$ and $\Phi\in C_b^4(\mathsf{D})$.
Define
\[
u(t,x):=(P_{t,T}\Phi)(x)=\E\big[\Phi(X(T))\,\big|\,X(t)=x\big],\qquad 0\le t\le T.
\]

\begin{assumption}[Backward regularity and transmission for the reset semigroup]\label{ass:weak_backward_reg}
For every terminal observable $\Phi\in C_b^4(\mathsf D)$ used below, the value function $u$
satisfies the following properties:
\begin{enumerate}[label=(\roman*),leftmargin=2.2em]
\item for every $\eta>0$, $u\in C^{1,4}([0,T-\eta]\times\mathsf D)$;
\item all spatial derivatives of $u$ up to order $4$ that appear in the weak expansion are
bounded on $[0,T-\eta]\times\mathsf D$;
\item on the interior, $u$ solves the backward Kolmogorov equation
$\partial_t u+\mathcal L_t u=0$ with terminal condition $u(T,\cdot)=\Phi$;
\item for every boundary face $\Sigma_p$, every $t<T$, and every boundary point
$x\in\Sigma_p$ reachable by upward crossing,
\begin{equation}\label{eq:transmission_u}
u(t,x)=u(t,\mathcal R_p x).
\end{equation}
\end{enumerate}
\end{assumption}

\begin{remark}[Status of Assumption~\ref{ass:weak_backward_reg}]
This is the main analytic input behind the weak theory.
General semigroup results for diffusion operators with nonlocal boundary conditions provide
background for the reset/gluing structure, and forward Fokker--Planck equations for certain
integrate-and-fire models are classical \cite{ArendtKunkelKunze2016,CarrilloGonzalezGualdaniSchonbek2013}.
However, these results do not directly yield the precise $C^{1,4}$ backward regularity required for
the present finite-network current-driven reset system.
Accordingly, all weak-order statements below are formulated conditionally on
Assumption~\ref{ass:weak_backward_reg}.
\end{remark}

Under Assumption~\ref{ass:weak_backward_reg}, the interior generator is
\begin{equation}\label{eq:generator}
\mathcal{L}_t
=
\sum_{p\in\mathcal{J}}
\Bigl(-\tfrac{1}{\tauv}(v_{p}-\Vr)+I_{p}\Bigr)\partial_{v_{p}}
+
\sum_{p\in\mathcal{J}}
\Bigl(-\tfrac{1}{\tauc}(I_{p}-b_{p}(t))\Bigr)\partial_{I_{p}}
+
\sum_{p\in\mathcal{J}}\frac{\sigma_{p}^2}{2\tauc^2}\partial_{I_{p}I_{p}}.
\end{equation}

\subsection{Averaged one-step control package}
Write $X_m^h:=(V_m^h,I_m^h):=X^h(t_m)$ for the numerical chain.
For each step $m$ and each realization of $X_m^h$, let
$\widetilde X^{(m)}(t)$, $t\in[t_m,t_{m+1}]$, denote the \emph{exact} feedforward process started from
$\widetilde X^{(m)}(t_m)=X_m^h$ and driven by fresh Brownian increments over $[t_m,t_{m+1}]$.
Let
\[
\Delta \widetilde N_{\ell}^{(m)}
:=
\sum_{i=1}^{n_\ell}\Bigl(\widetilde N_{\ell,i}^{(m)}(t_{m+1})-\widetilde N_{\ell,i}^{(m)}(t_m)\Bigr),
\qquad
\Delta \widetilde N_{\mathrm{tot}}^{(m)}:=\sum_{\ell=1}^L \Delta \widetilde N_{\ell}^{(m)} .
\]
Thus $\Delta \widetilde N_{\ell}^{(m)}$ counts \emph{exact} layer-$\ell$ spikes during the single step
when the exact evolution is started from the numerical state $X_m^h$.

For each neuron $p\in\mathcal J$, let $v_p^{\mathrm{pre}}(X_m^h)$ be the EM pre-update coordinate defined
above, and let
\[
D_{p,m}
:=
\Bigl\{
\mathbf 1_{\{\widetilde N^{(m)}_p(t_{m+1})-\widetilde N^{(m)}_p(t_m)\ge 1\}}
\neq
\mathbf 1_{\{v_p^{\mathrm{pre}}(X_m^h)\ge \Vth\}}
\Bigr\}
\]
be the exact-vs.-grid spike-decision discrepancy event for neuron $p$ in step $m$.

\begin{assumption}[Averaged one-step rate, strip, and factorial-moment bounds]\label{ass:weak_step_pkg}
There exist finite constants $R_\ell(T)$, $R_{\mathrm{ff},2}(T)$, $\rho^V_{\max,p}(T)$,
$A_{2,p}(T)$, and $C_{\eta,p}(T)$ such that for every step $m$ with $t_{m+1}\le T$:
\begin{enumerate}[label=(\roman*),leftmargin=2.2em]
\item \textbf{Averaged layerwise rate bound:}
\begin{equation}\label{eq:weak_rate_bound_layer}
\E[\Delta \widetilde N_\ell^{(m)}]\le R_\ell(T)\,h,
\qquad \ell=1,\dots,L.
\end{equation}
\item \textbf{Second factorial moment of exact spikes in the step:}
\begin{equation}\label{eq:weak_ff_factorial}
\E\!\Big[\Delta \widetilde N_{\mathrm{tot}}^{(m)}\bigl(\Delta \widetilde N_{\mathrm{tot}}^{(m)}-1\bigr)\Big]
\le R_{\mathrm{ff},2}(T)\,h^2.
\end{equation}
\item \textbf{Numerical strip density and velocity moment bounds:}
for each $p\in\mathcal J$, the marginal law of $V_{p,m}^h$ admits a density $\varrho^h_{p,m}(v)$ and
\begin{equation}\label{eq:weak_num_strip}
\sup_{m:\,t_m\le T}\sup_{y\in[0,1]}\varrho_{p,m}^h(\Vth-y)\le \rho^V_{\max,p}(T),
\qquad
\sup_{m:\,t_m\le T}\E\big[(a_p^+(X_m^h))^2\big]\le A_{2,p}(T),
\end{equation}
where
\[
a_p(x):=-\tfrac{1}{\tauv}(v_p-\Vr)+I_p,
\qquad
a_p^+(x):=\max\{a_p(x),0\}.
\]
\item \textbf{One-step decision-strip remainder:}
there exist nonnegative random variables $\eta_{p,m}$ such that on the event
$\{\Delta \widetilde N_{\mathrm{tot}}^{(m)}\le 1\}$,
\begin{equation}\label{eq:weak_discrepancy_strip}
D_{p,m}
\subset
\Bigl\{
0<\Vth-V_{p,m}^h\le h\,a_p^+(X_m^h)+\eta_{p,m}
\Bigr\},
\end{equation}
and
\begin{equation}\label{eq:weak_eta_bound}
\E[\eta_{p,m}^2]\le C_{\eta,p}(T)\,h^3.
\end{equation}
\end{enumerate}
\end{assumption}

\begin{remark}[Why Assumption~\ref{ass:weak_step_pkg}(iv) is the right scale]
In the current-driven LIF model, the current fluctuation over one step is $O(\sqrt h)$, while the
voltage integrates that fluctuation over time and therefore acquires an $O(h^{3/2})$ correction.
The remainder scale in \eqref{eq:weak_eta_bound} is precisely the corresponding one-step boundary-layer scale.
\end{remark}

\paragraph{Backward derivative bounds.}
Under Assumption~\ref{ass:weak_backward_reg}, define
\begin{equation}\label{eq:M1_bounds}
M_v(T):=\sup_{t\in[0,T)}\max_{p\in\mathcal{J}}\|\partial_{v_p}u(t,\cdot)\|_\infty,
\qquad
M_I(T):=\sup_{t\in[0,T)}\|\nabla_I u(t,\cdot)\|_\infty.
\end{equation}

\subsection{OU proxy scales for the boundary constants}\label{subsec:OU_proxy}
The weak theorem itself uses only the abstract constants $\rho^V_{\max,p}(T)$ from
Assumption~\ref{ass:weak_step_pkg}.  It is tempting to estimate these constants from an effective
OU closure of the current.  This can be useful heuristically, but it should be kept separate from
the proof.

Assume formally that, for neuron $p=(\ell,j)$,
\begin{equation}\label{eq:I_eff}
dI_p(t) = -\tfrac{1}{\tauc}(I_p(t)-\mu_p(t))\,dt
+\tfrac{\sigma_{\mathrm{eff},p}(t)}{\tauc}\,dB_p(t),
\end{equation}
with $\sigma_{\mathrm{eff},p}(t)\ge \sigma_{\min}>0$, and in feedforward settings
\begin{equation}\label{eq:sigmaeff}
\sigma_{\mathrm{eff},p}^2(t)
\approx (\sigma_{p}^{\mathrm{ext}})^2
+\kappa\,\tauc^2\sum_{i=1}^{n_{\ell-1}} (W^{\ell-1,\ell}_{ji})^2\,r_{\ell-1,i}(t).
\end{equation}
For the \emph{full-space linear} $(v,I)$ OU system without threshold/reset,
\[
dv = \Bigl(-\tfrac{1}{\tauv}(v-\Vr)+I\Bigr)\,dt,
\qquad
dI = -\tfrac{1}{\tauc}(I-\mu)\,dt + \tfrac{\sigma_{\mathrm{eff}}}{\tauc}\,dB,
\]
the stationary covariance is explicit:
\[
\operatorname{Var}(I)=\frac{\sigma_{\mathrm{eff}}^2}{2\tauc},\qquad
\operatorname{Cov}(v,I)=\frac{\sigma_{\mathrm{eff}}^2\,\tauv}{2(\tauc+\tauv)},\qquad
\operatorname{Var}(v)=\frac{\sigma_{\mathrm{eff}}^2\,\tauv^2}{2(\tauc+\tauv)}.
\]
Hence the stationary Gaussian suprema are
\begin{equation}\label{eq:rhomax_proxy}
\sup_{(v,I)\in\R^2}\rho_{\mathrm{lin}}(v,I,\infty)
=
\frac{\sqrt{\tauc}(\tauc+\tauv)}{\pi\,\sigma_{\mathrm{eff}}^2\,\tauv^{3/2}},
\qquad
\sup_{v\in\R}\varrho_{\mathrm{lin}}^V(v,\infty)
=
\frac{\sqrt{\tauc+\tauv}}{\sqrt{\pi}\,\sigma_{\mathrm{eff}}\,\tauv}.
\end{equation}

\begin{remark}[Why \eqref{eq:rhomax_proxy} is only a proxy]\label{rem:weak_proxy_only}
Equation \eqref{eq:rhomax_proxy} describes the stationary \emph{unconstrained} OU surrogate.
It does \emph{not} by itself give a theorem for the true reset process: the latter involves killing
at $v=\Vth$, reinjection at $v=\Vr$, and a boundary flux that can create additional boundary layers.
Accordingly, we use \eqref{eq:rhomax_proxy} only as a scale estimate for
$\rho^V_{\max,p}(T)$ and never as a substitute for
Assumption~\ref{ass:weak_step_pkg}.
\end{remark}

\subsection{Three one-step lemmas}

\begin{lemma}[Jump discrepancy vanishes linearly at threshold]\label{lem:weak_jump_linear}
Fix a neuron $p\in\mathcal J$ and define, for $t<T$,
\[
g_p(t,x):=u(t,\mathcal R_p x)-u(t,x).
\]
Then $g_p(t,x)=0$ for $x\in\Sigma_p$, and for interior states
\begin{equation}\label{eq:g_linear}
\abs{g_p(t,x)}\le \|\partial_{v_p}u(t,\cdot)\|_\infty\,(\Vth-v_p).
\end{equation}
\end{lemma}

\begin{proof}
Let $\pi_p x$ denote the boundary projection replacing $v_p$ by $\Vth$ and leaving all other coordinates
unchanged. Since $\mathcal R_p$ overwrites $v_p$, we have $\mathcal R_p x=\mathcal R_p(\pi_p x)$.
Transmission \eqref{eq:transmission_u} yields
$u(t,\pi_p x)=u(t,\mathcal R_p(\pi_p x))=u(t,\mathcal R_p x)$ for $t<T$.
Therefore $g_p(t,x)=u(t,\pi_p x)-u(t,x)$, and the mean value theorem gives \eqref{eq:g_linear}.
\end{proof}

\begin{lemma}[Averaged timing/decay defect on the single-spike event]\label{lem:weak_timing}
Fix a step $m$ with $t_{m+1}<T$ and set $f=u(t_{m+1},\cdot)$.
On the event $\{\Delta \widetilde N_{\mathrm{tot}}^{(m)}\le 1\}$, shifting the unique exact spike
in the step (if one occurs) from its true time $\tau\in(t_m,t_{m+1})$ to the grid endpoint
$t_{m+1}$ changes the value of $f$ by at most
\[
\frac{h}{\tauc}\,(M_I(T)+M_v(T))\,\sum_{\ell=1}^{L-1}\|W^{\ell,\ell+1}\|_1
\]
times the number of exact spikes in the corresponding layer.
Consequently,
\begin{equation}\label{eq:timing_bound}
\big|\E[\textnormal{timing/decay defect in step }m]\big|
\le
\frac{h^2}{\tauc}\,(M_I(T)+M_v(T))
\sum_{\ell=1}^{L-1}R_\ell(T)\,\|W^{\ell,\ell+1}\|_1 .
\end{equation}
\end{lemma}

\begin{proof}
If the unique exact spike in the step belongs to layer $\ell$, then shifting its time from
$\tau$ to $t_{m+1}$ changes the postsynaptic current vector at the grid endpoint by at most
\[
\sum_{j=1}^{n_{\ell+1}} |W^{\ell,\ell+1}_{ji}|\,
\bigl(1-e^{-(t_{m+1}-\tau)/\tauc}\bigr)
\le
\frac{h}{\tauc}\,\|W^{\ell,\ell+1}\|_1 .
\]
The associated change in the postsynaptic voltages over the remaining part of the step is one more
time integration of the same current defect and is therefore smaller by an additional factor $h$.
Using the derivative bounds \eqref{eq:M1_bounds}, this changes $f$ by at most
$\frac{h}{\tauc}(M_I(T)+M_v(T))\|W^{\ell,\ell+1}\|_1$ per exact layer-$\ell$ spike.
Taking expectation and using \eqref{eq:weak_rate_bound_layer} yields \eqref{eq:timing_bound}.
\end{proof}

\begin{lemma}[Averaged boundary-strip defect on the single-spike event]\label{lem:weak_missed}
Fix $p\in\mathcal J$ and a step $m$ with $t_{m+1}<T$.
Then $\exists \tilde{A}_{2,p} <2A_{2,p}(T)$ s.t.,
\begin{equation}\label{eq:strip_integral}
\E\Big[\abs{g_p(t_{m+1},X_m^h)}\,
\mathbf 1_{D_{p,m}\cap\{\Delta \widetilde N_{\mathrm{tot}}^{(m)}\le 1\}}
\Big]
\le
M_v(T)\,\rho^V_{\max,p}(T)\,\tilde{A}_{2,p}\,h^2 .
\end{equation}
\end{lemma}

\begin{proof}
Let $d_p:=\Vth-V_{p,m}^h$.
By \eqref{eq:g_linear},
\[
\abs{g_p(t_{m+1},X_m^h)}
\le M_v(T)\,d_p\,\mathbf 1_{\{d_p>0\}}.
\]
On the single-spike event, Assumption~\ref{ass:weak_step_pkg}(iv) yields
\[
D_{p,m}
\subset
\{0<d_p\le h\,a_p^+(X_m^h)+\eta_{p,m}\}.
\]
Hence
\[
\E\Big[\abs{g_p(t_{m+1},X_m^h)}\,
\mathbf 1_{D_{p,m}\cap\{\Delta \widetilde N_{\mathrm{tot}}^{(m)}\le 1\}}\Big]
\le
M_v(T)\,\E\Big[d_p\,\mathbf 1_{\{0<d_p\le h a_p^+ + \eta_{p,m}\}}\Big].
\]
Conditioning on $(a_p^+(X_m^h),\eta_{p,m})$ and using the numerical strip density bound
\eqref{eq:weak_num_strip},
\[
\E\Big[d_p\,\mathbf 1_{\{0<d_p\le h a_p^+ + \eta_{p,m}\}}\ \Big|\ a_p^+(X_m^h),\eta_{p,m}\Big]
\le
\rho^V_{\max,p}(T)\,\frac{(h a_p^+ + \eta_{p,m})^2}{2}.
\]
Therefore
\[
\E[\cdots]
\le
\frac{M_v(T)\rho^V_{\max,p}(T)}{2}\,
\E\big[(h a_p^+(X_m^h)+\eta_{p,m})^2\big].
\]
Using $(x+y)^2\le 2x^2+2y^2$, \eqref{eq:weak_num_strip}, and \eqref{eq:weak_eta_bound},
\[
\E[(h a_p^++\eta_{p,m})^2]
\le
2h^2 A_{2,p}(T)+2 C_{\eta,p}(T) h^3
\le
\tilde{A}_{2,p}h^2
\]
because $h\le 1$.
Substituting this bound proves \eqref{eq:strip_integral}.
\end{proof}

\begin{lemma}[Multi-spike remainder]\label{lem:weak_multispike}
Fix a step $m$ with $t_{m+1}<T$ and set $f=u(t_{m+1},\cdot)$.
Then
\begin{equation}\label{eq:weak_multispike}
\big|\E[\textnormal{one-step defect on }\{\Delta \widetilde N_{\mathrm{tot}}^{(m)}\ge 2\}]\big|
\le
2\,\|u\|_\infty\,R_{\mathrm{ff},2}(T)\,h^2 .
\end{equation}
\end{lemma}

\begin{proof}
On the event $\{\Delta \widetilde N_{\mathrm{tot}}^{(m)}\ge 2\}$ we use the trivial bound
$|f(\cdot)|\le \|u\|_\infty$.
By Markov's inequality and \eqref{eq:weak_ff_factorial},
\[
\Pp(\Delta \widetilde N_{\mathrm{tot}}^{(m)}\ge 2)
\le
\frac{1}{2}\E\!\Big[\Delta \widetilde N_{\mathrm{tot}}^{(m)}
\bigl(\Delta \widetilde N_{\mathrm{tot}}^{(m)}-1\bigr)\Big]
\le
\frac12 R_{\mathrm{ff},2}(T)\,h^2.
\]
Absorbing the factor $1/2$ into the constant yields \eqref{eq:weak_multispike}.
\end{proof}

\subsection{One-step and global weak error bounds}
\begin{lemma}[Averaged one-step weak defect]\label{lem:weak_onestep}
Let $t_{m+1}<T$ and set $f=u(t_{m+1},\cdot)$.
Assume the subthreshold affine diffusion EM truncation yields the standard weak local bound
\begin{equation}\label{eq:diff_local_weak}
\sup_{x\in\mathsf D}
\big|
(Q_h^{\mathrm{diff}}-P_{t_m,t_{m+1}}^{\mathrm{diff}})f(x)
\big|
\le
C_{\mathrm{diff}}(T)\,h^2,
\end{equation}
where $Q_h^{\mathrm{diff}}$ and $P^{\mathrm{diff}}$ denote the one-step EM and exact operators for the
subthreshold affine diffusion without threshold/reset.
Then
\begin{align}
&\abs{\E\big[u(t_{m+1},X_{m+1}^h)-u(t_m,X_m^h)\big]}
\nonumber\\
&\le
h^2\Biggl(
C_{\mathrm{diff}}(T)
+
\frac{M_I(T)+M_v(T)}{\tauc}
\sum_{\ell=1}^{L-1}R_\ell(T)\,\|W^{\ell,\ell+1}\|_1 \nonumber\\
&\qquad+
M_v(T)\sum_{p\in\mathcal J}\rho^V_{\max,p}(T)\tilde{A}_{2,p}
+
2\|u\|_\infty R_{\mathrm{ff},2}(T)
\Biggr).
\label{eq:onestep_weak_full}
\end{align}
\end{lemma}

\begin{proof}
Write
\[
\E\big[u(t_{m+1},X_{m+1}^h)-u(t_m,X_m^h)\big]
=
\E\big[(Q_h-P_{t_m,t_{m+1}})f(X_m^h)\big].
\]
We decompose the one-step defect into four parts.

\emph{(i) Subthreshold affine-diffusion truncation.}
On the event that no exact spike occurs in the step, the one-step defect is exactly the standard
subthreshold weak local error \eqref{eq:diff_local_weak}.

\emph{(ii) Timing/decay of a unique exact spike.}
On the event $\{\Delta \widetilde N_{\mathrm{tot}}^{(m)}=1\}$, shift the unique exact spike to the grid endpoint
$t_{m+1}$.  Lemma~\ref{lem:weak_timing} gives the corresponding contribution.

\emph{(iii) Exact-vs.-grid spike-map decision for a unique exact spike.}
Still on $\{\Delta \widetilde N_{\mathrm{tot}}^{(m)}=1\}$, after the timing shift one compares
``apply $\mathcal R_p$ at the grid endpoint'' with ``do not apply $\mathcal R_p$''.
Lemma~\ref{lem:weak_jump_linear} shows that the jump discrepancy is proportional to the distance from
threshold, and Lemma~\ref{lem:weak_missed} bounds the averaged boundary-strip contribution.

\emph{(iv) Two or more exact spikes in the step.}
This is the multi-spike remainder controlled by Lemma~\ref{lem:weak_multispike}.

Summing the four contributions yields \eqref{eq:onestep_weak_full}.
\end{proof}

\begin{theorem}[Global averaged weak order $1$ with explicit $T$--$L$ scaling]\label{thm:weak_global}
Let $T=Mh$ and $\Phi\in C_b^4(\mathsf{D})$.
Assume Assumptions~\ref{ass:weak_backward_reg}--\ref{ass:weak_step_pkg} and that
Lemma~\ref{lem:weak_onestep} holds for steps $m=0,\dots,M-2$.
Then
\begin{equation}\label{eq:weak_global}
\abs{\E[\Phi(X(T))]-\E[\Phi(X^h(T))]}
\ \le\
T\,h\,C_{\mathrm{weak}}(T,L,W)
+
C_{\mathrm{term}}(T)\,h,
\end{equation}
where
\begin{align}
C_{\mathrm{weak}}(T,L,W)
&:=
C_{\mathrm{diff}}(T)
+
\frac{M_I(T)+M_v(T)}{\tauc}\sum_{\ell=1}^{L-1}R_\ell(T)\,\|W^{\ell,\ell+1}\|_{1} \nonumber\\
\label{eq:Cweak_def}
&\quad+
M_v(T)\sum_{p\in\mathcal{J}}\rho^V_{\max,p}(T)\tilde{A}_{2,p}
+
2\|u\|_\infty R_{\mathrm{ff},2}(T),
\end{align}
and $C_{\mathrm{term}}(T)$ controls the final step (where $u(T,\cdot)=\Phi$ and transmission need not hold).
If $\Phi$ is spike-map compatible at $T$ (i.e.\ $\Phi(x)=\Phi(\mathcal{R}_p x)$ for $x\in\Sigma_p$), then
$C_{\mathrm{term}}(T)=0$.
\end{theorem}

\begin{proof}
Use the telescoping identity
\[
\E[\Phi(X(T))]-\E[\Phi(X^h(T))]
=
\sum_{m=0}^{M-1}\E\big[u(t_{m+1},X_{m+1}^h)-u(t_m,X_m^h)\big].
\]
Apply Lemma~\ref{lem:weak_onestep} to $m=0,\dots,M-2$ to obtain $(M-1)h^2C_{\mathrm{weak}}$ and bound the
final step by $C_{\mathrm{term}}(T)h$.
\end{proof}

\begin{figure}[t]
\centering
\begin{subfigure}[t]{0.33\textwidth}
  \centering
  \begin{overpic}[width=\linewidth,trim=24 372 46 48,clip]{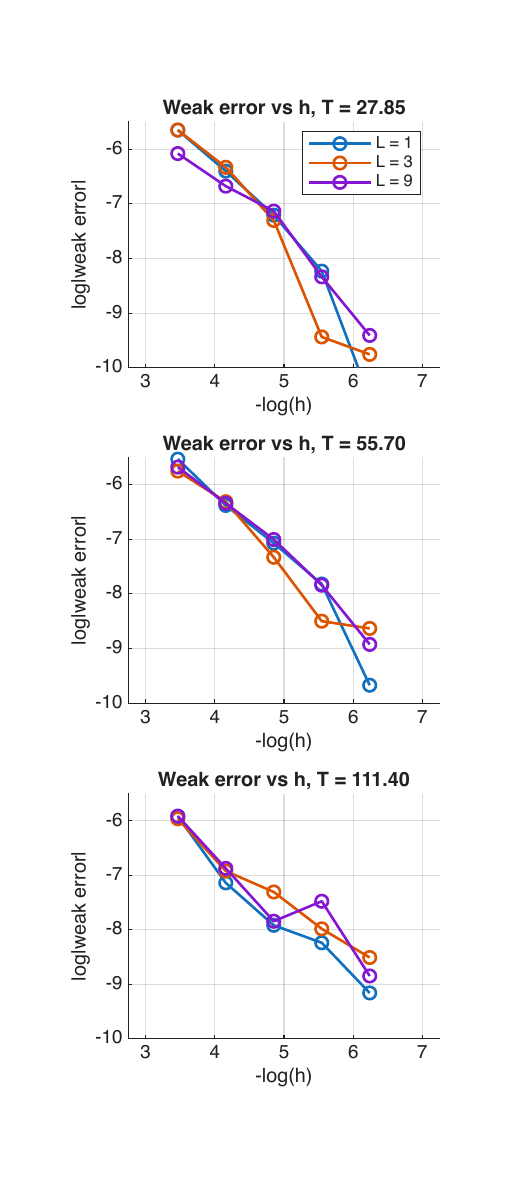}
    \put(2,85){\large\bfseries A}
  \end{overpic}
\end{subfigure}\hfill
\begin{subfigure}[t]{0.33\textwidth}
  \centering
  \includegraphics[width=\linewidth,trim=24 211 46 209,clip]{Fig3A.pdf}
\end{subfigure}\hfill
\begin{subfigure}[t]{0.33\textwidth}
  \centering
  \includegraphics[width=\linewidth,trim=24 50 46 370,clip]{Fig3A.pdf}
\end{subfigure}

\vspace{0.6em}

\begin{subfigure}[t]{0.33\textwidth}
  \centering
  \begin{overpic}[width=\linewidth,trim=12 12 31 12,clip]{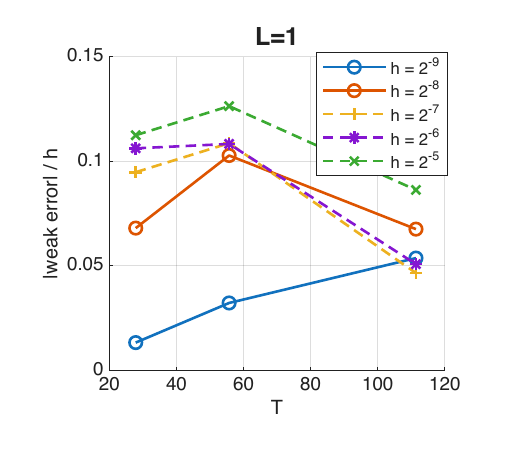}
    \put(2,92){\large\bfseries B}
  \end{overpic}
\end{subfigure}\hfill
\begin{subfigure}[t]{0.33\textwidth}
  \centering
  \includegraphics[width=\linewidth,trim=12 12 31 12,clip]{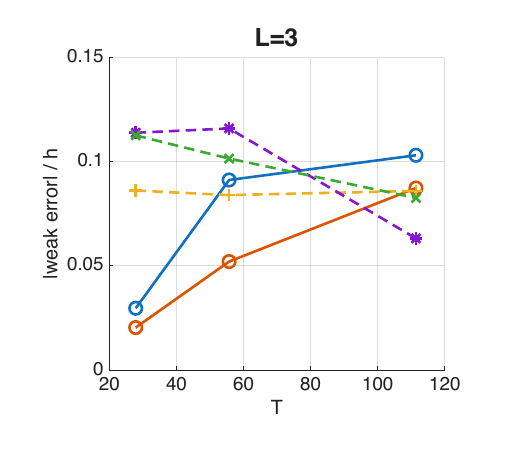}
\end{subfigure}\hfill
\begin{subfigure}[t]{0.33\textwidth}
  \centering
  \includegraphics[width=\linewidth,trim=12 12 31 12,clip]{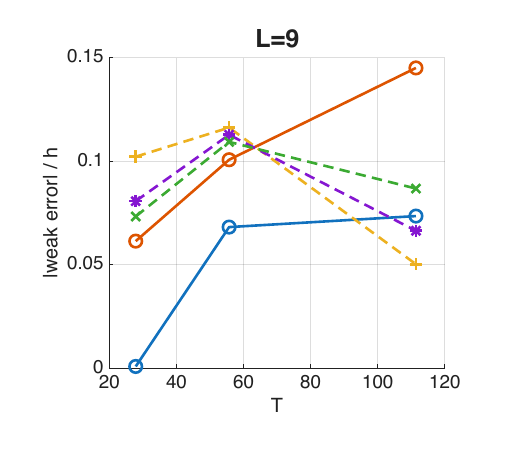}
\end{subfigure}
\caption{Weak-error illustration. A: Weak errors of layers 1, 3, and 9 are computed for three time-windows, going down with $h$. B: weak-error accumulation plots for $L=1,3,9$.}
\label{fig:weak_illustration}
\end{figure}

\begin{remark}[Linear growth in $L$ under uniform bounds]\label{rem:weak_linear_L}
If $R_\ell(T)\le R_{\max}(T)$, $\|W^{\ell,\ell+1}\|_1\le \mathcal{W}_1$, and
$\rho^{V}_{\max,p}(T)\bigl(A_{2,p}(T)+C_{\eta,p}(T)\bigr)$ is uniformly bounded while
$N=\sum_\ell n_\ell\lesssim L$, then $C_{\mathrm{weak}}(T,L,W)$ grows at most linearly in $L$.
Thus the weak error scales like $O(Th)$ with linear-in-$L$ constants under uniform controls.
\end{remark}
A numerical simulation illustrating the order one weak error is given  in Figure~\ref{fig:weak_illustration}.

\section{Dynamical stability and Lyapunov exponents}\label{sec:stability}

This section studies \emph{long-time} stability under \emph{common-noise coupling} and clarifies how
stability interacts with strong-error accumulation.
We integrate (i) a deterministic integrate-and-fire (IF) warm-up (monotonicity/interlacing and
finite-time spike-time sensitivity) and (ii) a general IF Lyapunov-exponent theorem based on
saltation factors and stationary boundary flux, then specialize to the current-based LIF neuron and
finally to feedforward networks.

\subsection{Common-noise coupling and what ``stability'' means}\label{subsec:stability_setup}

Fix a single neuron or a feedforward network as in Section~\ref{sec:model}.
Consider two copies $X(t)$ and $\widetilde X(t)$ driven by the \emph{same} Brownian motions
$\{B_{\ell,j}\}$ and, in the network setting, by their own internally generated spike trains.
We call this a \emph{common-noise coupling}.
Write $\delta v_{\ell,j}(t)=v_{\ell,j}(t)-\widetilde v_{\ell,j}(t)$ and similarly for currents.

Because the reset map is discontinuous, $L^\infty$-type path metrics are misleading:
even a small spike-time shift produces an $O(1)$ pointwise difference on a short \emph{misalignment window}
(see Lemma~\ref{lem:kernel_shift} and the discussion in Section~\ref{subsec:propagate}).
A more robust stability notion is based on (i) spike-time perturbations and/or (ii) time-integrated
errors (e.g.\ $L^1_t$), which ``pay'' only the window length.

The natural infinitesimal stability quantity is the \emph{top Lyapunov exponent} for perturbations in
voltage along a typical trajectory, defined (when it exists) by
\[
\lambda_v:=\lim_{T\to\infty}\frac{1}{T}\log\frac{|\delta v(T^+)|}{|\delta v(0^+)|},
\]
under common-noise coupling and an infinitesimal initial perturbation.
The key fact is that, between spikes, voltage perturbations contract (for LIF) at rate $e^{-t/\tauv}$,
while each spike produces a multiplicative \emph{saltation factor} due to the dependence of the spike
time on the perturbation (cf.\ \cite{Brette2004}).

To motivate (and to reuse later), we start with a deterministic 1D IF warm-up, which is also a
useful pathwise template because in the common-noise setting one may condition on the input path.

\subsection{Warm-up: deterministic 1D IF with common input}\label{subsec:IF_warmup}

We summarize and integrate deterministic results (monotonicity/interlacing and finite-time
spike-time sensitivity) in a form aligned with our notation.

\subsubsection{Model and spike times}
Let $x(t)\in[0,1)$ solve the deterministic IF dynamics
\begin{equation}\label{eq:IF_det}
\dot x(t)=f(x(t))+I(t),\qquad x(t)<1,
\end{equation}
with reset rule
\begin{equation}\label{eq:IF_reset}
x(t^-)=1 \ \Rightarrow\ x(t^+)=0.
\end{equation}
Assume
\begin{equation}\label{eq:IF_f_ass}
f\in C^1([0,1]),\qquad f'(x)\le 0\ \ \text{for all }x\in[0,1],
\end{equation}
and $I\in L^1_{\mathrm{loc}}(\R_+)$ (measurable input).
Given $x(0)=x_0\in[0,1)$, define spike times recursively:
\[
s^1(x_0):=\inf\{t>0:\ x(t)=1\},
\qquad
s^{k+1}(x_0):=\inf\{t>s^k(x_0):\ x(t)=1\},
\]
with the convention $\inf\emptyset=\infty$.
Let $N(T;x_0):=\sum_{k\ge 1}\1_{\{s^k(x_0)\le T\}}$.

\begin{remark}[Why $L^\infty$ is not the right stability metric]\label{rem:IF_Linfty_bad}
Even with identical input $I(t)$, a small initial perturbation can produce an $O(1)$ \emph{pointwise}
difference because one trajectory spikes earlier and resets while the other has not yet spiked.
However, this discrepancy persists only on an interval of length comparable to the spike-time shift.
Thus time-integrated metrics (and spike-time metrics) capture the correct smallness.
This is the same ``misalignment window'' mechanism used in the strong-error propagation analysis
(Section~\ref{subsec:propagate}).
\end{remark}

\subsubsection{Order preservation and contraction between spikes}
Let $x(t;s,\xi)$ denote the (subthreshold) ODE solution of \eqref{eq:IF_det} on an interval where
$x<1$, starting at time $s$ from $x(s)=\xi\in[0,1)$.
If $x_1(t)=x(t;s,\xi_1)$ and $x_2(t)=x(t;s,\xi_2)$ with $\xi_1\ge \xi_2$, then by comparison
(one-dimensional monotonicity),
\begin{equation}\label{eq:IF_order}
x_1(t)\ge x_2(t)\quad\text{for all times before the first threshold hit of either solution.}
\end{equation}
Moreover, letting $e(t):=x_1(t)-x_2(t)\ge 0$, we have
$\dot e(t)=f(x_1(t))-f(x_2(t))$ and $f'\le 0$ implies $\dot e(t)\le 0$, hence
\begin{equation}\label{eq:IF_contraction_between_spikes}
0\le x_1(t)-x_2(t)\le \xi_1-\xi_2
\quad\text{for all times before the first threshold hit.}
\end{equation}

\subsubsection{Monotonicity and (mean-driven) Lipschitz bounds for spike times}

\begin{assumption}[Mean-driven transversality for deterministic input]\label{ass:IF_mean_driven}
There exists a constant $a_*>0$ such that
\begin{equation}\label{eq:IF_a_star}
f(1)+\essinf_{t\ge 0} I(t)\ \ge\ a_* \ >0.
\end{equation}
\end{assumption}

\begin{theorem}[Spike times vs.\ initial condition: monotonicity, right continuity, and Lipschitz bounds]\label{thm:IF_spiketime}
Assume \eqref{eq:IF_f_ass}. Then for each fixed $k\ge 1$, the map $x_0\mapsto s^k(x_0)$ is
nonincreasing and right-continuous on $[0,1)$.

If in addition Assumption~\ref{ass:IF_mean_driven} holds, then all spike times are finite
($s^k(x_0)<\infty$ for every $k$), and for any $\delta>0$ with $x_0+\delta<1$,
\begin{equation}\label{eq:IF_Lipschitz_spiketime}
|s^k(x_0+\delta)-s^k(x_0)|
\ \le\
\frac{1}{a_*}\,S_*^{\,k-1}\,|\delta|,
\qquad
S_*:=\frac{f(0)+\essinf_{t\ge 0}I(t)}{f(1)+\essinf_{t\ge 0}I(t)}\ \ge 1.
\end{equation}
\end{theorem}

\begin{proof}
\emph{Step 1: monotonicity.}
Fix $\delta>0$ with $x_0+\delta<1$, and let $x^\delta(t)$ denote the solution from $x_0+\delta$,
$x(t)$ the solution from $x_0$, under the same input $I(t)$.
By \eqref{eq:IF_order}, $x^\delta(t)\ge x(t)$ until one hits threshold, so $s^1(x_0+\delta)\le s^1(x_0)$.
Restarting at reset and iterating yields $s^k(x_0+\delta)\le s^k(x_0)$ for all $k$.

\emph{Step 2: right continuity.}
Monotonicity of $s^k(\cdot)$ implies existence of one-sided limits; right continuity follows from
standard monotone-function properties together with the fact that threshold is hit from below and
solutions depend continuously on initial data for ODEs between events.

\emph{Step 3: Lipschitz under mean-driven transversality.}
Under \eqref{eq:IF_a_star}, for any subthreshold state $x(t)\in[0,1]$ we have
$\dot x(t)=f(x(t))+I(t)\ge f(1)+\essinf I \ge a_*$.
Hence the first spike time satisfies the crude bound
\[
|s^1(x_0+\delta)-s^1(x_0)|\le \frac{|\delta|}{a_*}.
\]
To propagate to later spikes we use the \emph{saltation factor} for 1D IF.
At a spike time $s=s^k(x_0)$, a perturbation in initial condition produces a spike-time shift
$\Delta s$ that is proportional to the pre-spike perturbation divided by the threshold velocity
$f(1)+I(s)$.
Immediately after reset, the state is $0$, but the post-reset state at a common clock time differs
because one trajectory has advanced for an extra $\Delta s$ under velocity $f(0)+I(s)$.
Thus the magnitude of the post-spike perturbation is multiplied by
\[
S(I(s)):=\frac{f(0)+I(s)}{f(1)+I(s)}.
\]
Since $f'\le 0$, $f(0)\ge f(1)$ and $S(\cdot)\ge 1$.
Under \eqref{eq:IF_a_star} and $I(s)\ge \essinf I$, we have
$S(I(s))\le S_*$ with $S_*$ defined in \eqref{eq:IF_Lipschitz_spiketime}.
Iterating over $k-1$ spikes gives the stated Lipschitz bound.
\end{proof}

\begin{remark}[Geometric sensitivity vs.\ Lyapunov growth]\label{rem:IF_geom_vs_lambda}
The factor $S_*^{k-1}$ in \eqref{eq:IF_Lipschitz_spiketime} is the finite-time counterpart of a
Lyapunov exponent: roughly, if the spike count up to time $T$ satisfies $N(T)\approx rT$, then
$S_*^{N(T)}\approx e^{r(\log S_*)T}$.
The next subsection makes this precise in a stationary stochastic setting.
\end{remark}

\subsubsection{Interlacing and spike-count mismatch is at most one}
A crucial structural consequence of \eqref{eq:IF_order}--\eqref{eq:IF_contraction_between_spikes} is
that two solutions driven by the \emph{same} input cannot drift arbitrarily far apart in spike count.

\begin{lemma}[Interlacing of spike times]\label{lem:IF_interlacing}
Assume \eqref{eq:IF_f_ass}. Consider two solutions $x_A,x_B$ of \eqref{eq:IF_det}--\eqref{eq:IF_reset}
with the same input $I(t)$ and initial data $x_A(0)\ge x_B(0)$.
Then their spike times interlace:
\[
s_A^1\ \le\ s_B^1\ \le\ s_A^2\ \le\ s_B^2\ \le\ \cdots,
\]
with the understanding that if one side has fewer than $k$ spikes then the corresponding time is $\infty$.
In particular, for every $T\ge 0$,
\begin{equation}\label{eq:IF_spikecount_diff}
|N_A(T)-N_B(T)|\le 1.
\end{equation}
\end{lemma}

\begin{proof}
We sketch the standard induction argument.
Since $x_A(0)\ge x_B(0)$, order preservation implies the first threshold hit of $x_A$ occurs no later
than that of $x_B$, hence $s_A^1\le s_B^1$.
At time $s_A^1$, $x_A$ resets to $0$, while $x_B(s_A^1)\in[0,1)$, so now $x_B(s_A^1)\ge x_A(s_A^1)=0$.
Applying order preservation from $s_A^1$ onward yields $s_B^1\le s_A^2$, and repeating gives the interlacing chain.
The spike-count bound \eqref{eq:IF_spikecount_diff} is immediate from interlacing.
\end{proof}

\begin{remark}[What interlacing does \emph{not} prevent]\label{rem:IF_permanent_loss}
Interlacing controls spike \emph{count} mismatch, but it does not guarantee that a trajectory that
missed a spike will eventually ``catch up'' in a way that makes trajectories close pointwise.
For example, if after one trajectory spikes the input becomes subthreshold, the lagging trajectory
may never spike (a permanent loss, producing $O(1)$ separation).
This phenomenon is relevant when interpreting the strong bounds with mismatch decomposition: in some regimes a missed
spike truly can be irrecoverable.
\end{remark}

\subsubsection{Spike-count bounds from input integrals}
The deterministic notes also provide useful spike-count bounds, which become rate bounds once inputs
are bounded.

\begin{proposition}[Spike-count upper bound]\label{prop:IF_spikecount_upper}
Assume \eqref{eq:IF_f_ass} and $I\in L^1([0,T])$.
Then any solution satisfies
\begin{equation}\label{eq:IF_spikecount_upper}
N(T;x_0)\ \le\ x_0 + \int_0^T (I(s)+f(0))_+\,ds.
\end{equation}
In particular, if $I(s)\le I_{\max}$ a.e.\ on $[0,T]$, then
\begin{equation}\label{eq:IF_spikecount_linear}
N(T;x_0)\ \le\ (I_{\max}+f(0))_+\,T + x_0.
\end{equation}
\end{proposition}

\begin{proof}
Between spikes, $x(t)$ increases at speed $\dot x=I(t)+f(x(t))\le I(t)+f(0)$ since $f$ is nonincreasing.
Each spike requires an increment of $1$ in $x$.
Summing increments over interspike intervals yields the integral bound \eqref{eq:IF_spikecount_upper}.
\end{proof}

\begin{remark}[Relevance for our network constants]\label{rem:IF_rate_to_Lambda}
Bounds like \eqref{eq:IF_spikecount_linear} provide explicit deterministic rate controls under bounded
inputs. In our strong/weak network bounds, such rate controls enter via spike-load parameters and
stepwise rate bounds.
\end{remark}

\subsection{A general 1D IF Lyapunov exponent theorem}\label{subsec:IF_general_lambda}

We now state a general Lyapunov-exponent theorem for one-dimensional IF models with an \emph{autonomous}
(stationary ergodic) input process.
The theorem isolates (i) continuous-time contraction/expansion between spikes and (ii) discrete
spike-induced expansion via a saltation factor.
It subsumes the LIF formula used later.

\subsubsection{Stochastic IF with autonomous input}
Let $(\Omega,\mathcal{F},\Pp)$ carry a stationary ergodic process $I(t)$ with càdlàg paths.
Consider the IF dynamics
\begin{equation}\label{eq:IF_general}
\dot x(t)=f(x(t))+I(t),\qquad x(t)<1,
\qquad
x(t^-)=1 \Rightarrow x(t^+)=0,
\end{equation}
with $f\in C^1([0,1])$.
Assume spikes are \emph{transversal} almost surely: whenever $x(t^-)=1$ at a spike time $t=s^k$,
\begin{equation}\label{eq:IF_transversal}
f(1)+I(s^{k-})>0.
\end{equation}

Let $\{s^k\}_{k\ge 1}$ be the spike times and $N(t)=\sum_k \1_{\{s^k\le t\}}$.
Define the \emph{saltation factor} (ratio of post-/pre-spike infinitesimal perturbations)
\begin{equation}\label{eq:IF_saltation_general}
S(I):=\frac{f(0)+I}{f(1)+I},
\qquad\text{defined for }f(1)+I>0.
\end{equation}
This is the 1D specialization of the saltation matrix for event-driven/hybrid flows; for spiking
neurons it appears explicitly in \cite{Brette2004}.

\subsubsection{Invariant measure and boundary flux}
Assume that the joint process $(x(t),I(t))$ on $\{x<1\}$ with reinjection at $x=0$ is positive
recurrent and admits a stationary ergodic distribution $\pi_\infty$ on $\mathsf{D}:=[0,1)\times\R$.
Assume moreover that the stationary law admits a density $\rho_\infty(x,I)$ that is bounded in a
neighborhood of $(1,I)$ for $I$ near the tangency set $\{f(1)+I=0\}$.
In the Markov case (e.g.\ $I$ an OU diffusion), the corresponding stationary Fokker--Planck equation
implies a stationary \emph{outflow boundary flux} through $x=1$ given by
\begin{equation}\label{eq:IF_flux_general}
J_\infty(I) = (f(1)+I)\,\rho_\infty(1,I)\,\1_{\{f(1)+I>0\}}.
\end{equation}
More generally (beyond Markov), $J_\infty$ can be understood as the stationary \emph{Palm} measure of
$I$ at spike times; for our purposes the Markov/flux setting is sufficient.

\begin{theorem}[General 1D IF Lyapunov exponent via stationary bulk measure and boundary flux]\label{thm:lambda_IF_general}
Assume:
\begin{enumerate}[label=(\roman*),leftmargin=2.2em]
\item $f\in C^1([0,1])$ and spikes are transversal a.s.\ in the sense of \eqref{eq:IF_transversal};
\item $(x(t),I(t))$ is stationary ergodic with invariant density $\rho_\infty(x,I)$ on $[0,1)\times\R$;
\item the stationary boundary flux measure exists and is given by \eqref{eq:IF_flux_general};
\item (integrability) $\rho_\infty(1,I)$ is locally bounded near the tangency set $f(1)+I=0$.
\end{enumerate}
Then the top Lyapunov exponent for infinitesimal perturbations in the $x$-direction under common-input
coupling equals
\begin{equation}\label{eq:lambda_IF_general}
\lambda_x
=
\int_{[0,1)\times\R} f'(x)\,\rho_\infty(x,I)\,dx\,dI
\;+\;
\int_{\R} J_\infty(I)\,\log S(I)\,dI,
\qquad
S(I)=\frac{f(0)+I}{f(1)+I}.
\end{equation}
Moreover, the boundary integral is finite despite the logarithmic singularity at $f(1)+I\downarrow 0$.
\end{theorem}

\begin{proof}
\emph{Step 1: multiplicative decomposition of the tangent dynamics.}
Consider the derivative $D(t):=\partial_{x_0}x(t)$ along a fixed input path.
Between spikes, $D$ solves $\dot D(t)=f'(x(t))D(t)$, hence
$D(t)=\exp(\int_{s^k}^{t} f'(x(u))\,du)\,D(s^{k+})$ on each interspike interval.
At a spike time $s=s^k$, the perturbation affects the spike time itself; the standard saltation
calculation (or \cite{Brette2004}) yields
\[
D(s^{+}) = S(I(s^-))\,D(s^{-}),\qquad S(I)=\frac{f(0)+I}{f(1)+I}.
\]
Combining yields, for any $T>0$,
\begin{equation}\label{eq:lambda_decomp_pathwise}
\log |D(T^+)|
=
\log|D(0^+)|
+
\int_0^T f'(x(t))\,dt
+
\sum_{k:\,s^k\le T}\log S(I(s^{k-})).
\end{equation}

\emph{Step 2: ergodic limits.}
Divide \eqref{eq:lambda_decomp_pathwise} by $T$ and let $T\to\infty$.
Stationary ergodicity gives
\[
\frac{1}{T}\int_0^T f'(x(t))\,dt \to \int f'(x)\,\rho_\infty(x,I)\,dx\,dI
\quad\text{a.s.}
\]
For the spike sum, define the random measure
$\mu_T(dI):=\frac{1}{T}\sum_{k:s^k\le T}\delta_{I(s^{k-})}(dI)$.
Under stationarity/ergodicity, $\mu_T$ converges (in the Cesàro/Palm sense) to the stationary spike
flux measure $J_\infty(I)\,dI$ characterized by \eqref{eq:IF_flux_general}, i.e.
\[
\frac{1}{T}\sum_{k:s^k\le T}\varphi(I(s^{k-}))\to \int \varphi(I)\,J_\infty(I)\,dI
\quad\text{for bounded continuous }\varphi.
\]
Applying this with $\varphi=\log S$ (justified by the integrability step below) yields the second term
in \eqref{eq:lambda_IF_general}, and the initial factor $\log|D(0^+)|/T\to 0$.

\emph{Step 3: finiteness of the logarithmic singularity.}
Let $a:=f(1)+I$ and consider the near-tangency regime $a\downarrow 0^+$.
Then $\log S(I)=\log\frac{f(0)+I}{f(1)+I}=\log\frac{f(0)-f(1)+a}{a}\sim \log\frac{1}{a}$, while the
flux density has a factor $a$:
$J_\infty(I)=(f(1)+I)\rho_\infty(1,I)\1_{\{a>0\}}=a\,\rho_\infty(1,f(1)+a)$.
If $\rho_\infty(1,\cdot)$ is locally bounded, then near $a=0$ the integrand behaves like
$a\log(1/a)$, which is integrable:
$\int_0^\alpha a\log(1/a)\,da<\infty$.
Thus the boundary integral in \eqref{eq:lambda_IF_general} is finite.
\end{proof}

\begin{remark}[A useful crude upper bound]\label{rem:lambda_general_upper}
Using $\log(1+x)\le x$ one gets
\[
0\le (f(1)+I)\log S(I)
=(f(1)+I)\log\!\Bigl(1+\frac{f(0)-f(1)}{f(1)+I}\Bigr)
\le f(0)-f(1).
\]
Hence
\[
\int J_\infty(I)\log S(I)\,dI
\le (f(0)-f(1))\int_{\{f(1)+I>0\}}\rho_\infty(1,I)\,dI,
\]
showing finiteness under minimal boundary-density integrability assumptions.
\end{remark}

\subsection{Single current-based LIF neuron: explicit $\lambda$ and bounds}\label{subsec:lambda_single_LIF}

We specialize Theorem~\ref{thm:lambda_IF_general} to the current-based LIF neuron used throughout this paper.
Recall the single-neuron dynamics \eqref{eq:dv}--\eqref{eq:dI} with reset \eqref{eq:reset}.
For the \emph{voltage} equation, conditional on the current path $I(t)$ the voltage is a 1D IF
equation with drift $f(v)=-(v-\Vr)/\tauv$ and input $I(t)$.
Thus $f'(v)=-1/\tauv$ and
\[
f(\Vr)=0,\qquad f(\Vth)= -\frac{\Vth-\Vr}{\tauv}=-\Ith,
\qquad \Ith:=\frac{\Vth-\Vr}{\tauv}.
\]
The corresponding saltation factor at a spike with pre-spike current $I(s^-)$ is
\begin{equation}\label{eq:LIF_saltation}
S(I)=\frac{f(\Vr)+I}{f(\Vth)+I}=\frac{I}{I-\Ith},
\qquad I>\Ith,
\end{equation}
in agreement with \cite{Brette2004}.

Assume the stationary subthreshold density $\rho_\infty(v,I)$ exists on $\{v<\Vth\}$ and is bounded
near $(\Vth,\Ith)$.
Then the stationary boundary flux is
\[
J_\infty(I)=(I-\Ith)\rho_\infty(\Vth,I)\1_{\{I>\Ith\}}.
\]

\begin{theorem}[Single-neuron LIF Lyapunov exponent (voltage direction)]\label{thm:lambda_single_LIF}
Under stationarity/ergodicity and boundedness of $\rho_\infty(\Vth,I)$ near $I=\Ith$, the voltage
Lyapunov exponent under common-noise coupling equals
\begin{equation}\label{eq:lambda_single_LIF}
\lambda_v
=
-\frac{1}{\tauv}
+
\int_{I>\Ith} (I-\Ith)\,\rho_\infty(\Vth,I)\,
\log\!\Bigl(\frac{I}{I-\Ith}\Bigr)\,dI
=
-\frac{1}{\tauv}+\int_{\R} J_\infty(I)\,\log S(I)\,dI,
\end{equation}
and the integral is finite despite the logarithmic singularity of $\log\frac{I}{I-\Ith}$ at $I\downarrow \Ith$.
\end{theorem}

\begin{proof}
Apply Theorem~\ref{thm:lambda_IF_general} with $f'(v)=-1/\tauv$ and $S(I)$ given by \eqref{eq:LIF_saltation}.
Integrability follows as in Theorem~\ref{thm:lambda_IF_general} because the flux density has the
factor $(I-\Ith)$ and $\rho_\infty(\Vth,I)$ is locally bounded near $I=\Ith$.
\end{proof}

\begin{remark}[Upper bounds in terms of boundary density]\label{rem:lambda_single_bounds}
Using $\log(1+x)\le x$ with $x=\Ith/(I-\Ith)$,
\[
0\le (I-\Ith)\log\!\Bigl(\frac{I}{I-\Ith}\Bigr)
=(I-\Ith)\log\!\Bigl(1+\frac{\Ith}{I-\Ith}\Bigr)\le \Ith.
\]
Thus
\begin{equation}\label{eq:lambda_simple_upper}
\lambda_v \le -\frac{1}{\tauv}+\Ith\int_{I>\Ith}\rho_\infty(\Vth,I)\,dI.
\end{equation}
A sharper bound splits $I\in(\Ith,\Ith+\alpha)$ and $I\ge \Ith+\alpha$:
on $(\Ith,\Ith+\alpha)$ use $\rho_\infty(\Vth,I)\le \rho_{\max}$ (Section~\ref{subsec:OU_proxy}) to get an
explicit $O(\alpha^2\log(1/\alpha))$ contribution; on $[\Ith+\alpha,\infty)$ the logarithm is uniformly bounded.
\end{remark}

\begin{remark}[Connection to finite-time spike-time sensitivity (warm-up theorem)]\label{rem:warmup_to_LIF}
Assumption~\ref{ass:IF_mean_driven} corresponds for LIF (in normalized voltage coordinates) to a
uniform transversality margin $I(t)-\Ith\ge \alpha>0$.
Then Theorem~\ref{thm:IF_spiketime} yields geometric spike-time sensitivity with factor
$S_*=\frac{\Ith+\alpha}{\alpha}$, matching the saltation factor bound in \eqref{eq:LIF_saltation}.
\end{remark}

\subsection{Feedforward networks: upstream exponents force downstream divergence}\label{subsec:lambda_feedforward_new}

We now incorporate the crucial feedforward point:
\emph{even if a downstream neuron is contractive for a frozen presynaptic spike train,}
a \emph{positive} upstream Lyapunov exponent produces exponentially diverging spike times, which
generate exponentially diverging synaptic currents downstream.

\subsubsection{A forcing estimate: spike-time divergence $\Rightarrow$ current divergence}
Consider one presynaptic neuron (layer $1$) and one postsynaptic neuron (layer $2$) coupled by a
weight $w$ with exponential synapse kernel $K(t;s)=e^{-(t-s)/\tauc}\1_{\{t\ge s\}}$.
Let $\{s_1^k\}$ and $\{\widetilde s_1^k\}$ be the spike times of the two copies of neuron $1$
(with matched indices), and define $\varepsilon_1^k:=\widetilde s_1^k-s_1^k$.

\begin{lemma}[Upstream spike-time divergence forces downstream current divergence]\label{lem:forced_current_by_STE}
Assume spike indices of neuron $1$ remain matched up to time $T$.
Let $I_2^{\mathrm{ff}}$ and $\widetilde I_2^{\mathrm{ff}}$ be the \emph{feedforward} parts of the postsynaptic
currents driven by $N_1$ and $\widetilde N_1$.
Then
\begin{equation}\label{eq:forced_current_L1}
\int_0^T \big|I_2^{\mathrm{ff}}(t)-\widetilde I_2^{\mathrm{ff}}(t)\big|\,dt
\ \le\ 2|w|\sum_{k:\,s_1^k\le T}|\varepsilon_1^k|.
\end{equation}
Consequently, by Lemma~\ref{lem:v_from_I}, the induced voltage difference satisfies
\begin{equation}\label{eq:forced_voltage_L1}
\sup_{t\le T}|v_2(t)-\widetilde v_2(t)|
\ \le\ \int_0^T |I_2(t)-\widetilde I_2(t)|\,dt
\ \lesssim\ 2|w|\sum_{k:\,s_1^k\le T}|\varepsilon_1^k| \ +\ \text{(local terms)}.
\end{equation}
\end{lemma}

\begin{proof}
This is the single-synapse specialization of the misalignment-window estimate:
apply Lemma~\ref{lem:kernel_shift} to each shifted presynaptic spike and sum.
Then use Lemma~\ref{lem:v_from_I} to convert an $L^1_t$ current difference to a voltage sup-norm bound.
\end{proof}

\begin{remark}[Exponential upstream STE implies exponential downstream forcing]\label{rem:exp_forcing}
If $|\varepsilon_1^k|$ grows like $e^{\lambda_1 t}$ along spike index $k$ (e.g.\ $\max_{s_1^k\le T}|\varepsilon_1^k|\asymp e^{\lambda_1 T}$)
and the spike count grows at most linearly $N_1(T)\lesssim T$, then the RHS of \eqref{eq:forced_current_L1} is
$\lesssim T e^{\lambda_1 T}$.
Thus downstream differences cannot decay faster than the upstream exponent, up to polynomial factors.
\end{remark}

\subsubsection{A max-rule for layerwise Lyapunov exponents}
We now formulate a clean propagation statement in the feedforward setting.
For each layer $\ell$, define the \emph{diagonal} (frozen-input) exponent $\lambda_\ell^{\mathrm{diag}}$:
this is the top Lyapunov exponent of layer $\ell$ under common-noise coupling when the presynaptic
spike trains from layers $<\ell$ are treated as \emph{fixed external inputs shared by both copies}.
In particular, for a single neuron in layer $\ell$ with frozen input, $\lambda_\ell^{\mathrm{diag}}$
is given by the single-neuron flux/saltation formula (Theorem~\ref{thm:lambda_single_LIF}), with the
stationary boundary density induced by that frozen input.

Let $\lambda_\ell$ be the actual layer-$\ell$ exponent within the full coupled feedforward network
(so presynaptic spikes may differ).

\begin{theorem}[Conditional max rule for feedforward Lyapunov exponents]\label{thm:lambda_feedforward_max}
Assume a stationary regime in which:
\begin{enumerate}[label=(\roman*),leftmargin=2.2em]
\item all layerwise diagonal exponents $\lambda_\ell^{\mathrm{diag}}$ are well defined;
\item spike times remain transversal and spike indices can be matched between the two coupled copies;
\item each layer $\ell\ge 2$ receives at least one nonzero synaptic weight from layer $\ell-1$;
\item for every $\ell\ge 2$ and every $\varepsilon>0$, there exists
$C_{\ell,\varepsilon}>0$ such that all perturbations satisfy the upper forcing estimate
\begin{equation}\label{eq:ff_upper_forcing}
\|\delta X_\ell(t)\|
\le
C_{\ell,\varepsilon}\,e^{(\lambda_\ell^{\mathrm{diag}}+\varepsilon)t}\|\delta X_\ell(0)\|
+
C_{\ell,\varepsilon}\int_0^t
e^{(\lambda_\ell^{\mathrm{diag}}+\varepsilon)(t-s)}
\|\delta X_{\ell-1}(s)\|\,ds;
\end{equation}
\item for every $\ell\ge 2$, there exist constants $c_\ell>0$ and $q_\ell\ge 0$ such that
every perturbation with $\delta X_\ell(0)=0$ satisfies the lower forcing estimate
\begin{equation}\label{eq:ff_lower_forcing}
\sup_{0\le s\le t}\|\delta X_\ell(s)\|
\ge
c_\ell(1+t)^{-q_\ell}
\sup_{0\le s\le t}\|\delta X_{\ell-1}(s)\|,
\qquad t\ge 0.
\end{equation}
\end{enumerate}
Then for $\ell=2,\dots,L$,
\begin{equation}\label{eq:lambda_exact_max_rule_new}
\lambda_\ell = \max\{\lambda_\ell^{\mathrm{diag}},\lambda_{\ell-1}\}.
\end{equation}
In particular, if $\lambda_1>0$ then $\lambda_\ell\ge \lambda_1$ for all $\ell$.
\end{theorem}

\begin{proof}
Fix $\ell\ge2$.

\emph{Lower bound by the diagonal exponent.}
Choose initial perturbations supported only in layer $\ell$.
Then layer $\ell-1$ does not force layer $\ell$, so the growth rate is exactly the diagonal one.
Hence
\[
\lambda_\ell\ge \lambda_\ell^{\mathrm{diag}}.
\]

\emph{Lower bound by the upstream exponent.}
Now choose initial perturbations with $\delta X_\ell(0)=0$ and nontrivial upstream perturbation.
By \eqref{eq:ff_lower_forcing},
\[
\sup_{0\le s\le t}\|\delta X_\ell(s)\|
\ge
c_\ell(1+t)^{-q_\ell}
\sup_{0\le s\le t}\|\delta X_{\ell-1}(s)\|.
\]
Taking $\frac1t\log$ and letting $t\to\infty$ gives
\[
\lambda_\ell\ge \lambda_{\ell-1},
\]
since the polynomial prefactor $(1+t)^{-q_\ell}$ does not affect the exponential rate.

Combining the two lower bounds yields
\[
\lambda_\ell\ge \max\{\lambda_\ell^{\mathrm{diag}},\lambda_{\ell-1}\}.
\]

\emph{Upper bound.}
Fix $\varepsilon>0$.
By the definition of $\lambda_{\ell-1}$, there exists $C_\varepsilon'>0$ such that
\[
\|\delta X_{\ell-1}(s)\|\le C_\varepsilon' e^{(\lambda_{\ell-1}+\varepsilon)s},
\qquad s\ge0.
\]
Substituting this into \eqref{eq:ff_upper_forcing} yields
\begin{align*}
\|\delta X_\ell(t)\|
&\le
C_{\ell,\varepsilon}e^{(\lambda_\ell^{\mathrm{diag}}+\varepsilon)t}\|\delta X_\ell(0)\|
+
C_{\ell,\varepsilon}C_\varepsilon'
\int_0^t
e^{(\lambda_\ell^{\mathrm{diag}}+\varepsilon)(t-s)}
e^{(\lambda_{\ell-1}+\varepsilon)s}\,ds \\
&\le
\widetilde C_{\ell,\varepsilon}(1+t)\,
e^{(\max\{\lambda_\ell^{\mathrm{diag}},\lambda_{\ell-1}\}+2\varepsilon)t}.
\end{align*}
Therefore
\[
\lambda_\ell\le \max\{\lambda_\ell^{\mathrm{diag}},\lambda_{\ell-1}\}+2\varepsilon.
\]
Since $\varepsilon>0$ is arbitrary,
\[
\lambda_\ell\le \max\{\lambda_\ell^{\mathrm{diag}},\lambda_{\ell-1}\}.
\]
Together with the lower bound, this proves \eqref{eq:lambda_exact_max_rule_new}.
\end{proof}

\begin{remark}[Where the warm-up interlacing can help]\label{rem:warmup_interlacing_help}
For 1D IF driven by a \emph{common} input and with $f'\le 0$, Lemma~\ref{lem:IF_interlacing} implies
that spike-count mismatch is uniformly bounded by $1$.
This supports the matched-indices assumption above for single neurons and in regimes where effective
order-preservation holds.
\end{remark}

\subsection{Implications for strong error accumulation in time}\label{subsec:lambda_vs_strong_new}

On the good event $G_L(h,T;\alpha)$ in the strong analysis (Section~\ref{sec:strong}), spike indices are matched and crossings are
transversal, so error propagation is governed by the local sensitivity of spike times and state variables.
In this regime, the hybrid/common-noise Lyapunov exponent $\lambda_\ell$ derived above is the natural stability diagnostic:
if $\lambda_\ell<0$, small perturbations contract on average and the good-set strong error will shrink on time, whereas the current discussion in Section~\ref{sec:strong} treat strong error as accumulating on time. This will affect the strong error estimate when $T\to\infty$ (see more in Discussion).

When $\lambda_\ell<0$, many trajectories that temporarily enter a mismatch configuration are expected to \emph{recover}.
Intuitively, contraction reduces the probability that a one-step mismatch triggers a long cascade of subsequent mismatches:
after a transient discrepancy, the common-noise flow pulls the two trajectories back together, making them ``good again''.
In 1D integrate-and-fire with a common input and $f'\le0$, interlacing (Lemma~\ref{lem:IF_interlacing}) already implies that spike-count
mismatch is at most one; in such order-preserving settings, recovery after a transient mismatch is plausible unless the dynamics enters a
subthreshold trapping regime (Remark~\ref{rem:IF_permanent_loss}).
For networks, the same idea suggests that negative Lyapunov behavior should reduce the expected \emph{time spent} in mismatch episodes, and
hence reduce the contribution of bad trajectories in long-time strong error.


\section{Conditional recurrent extensions and mechanisms: strong error}\label{sec:recurrent_strong}

\subsection{Intuition: recurrence creates an effective depth through feedback loops}
A recurrent spiking network can be unfolded in time. For a finite horizon, it behaves as a
feedforward system until the influence of a spike traverses a directed synaptic loop and returns to
a neuron already affected by that spike.
If the shortest feedback loop has synaptic length $L^\ast$ (number of synapses on the loop), then a
single spike-time perturbation can be amplified over an \emph{effective depth} comparable to $L^\ast$
before closing the loop.
In regimes with spike cascades (a nontrivial fraction of neurons fire within a short time window),
the effective depth over $[0,T]$ can be much larger than the anatomical loop length, because many
events occur in causal chains.
This is precisely the mechanism by which strong-error bounds based on STE propagations can deteriorate rapidly
in recurrent networks.

\subsection{Model: recurrent current-based LIF network}
We consider $N_{\mathrm r}$ neurons indexed by $i\in\{1,\dots,N_{\mathrm r}\}$ with recurrent weight matrix
$W\in\R^{N_{\mathrm r}\times N_{\mathrm r}}$.
The continuous-time dynamics is
\begin{align}
dv_i(t) &=
\Bigl(-\tfrac{1}{\tauv}(v_i(t)-\Vr)+I_i(t)\Bigr)\,dt,\qquad v_i(t)<\Vth,
\\
dI_i(t) &=
-\tfrac{1}{\tauc}\bigl(I_i(t)-b_i(t)\bigr)\,dt
+\tfrac{\sigma_i}{\tauc}\,dB_i(t)
+\sum_{j=1}^{N_{\mathrm r}} W_{ij}\,dN_j(t),
\end{align}
with instantaneous reset $v_i((s_i^k)^+)=\Vr$ at each spike time $s_i^k$ defined by upward threshold
crossing. The threshold current is $\Ith=(\Vth-\Vr)/\tauv$ and the crossing speed is
$A_i^k=I_i((s_i^k)^-)-\Ith$.

The numerical method is Euler--Maruyama for $I$ and forward Euler for $v$, with grid-based spike
detection and reset (as in Section~\ref{sec:model}).

\subsection{Scenario 1: deterministic mean-driven recurrent hybrid system ($\sigma\equiv 0$)}
In this subsection we assume $\sigma_i\equiv 0$ for all $i$, so the network is a deterministic
hybrid system.

\subsubsection{Continuous variational flow}
Let $x=(v,I)\in\R^{2N_{\mathrm r}}$, where $v,I\in\R^{N_{\mathrm r}}$.
Between spikes, the Jacobian of the ODE is the constant block matrix
\[
A=
\begin{pmatrix}
-\tauv^{-1}I_{N_{\mathrm r}} & I_{N_{\mathrm r}}\\
0 & -\tauc^{-1}I_{N_{\mathrm r}}
\end{pmatrix},
\qquad
F(\Delta):=e^{A\Delta} = \begin{pmatrix}
e^{-\Delta/\tauv}I_{N_{\mathrm r}} & B(\Delta)\,I_{N_{\mathrm r}}\\
0 & e^{-\Delta/\tauc}I_{N_{\mathrm r}}
\end{pmatrix},
\]
where
\[
B(\Delta)=\int_0^\Delta e^{-(\Delta-s)/\tauv}e^{-s/\tauc}\,ds
=
\begin{cases}
\displaystyle \frac{\tauv\tauc}{\tauc-\tauv}\Big(e^{-\Delta/\tauc}-e^{-\Delta/\tauv}\Big), & \tauv\neq\tauc,
\\[2mm]
\displaystyle \Delta e^{-\Delta/\tauv}, & \tauv=\tauc.
\end{cases}
\]

\subsubsection{Saltation matrix at a spike event}
Let neuron $p$ spike at time $s$.
The event surface is $g_p(x)=v_p-\Vth=0$ with normal $n_p=e_{v_p}$.
Denote the pre-spike speed $A_p=I_p(s^-)-\Ith$.
The reset map sets $v_p\mapsto \Vr$ and adds the synaptic column jump $I\mapsto I+W_{:p}$.
The saltation matrix $S_p$ mapping $\delta x^-$ to $\delta x^+$ at fixed global time is
\begin{equation}\label{eq:saltation_recurrent}
S_p
=
\Big(I_{2N_{\mathrm r}}-e_{v_p}e_{v_p}^\top\Big)
+
\frac{1}{A_p}\,u_p\,e_{v_p}^\top,
\qquad
u_p=
\binom{W_{:p}+I_p^- e_p}{-\tauc^{-1}W_{:p}},
\end{equation}
where $I_p^-=I_p(s^-)$ and $e_p$ is the unit vector in the $v$-block at index $p$.
Equivalently, acting on perturbations $(\delta v^-,\delta I^-)$,
\begin{equation}\label{eq:saltation_action_recurrent}
\delta v^+ = \delta v^- - e_p\,\delta v_p^- + \frac{\delta v_p^-}{A_p}\Big(W_{:p}+I_p^- e_p\Big),
\qquad
\delta I^+ = \delta I^- - \frac{\delta v_p^-}{\tauc A_p}W_{:p}.
\end{equation}

\subsubsection{Hybrid fundamental matrix and a Lyapunov upper bound}
Let $0<s_1<\cdots<s_M\le T$ be the spike-event times (across all neurons) and let $p_m$ be the neuron
spiking at $s_m$. The hybrid fundamental matrix is
\begin{equation}\label{eq:Phi_hyb_product}
\Phi_{\mathrm{hyb}}(T,0)
=
F(T-s_M)\,S_{p_M}\,F(s_M-s_{M-1})\cdots S_{p_1}\,F(s_1).
\end{equation}

\begin{assumption}[Deterministic transversality]\label{ass:det_transv}
There exists $\alpha_\ast>0$ such that at every spike event, $A_p\ge \alpha_\ast$.
\end{assumption}

\begin{theorem}[Explicit hybrid fundamental matrix and an upper bound on $\lambda_{\mathrm{hyb}}$]\label{thm:lambda_hyb_upper}
Assume Assumption~\ref{ass:det_transv}.
Let $N_{\mathrm{tot}}(T)=M$ be the total number of spike events in $[0,T]$ and
$r_{\mathrm{tot}}(T)=N_{\mathrm{tot}}(T)/T$.
Let $I_{\max}(T)=\max_i\sup_{t\le T}|I_i(t)|$ and $\|W\|_{\max}=\max_{i,j}|W_{ij}|$.
Then the saltation matrices satisfy the uniform bound (in the induced $\|\cdot\|_\infty$ matrix norm)
\[
\|S_p\|_\infty \le \kappa_S(T),
\qquad
\kappa_S(T):=
\max\Big\{
\frac{I_{\max}(T)+\|W\|_{\max}}{\alpha_\ast},\
1+\frac{\|W\|_{\max}}{\alpha_\ast},\
1+\frac{\|W\|_{\max}}{\tauc\alpha_\ast}
\Big\}.
\]
Moreover, since the continuous variational flow is exponentially stable, there exists a Lyapunov norm
$\|\cdot\|_\ast$ such that $\|F(\Delta)\|_\ast\le e^{-\Delta/\tau_{\max}}$ with
$\tau_{\max}=\max\{\tauv,\tauc\}$.
In that norm,
\[
\|\Phi_{\mathrm{hyb}}(T,0)\|_\ast \le e^{-T/\tau_{\max}}\kappa_S(T)^{N_{\mathrm{tot}}(T)}.
\]
Consequently the top hybrid Lyapunov exponent satisfies
\[
\lambda_{\mathrm{hyb}}
:=
\limsup_{T\to\infty}\frac{1}{T}\log\|\Phi_{\mathrm{hyb}}(T,0)\|_\ast
\ \le\
-\frac{1}{\tau_{\max}} + \bar r_{\mathrm{tot}}\log \kappa_S,
\]
where $\bar r_{\mathrm{tot}}=\limsup_{T\to\infty}r_{\mathrm{tot}}(T)$ and
$\kappa_S=\limsup_{T\to\infty}\kappa_S(T)$.
\end{theorem}

\begin{remark}[Connection to cycle-gain amplification]
The saltation structure \eqref{eq:saltation_action_recurrent} shows that each spike of neuron $p$
injects its voltage perturbation $\delta v_p^-$ into downstream currents and voltages through the
synaptic column $W_{:p}$ with amplification proportional to $1/A_p$.
This is the mechanism underlying the cycle-gain criterion: if a directed cycle $C$ is repeatedly
realized by the dynamics and the product of per-edge gains exceeds $1$, then the spike-time
Lyapunov exponent is forced positive.
\end{remark}

\subsubsection{Mean-driven deterministic numerics: error growth controlled by $\lambda_{\mathrm{hyb}}$}
We now make explicit how the deterministic dynamical stability rate (hybrid Lyapunov exponent)
controls deterministic discretization error accumulation in a mean-driven regime.

\begin{assumption}[No multiple spikes per neuron per time step]\label{ass:no_multispike_step}
There exists $h_0>0$ such that for all $h\in(0,h_0)$ and for all neurons $i$, the exact trajectory
has at most one spike in any interval $(t_m,t_{m+1}]$ of length $h$ up to time $T$.
\end{assumption}

\begin{lemma}[Event-to-event deterministic defect]\label{lem:det_event_recursion}
Assume $\sigma\equiv 0$, Assumption~\ref{ass:det_transv}, and
Assumption~\ref{ass:no_multispike_step}.  Assume also that the exact trajectory remains in a
bounded set on $[0,T]$.
Then there exist $h_1>0$ and $C_{\mathrm{ev}}(T)>0$ such that for all $h\in(0,h_1)$:
\begin{enumerate}[label=(\roman*),leftmargin=2.2em]
\item the exact and numerical spike sequences can be matched in order up to time $T$;
\item if $0=s_0<s_1<\cdots<s_M\le T$ are the exact spike-event times and $p_m$ is the neuron
spiking at $s_m$, then the post-event errors
\[
e_m^+ := (x_m^h)^+ - x_m^+,
\qquad
x_m^+ := x(s_m^+),
\]
satisfy the recursion
\begin{equation}\label{eq:det_event_recursion}
e_m^+
=
S_{p_m}F(s_m-s_{m-1})e_{m-1}^+
+
\eta_m,
\qquad
\|\eta_m\|_\ast\le C_{\mathrm{ev}}(T)\,h .
\end{equation}
\end{enumerate}
\end{lemma}

\begin{proof}
Between two consecutive exact spike times, the exact dynamics is a linear inhomogeneous ODE with
bounded coefficients on a compact set.  Therefore the deterministic Euler scheme on the underlying
grid approximates the exact inter-event flow with uniform global error $O(h)$ on each interval
$[s_{m-1},s_m)$:
\[
\|x^h(s_m^-)-F(s_m-s_{m-1})x_{m-1}^+ - b_m\|_\ast \le C(T)\,h,
\]
where $b_m$ denotes the exact inhomogeneous contribution on that interval.

By Assumption~\ref{ass:det_transv}, each exact crossing is transversal with speed bounded below by
$\alpha_\ast$.  Together with Assumption~\ref{ass:no_multispike_step} and the deterministic $O(h)$
inter-event error above, this implies that for $h$ small enough the numerical grid-based detector
registers exactly one corresponding spike in the same step, and hence the event ordering is preserved.

At the spike time $s_m$, the reset map linearizes through the saltation matrix $S_{p_m}$.
Since the numerical event time differs from $s_m$ by at most $O(h)$, the reset mismatch contributes
an additional defect $r_m$ with $\|r_m\|_\ast\le C(T)\,h$.
Combining the inter-event Euler defect and the reset defect gives
\[
e_m^+
=
S_{p_m}F(s_m-s_{m-1})e_{m-1}^+
+
\eta_m,
\qquad
\|\eta_m\|_\ast\le C_{\mathrm{ev}}(T)\,h,
\]
after absorbing constants.
\end{proof}

\begin{theorem}[Deterministic mean-driven recurrent network: numerical error grows at rate $\lambda_{\mathrm{hyb}}$]\label{thm:det_error_lambda}
Assume $\sigma\equiv 0$, Assumption~\ref{ass:det_transv}, and Assumption~\ref{ass:no_multispike_step}.
Let $x(t)=(v(t),I(t))$ be the exact hybrid trajectory and let $x^h(t)$ be the time-driven Euler
trajectory with grid-based reset, driven by the same initial data $x^h(0)=x(0)$.
Assume the exact trajectory remains in a bounded set on $[0,T]$ (so all coefficients and event gains
are uniformly bounded).

Then there exists $h_1\in(0,h_0]$ such that for all $h\in(0,h_1)$:
\begin{enumerate}[label=(\roman*),leftmargin=2.2em]
\item the exact and numerical spike sequences have the same event ordering (no missed/extra spikes);
\item for any $\varepsilon>0$ there exists a constant $C_\varepsilon(T)$ such that the global state
error satisfies
\begin{equation}\label{eq:det_global_error_lambda}
\sup_{t_m\le T}\|x^h(t_m)-x(t_m)\|_\ast
\ \le\
C_\varepsilon(T)\,e^{(\lambda_{\mathrm{hyb}}+\varepsilon)T}\,h;
\end{equation}
\item in particular, the spike-time errors satisfy
\begin{equation}\label{eq:det_STE_lambda}
\max_{i}\max_{k:s_i^k\le T}\abs{\hat s_i^k-s_i^k}
\ \le\
\frac{1}{\alpha_\ast}\,C_\varepsilon(T)\,e^{(\lambda_{\mathrm{hyb}}+\varepsilon)T}\,h.
\end{equation}
\end{enumerate}
Consequently, if $\lambda_{\mathrm{hyb}}<0$ then the deterministic scheme is uniformly accurate in
time ($O(h)$ as $T\to\infty$), whereas if $\lambda_{\mathrm{hyb}}>0$ deterministic errors can grow at
most exponentially at rate $\lambda_{\mathrm{hyb}}$.
\end{theorem}

\begin{proof}
Lemma~\ref{lem:det_event_recursion} yields the event-to-event recursion
\eqref{eq:det_event_recursion}.  Unrolling it from $e_0^+=0$ gives
\[
e_m^+
=
\sum_{r=1}^m \Phi_{\mathrm{hyb}}(s_m,s_r)\,\eta_r,
\qquad m=1,\dots,M,
\]
where $\Phi_{\mathrm{hyb}}(t,s)$ is the hybrid fundamental matrix from time $s$ to time $t$.

Fix $\varepsilon>0$.  By the definition of the top hybrid Lyapunov exponent, there exists
$C_\varepsilon(T)>0$ such that
\[
\|\Phi_{\mathrm{hyb}}(t,s)\|_\ast
\le
C_\varepsilon(T)e^{(\lambda_{\mathrm{hyb}}+\varepsilon)(t-s)},
\qquad 0\le s\le t\le T.
\]
Hence
\begin{align*}
\max_{1\le m\le M}\|e_m^+\|_\ast
&\le
C_{\mathrm{ev}}(T)h
\sum_{r=1}^M
\|\Phi_{\mathrm{hyb}}(T,s_r)\|_\ast \\
&\le
C_{\mathrm{ev}}(T)h\,C_\varepsilon(T)
\sum_{r=1}^M e^{(\lambda_{\mathrm{hyb}}+\varepsilon)(T-s_r)}
\le
\widetilde C_\varepsilon(T)e^{(\lambda_{\mathrm{hyb}}+\varepsilon)T}h.
\end{align*}
This proves \eqref{eq:det_global_error_lambda} at event times.

Between events, both the exact and numerical trajectories evolve under the same exponentially stable
linear flow, and the deterministic Euler defect on each inter-grid subinterval is again $O(h)$.
Therefore the same bound holds uniformly at grid times:
\[
\sup_{t_m\le T}\|x^h(t_m)-x(t_m)\|_\ast
\le
\widetilde C_\varepsilon(T)e^{(\lambda_{\mathrm{hyb}}+\varepsilon)T}h.
\]

Finally, the standard deterministic threshold-time sensitivity estimate under the uniform
transversality bound $A_p\ge \alpha_\ast$ gives
\[
\max_i\max_{k:s_i^k\le T}\abs{\hat s_i^k-s_i^k}
\le
\frac{2}{\alpha_\ast}
\sup_{t_m\le T}\|x^h(t_m)-x(t_m)\|_\ast,
\]
which yields \eqref{eq:det_STE_lambda}.
\end{proof}

\begin{remark}[Interpretation]
Theorem~\ref{thm:det_error_lambda} is the deterministic analog of the ``stability vs.\ long-time error''
discussion in Section~\ref{subsec:lambda_vs_strong_new}: in deterministic mean-driven recurrence, the
hybrid Lyapunov exponent directly controls numerical error accumulation.
This mechanism is fundamentally different from spike-mismatch-based bounds, where bad-event probability
terms cannot be canceled by contraction on good paths.
\end{remark}

\subsection{Scenario 2: noisy recurrent networks under rate and density bounds}
We now return to $\sigma_i>0$ and treat recurrent networks by the same spike-time + horizon mismatch framework
as in the feedforward case, but with loop amplification encoded via rate and graph bounds.

\subsubsection{Rate bounds and synchrony control}
Let $\Delta N_{\mathrm{tot}}(m)=\sum_{j=1}^{N_{\mathrm r}}(N_j(t_{m+1})-N_j(t_m))$ be the total number of exact
spikes in step $m$.
We assume the existence of finite constants $R_{\mathrm{net}}(T)$ and $R_{\mathrm{net},2}(T)$ such that
\begin{equation}\label{eq:rate_bounds_recurrent}
\E[\Delta N_{\mathrm{tot}}(m)]\le R_{\mathrm{net}}(T)\,h,
\qquad
\E[\Delta N_{\mathrm{tot}}(m)(\Delta N_{\mathrm{tot}}(m)-1)]\le R_{\mathrm{net},2}(T)\,h^2,
\end{equation}

\begin{remark}[On $R_{\mathrm{net}}(T)$ versus $R_{\mathrm{net},2}(T)$]\label{rem:Rnet_not_redundant}
The two bounds in \eqref{eq:rate_bounds_recurrent} encode different information and are not redundant.
The first-moment bound $R_{\mathrm{net}}(T)$ is implied, for instance, by deterministic firing-rate envelopes
$r_j(t)\le r_j^{\max}(T)$ via $R_{\mathrm{net}}(T)=\sum_j r_j^{\max}(T)$.
The second factorial-moment bound $R_{\mathrm{net},2}(T)$ controls within-step \emph{synchrony/burstiness}:
by Markov's inequality,
$\Pp(\Delta N_{\mathrm{tot}}(m)\ge 2)\le \tfrac12 \E[\Delta N_{\mathrm{tot}}(m)(\Delta N_{\mathrm{tot}}(m)-1)]$,
so $R_{\mathrm{net},2}(T)$ is exactly what limits the probability of multiple spikes in one step.
Such control cannot be inferred from $R_{\mathrm{net}}(T)$ alone, and it is precisely where synchronized regimes
deteriorate time-driven error constants.
\end{remark}

for all steps with $t_{m+1}\le T$.
Scenario R2 corresponds to moderate $R_{\mathrm{net}},R_{\mathrm{net},2}$ (asynchronous, weak synchrony),
while Scenario R4 has the same form but much larger constants (burstier synchrony).

\subsubsection{Boundary density constants and their estimates}
For each neuron $i$, define the boundary density constant (near tangency)
\[
\rho_{\max,i}(T,\alpha_0)
:=
\sup_{t\in[0,T]}\sup_{a\in[0,\alpha_0]}\rho_i(\Vth,\Ith+a,t),
\]
and the voltage marginal bound near threshold
\[
\rho^V_{\max,i}(T)
:=
\sup_{t\in[0,T]}\sup_{y\in[0,1]}\varrho_i(\Vth-y,t),
\]
where $\rho_i$ is the $(v_i,I_i)$ subthreshold density and $\varrho_i$ is the voltage marginal.

As in the feedforward analysis, $\rho_{\max,i}(T,\alpha_0)$ controls the contribution of slow threshold crossings inside the boundary-flux averaged matched strong term, while $\rho^V_{\max,i}(T)$ enters conservative strip-based density bounds (used later for weak error estimates).
Our sharp spike-count mismatch bound is instead formulated in terms of rate and burst constants; see Lemma~\ref{lem:mis_prob}.
In regimes with sufficiently strong independent Brownian noise in each current, hypoellipticity
suggests these constants are finite; however they can become large (and can blow up in the small-noise
limit) in bursty synchronized regimes.

\paragraph{A baseline explicit proxy and its limitation in recurrent networks.}
A convenient explicit proxy is obtained by \emph{freezing} the recurrent input into an exogenous drift $\mu_i$ and approximating
the current by a linear OU process (diffusion only in $I_i$):
\[
dI_i = -\tfrac{1}{\tauc}(I_i-\mu_i)\,dt + \tfrac{\sigma_i}{\tauc}\,dB_i,
\qquad
dv_i = \Bigl(-\tfrac{1}{\tauv}(v_i-\Vr)+I_i\Bigr)\,dt.
\]
Its stationary density yields the baseline proxy estimates
\begin{equation}\label{eq:rho_bounds_linear_surrogate}
\rho_{\max,i}(T,\alpha_0)\ \lesssim\
\frac{\sqrt{\tauc}(\tauc+\tauv)}{\pi\,\sigma_i^2\,\tauv^{3/2}},
\qquad
\rho^V_{\max,i}(T)\ \lesssim\
\frac{\sqrt{\tauc+\tauv}}{\sqrt{\pi}\,\sigma_i\,\tauv}.
\end{equation}
A sharper (rate- and weight-dependent) proxy replaces $\sigma_i$ by an effective diffusion $\sigma_{\mathrm{eff},i}$,
\[
\sigma_{\mathrm{eff},i}^2(t)\approx
\sigma_i^2+\kappa\,\tauc^2\sum_{j=1}^{N_{\mathrm r}} W_{ij}^2 r_j(t),
\]
which is the standard diffusion approximation of synaptic shot noise in asynchronous regimes.

\medskip
\begin{remark}[Important caveat: recurrent vs.\ feedforward.]
\label{rem:caveat_rec_ffwd}
The OU surrogate treats the input $\mu_i(t)$ (and the rate field $r_j(t)$ used in $\sigma_{\mathrm{eff},i}$) as \emph{given} and,
crucially, \emph{independent} of the intrinsic Brownian motion $B_i$.
This independence is natural in feedforward networks and can be a good approximation for large recurrent networks in weakly correlated
(asynchronous / mean-field) regimes, where the contribution of a single neuron's noise to its own net input is negligible.
In genuinely recurrent, strongly correlated regimes, however, $\mu_i(t)$ depends on the network spike trains and therefore inherits
correlations with $B_i$.
In particular, in tight E--I balance or near-synchronous cascade regimes, recurrent input can partially cancel intrinsic noise and produce
an \emph{effective} variance $\sigma_{\mathrm{eff},i}^2(t)$ that is small, leading to very slow threshold crossings and potentially
near-singular voltage marginals near $\Vth$.
For this reason, in the recurrent sections we treat $\rho_{\max,i}(T,\alpha_0)$ and $\rho^V_{\max,i}(T)$ primarily as assumptions (or as
empirically estimated constants), and we use \eqref{eq:rho_bounds_linear_surrogate} only as a baseline proxy when such decoupling is justified.
\end{remark}

\subsubsection{Boundary-flux matched impact and spike-count mismatch probability}
Define the recurrent matched event
\[
\mathcal G^{\mathrm{match}}_{\mathrm{rec}}(h,T)
:=
\bigcap_{i=1}^{N_{\mathrm r}}\{N_i^h(T)=N_i(T)\}.
\]
On this event, spike indices are paired up to time $T$ and we define the matched spike-time impact
\[
\mathcal M_i^{\mathrm{rec}}(T)
:=
\E\Bigg[
\sum_{k:\,s_i^k\le T}\psi\!\bigl(\abs{\varepsilon_i^k}\bigr)\,
\mathbf 1_{\mathcal G^{\mathrm{match}}_{\mathrm{rec}}(h,T)}
\Bigg].
\]

The noisy recurrent extension follows the same logic as Section~\ref{subsec:flux}: one first proves
a conditional one-spike impact bound at fixed crossing speed,
\[
\Gamma_i(a;h,T)
:=
C_i^{\mathrm{hit}}(T)\min\!\Bigl\{1,\frac{R_i(h,T)^2}{a^2}\Bigr\},
\qquad a>0,
\]
where the local scale $R_i(h,T)$ will be specified below, and then averages $\Gamma_i$ against the
threshold flux.  Accordingly, there exists $C_i^{\mathrm{dir}}(T)<\infty$ such that
\begin{equation}\label{eq:direct_flux_recurrent}
\mathcal M_i^{\mathrm{rec}}(T)
\le
\int_0^T\int_0^\infty
\Gamma_i(a;h,T)\,a\,\rho_i(\Vth,\Ith+a,t)\,da\,dt
\le
C_i^{\mathrm{dir}}(T)\,R_i(h,T)^2\bigl(1+\abs{\log R_i(h,T)}\bigr),
\end{equation}
where we have stated the generic fluctuation-driven case corresponding to a bounded boundary density
near $a=0$.  If one has stronger boundary depletion
$\rho_i(\Vth,\Ith+a,t)\lesssim a^{\gamma_i}$ with $\gamma_i>0$, the logarithm drops out exactly as in
Section~\ref{subsec:flux}.  Thus, in the noisy recurrent strong analysis, small crossing speeds are
handled inside the matched term through the boundary-flux integral rather than by declaring a
separate tangency bad set.

For mismatch, we use the \emph{horizon-based} spike-count mismatch event
\[
B_i^{\mathrm{mis}}(h,T):=\{N_i^h(T)\neq N_i(T)\},
\]
which treats false positives and false negatives symmetrically and ignores transient stepwise
discrepancies that catch up before the observation time.  Assume the local rate and burst controls
\eqref{eq:rate_local_bound}--\eqref{eq:pair_bound} and the sub-Gaussian spike-time tail
\eqref{eq:STE_subG} hold for neuron $i$.  Then Lemma~\ref{lem:mis_prob} yields, for any
$\delta\in(0,1]$,
\begin{equation}\label{eq:mis_count_bound_recurrent}
\Pp\big(B_i^{\mathrm{mis}}(h,T)\big)
\le
2\,r_i^{\max}(T)\,\delta
+
4\,R_i^{(2)}(T)\,\delta^2
+
C_i^{\mathrm{err}}(T)\exp\!\Bigl(-c_{\mathrm{sg}}\frac{\delta^2}{h}\Bigr),
\end{equation}
and, with the canonical choice $\delta=\delta_h:=\kappa\sqrt{h\log(1/h)}$,
\begin{equation}\label{eq:mis_count_rate_recurrent}
\Pp\big(B_i^{\mathrm{mis}}(h,T)\big)
\lesssim
r_i^{\max}(T)\sqrt{h\log(1/h)}
+
R_i^{(2)}(T)\,h\log(1/h)
+
C_i^{\mathrm{err}}(T)\,h^{c_0},
\qquad c_0:=c_{\mathrm{sg}}\kappa^2.
\end{equation}

\subsubsection{Matched-trajectory impact recursion and explicit loop-length truncation}
For each neuron $j$, define the matched spike-load bound
\[
\Lambda^{\mathrm{rec}}_j(T)
:=
\sup_{\omega\in \mathcal G^{\mathrm{match}}_{\mathrm{rec}}(h,T)}
\#\{k:\ s_j^k(\omega)\le T\},
\qquad
\Lambda^{\mathrm{rec}}_{\mathrm{tot}}(T):=\sum_{j=1}^{N_{\mathrm r}}\Lambda^{\mathrm{rec}}_j(T).
\]
(For example, $\Lambda^{\mathrm{rec}}_j(T)\le r_j^{\max}T$ on
$\mathcal G^{\mathrm{match}}_{\mathrm{rec}}(h,T)$ if one has deterministic rate bounds.)

The recurrent analogue of Lemmas~\ref{lem:I_from_STE_L1}--\ref{lem:v_from_I} gives the local-gap scale
\begin{equation}\label{eq:Rrec_def}
R_i(h,T)^2
:=
C_i^{\mathrm{loc}}(T)\,h
+
8\tauc^2
\sum_{j=1}^{N_{\mathrm r}}
|W_{ij}|^2\,\Lambda_j^{\mathrm{rec}}(T)\,
\mathcal M_j^{\mathrm{rec}}(T).
\end{equation}
Combining \eqref{eq:Rrec_def} with \eqref{eq:direct_flux_recurrent} yields a nonlinear
matched-trajectory recursion.  Since $R_i(h,T)^2\ge C_i^{\mathrm{loc}}(T)h$, for sufficiently small
$h$ we may bound
\[
1+\abs{\log R_i(h,T)}
\le
C_i^{\log}(T)\bigl(1+\abs{\log h}\bigr).
\]
Define
\[
b_i(T,h)
:=
C_i^{\mathrm{dir}}(T)\,C_i^{\log}(T)\,C_i^{\mathrm{loc}}(T)\,
h\bigl(1+\abs{\log h}\bigr),
\]
and the nonnegative gain matrix $H(T,h)\in\mathbb R^{N_{\mathrm r}\times N_{\mathrm r}}$ by
\begin{equation}\label{eq:H_rec_def}
H_{ij}(T,h)
:=
8\tauc^2\,
C_i^{\mathrm{dir}}(T)\,C_i^{\log}(T)\,
|W_{ij}|^2\,\Lambda_j^{\mathrm{rec}}(T)\,
\bigl(1+\abs{\log h}\bigr).
\end{equation}
Then the matched impact vector
$\mathcal M^{\mathrm{rec}}(T):=(\mathcal M_i^{\mathrm{rec}}(T))_{i=1}^{N_{\mathrm r}}$ satisfies
\begin{equation}\label{eq:recurrent_M_vector_ineq}
\mathcal M^{\mathrm{rec}}(T)
\le
b(T,h)+H(T,h)\,\mathcal M^{\mathrm{rec}}(T)
\qquad\text{componentwise.}
\end{equation}
If one has boundary depletion with exponent $\gamma_i>0$, the factor $1+\abs{\log h}$ may be
replaced by a bounded constant for the corresponding neuron.

The inequality \eqref{eq:recurrent_M_vector_ineq} is algebraic.  To obtain a finite-horizon bound
without imposing a contraction condition on $H(T,h)$, we again exploit that causality in time
bounds the maximum number of synaptic hops through which a local defect can propagate by time $T$.
This yields an explicit loop-length truncation.

\begin{definition}[Shortest directed cycle length]\label{def:short_cycle}
Let $\mathcal{G}=(\{1,\dots,N_{\mathrm r}\},E)$ be the directed synaptic graph where $(j\to i)\in E$
iff $W_{ij}\neq 0$.
Define $L^\ast\in\{2,3,\dots\}\cup\{\infty\}$ as the length of the shortest directed cycle in
$\mathcal{G}$ (take $L^\ast=\infty$ if $\mathcal{G}$ is acyclic).
\end{definition}

\begin{assumption}[Effective causal depth on the matched set]\label{ass:causal_depth}
Fix $\omega\in \mathcal G^{\mathrm{match}}_{\mathrm{rec}}(h,T)$.
A \emph{time-ordered influence chain} is a finite sequence
\[
(j_0,k_0)\to (j_1,k_1)\to\cdots\to (j_m,k_m)
\]
such that:
\begin{enumerate}[label=(\roman*),leftmargin=2.2em]
\item for each $r=1,\dots,m$, there is a directed synaptic edge
$j_{r-1}\to j_r$, i.e. $W_{j_r j_{r-1}}\neq 0$;
\item the spike times are strictly increasing and lie in the observation window,
\[
0<s_{j_0}^{k_0}(\omega)<s_{j_1}^{k_1}(\omega)<\cdots<s_{j_m}^{k_m}(\omega)\le T.
\]
\end{enumerate}
We assume there exists an integer $K_{\mathrm{eff}}(T)\ge 0$ such that every
time-ordered influence chain on $\mathcal G^{\mathrm{match}}_{\mathrm{rec}}(h,T)$ has length at most
$K_{\mathrm{eff}}(T)$.

Equivalently: on the matched set, any local spike-time defect can be transmitted through at
most $K_{\mathrm{eff}}(T)$ successive synaptic events before time $T$.
\end{assumption}

\begin{remark}[Crude bounds and interpretation of $K_{\mathrm{eff}}(T)$]
The quantity $K_{\mathrm{eff}}(T)$ is an effective causal depth in time.
A universal crude bound is
\[
K_{\mathrm{eff}}(T)\le \Lambda_{\mathrm{tot}}^{\mathrm{rec}}(T),
\]
because each transmission step consumes one spike in the time-ordered chain.
If one has a deterministic lower bound on the time needed for one effective transmission, say a
causal lag $\tau_{\mathrm{lag}}>0$, then
\[
K_{\mathrm{eff}}(T)\le \left\lfloor \frac{T}{\tau_{\mathrm{lag}}}\right\rfloor.
\]
In asynchronous regimes with moderate within-step rates and weak synchrony, it is often reasonable
to use the heuristic proxy
\[
K_{\mathrm{eff}}(T)\approx R_{\mathrm{net}}(T)\,T,
\]
but this proxy is not part of the formal assumption.
\end{remark}

\begin{theorem}[Explicit loop-truncated matched-impact bound]\label{thm:recurrent_goodset_explicit}
Assume Assumption~\ref{ass:causal_depth}.  Then on the matched set,
\begin{equation}\label{eq:Delta_explicit_trunc}
\mathcal M^{\mathrm{rec}}(T)
\le
\Big(\sum_{k=0}^{K_{\mathrm{eff}}(T)} H(T,h)^k\Big)\,b(T,h),
\qquad\text{componentwise.}
\end{equation}
Consequently, in the induced $\|\cdot\|_\infty$ norm,
\begin{equation}\label{eq:Delta_inf_bound_trunc}
\|\mathcal M^{\mathrm{rec}}(T)\|_\infty
\le
\Big\|\sum_{k=0}^{K_{\mathrm{eff}}(T)} H(T,h)^k\Big\|_\infty\,\|b(T,h)\|_\infty.
\end{equation}
Moreover, let $L^\ast$ be the shortest directed cycle length from Definition~\ref{def:short_cycle}.
If $L^\ast=\infty$ (acyclic synaptic graph, essentially still feedforward), then the truncated Neumann series can be bounded directly by
\begin{equation}\label{eq:acyclic_neumann_bound}
\Big\|\sum_{k=0}^{K_{\mathrm{eff}}(T)} H(T,h)^k\Big\|_\infty
\le
\sum_{k=0}^{K_{\mathrm{eff}}(T)} \|H(T,h)\|_\infty^k.
\end{equation}
If $L^\ast<\infty$, write $K_{\mathrm{eff}}(T)=M(T)L^\ast+R(T)$ with $R(T)\in\{0,\dots,L^\ast-1\}$.
Then we have the loop-block estimate
\begin{equation}\label{eq:loop_block_bound}
\Big\|\sum_{k=0}^{K_{\mathrm{eff}}(T)} H(T,h)^k\Big\|_\infty
\le
\Big(\sum_{r=0}^{L^\ast-1}\|H(T,h)\|_\infty^r\Big)\,
\sum_{m=0}^{M(T)} \|H(T,h)^{L^\ast}\|_\infty^m,
\end{equation}
and, in particular, if $\|H(T,h)^{L^\ast}\|_\infty<1$ (when $|\log h|$ is not very large), then
\begin{equation}\label{eq:loop_block_geometric}
\Big\|\sum_{k=0}^{K_{\mathrm{eff}}(T)} H(T,h)^k\Big\|_\infty
\le
\Big(\sum_{r=0}^{L^\ast-1}\|H(T,h)\|_\infty^r\Big)\,
\frac{1}{1-\|H(T,h)^{L^\ast}\|_\infty},
\end{equation}
uniformly in $K_{\mathrm{eff}}(T)$.
\end{theorem}

\begin{proof}
\emph{Step 1: truncation by causal depth.}
Iterating \eqref{eq:recurrent_M_vector_ineq} gives
\[
\mathcal M^{\mathrm{rec}}
\le
\sum_{k=0}^{K}H^k b + H^{K+1}\mathcal M^{\mathrm{rec}}.
\]
Under Assumption~\ref{ass:causal_depth}, the remainder term $H^{K+1}\mathcal M^{\mathrm{rec}}$ is
absent on the matched set once $K\ge K_{\mathrm{eff}}(T)$: each multiplication by $H$ corresponds
to taking one synaptic predecessor, so $(H^{K+1}\mathcal M^{\mathrm{rec}})_i$ expands into
contributions along directed, time-ordered influence chains of length $K+1$ ending at neuron $i$.
By definition of $K_{\mathrm{eff}}(T)$, such chains do not exist on
$\mathcal G^{\mathrm{match}}_{\mathrm{rec}}(h,T)$ for $K=K_{\mathrm{eff}}(T)$.  Therefore
\eqref{eq:Delta_explicit_trunc} holds.  Taking $\|\cdot\|_\infty$ gives
\eqref{eq:Delta_inf_bound_trunc}.

\emph{Step 2: series bounds (acyclic vs.\ cyclic).}
If $L^\ast=\infty$ (acyclic synaptic graph), then submultiplicativity gives
\[
\Big\|\sum_{k=0}^{K_{\mathrm{eff}}}H^k\Big\|_\infty
\le
\sum_{k=0}^{K_{\mathrm{eff}}}\|H\|_\infty^k,
\]
which is \eqref{eq:acyclic_neumann_bound} (with $H=H(T,h)$ and
$K_{\mathrm{eff}}=K_{\mathrm{eff}}(T)$).

If $L^\ast<\infty$, write $K_{\mathrm{eff}}=M L^\ast+R$ with $R\in\{0,\dots,L^\ast-1\}$ and
decompose the series into blocks:
\[
\sum_{k=0}^{K_{\mathrm{eff}}}H^k
=
\sum_{m=0}^{M-1}\sum_{r=0}^{L^\ast-1}H^{mL^\ast+r} + \sum_{r=0}^{R}H^{ML^\ast+r}
=
\sum_{m=0}^{M}H^{mL^\ast}\Big(\sum_{r=0}^{L^\ast-1}H^r\Big),
\]
where the last identity uses a conservative completion of the final block.  Taking
$\|\cdot\|_\infty$ and using submultiplicativity yields \eqref{eq:loop_block_bound}.
If $\|H^{L^\ast}\|_\infty<1$, the geometric series gives \eqref{eq:loop_block_geometric}.
\end{proof}

\subsubsection{Strong bound for matched noisy recurrent networks}
Let $X(T)$ be a bounded network output functional such that on matched spike histories,
\[
\abs{X^h(T)-X(T)}^2
\le
C_X\sum_{i=1}^{N_{\mathrm r}}\sum_{k:\,s_i^k\le T}\psi\!\bigl(\abs{\varepsilon_i^k}\bigr).
\]
Then
\begin{equation}\label{eq:recurrent_strong_pruning}
\E\Bigl[
\abs{X^h(T)-X(T)}^2\,\mathbf 1_{\mathcal G^{\mathrm{match}}_{\mathrm{rec}}(h,T)}
\Bigr]
\le
C_X\,N_{\mathrm r}\,\|\mathcal M^{\mathrm{rec}}(T)\|_\infty.
\end{equation}
If, in addition, $\abs{X^h(T)-X(T)}\le C_{\max}$ on
$\mathcal G^{\mathrm{match}}_{\mathrm{rec}}(h,T)^c$, then
\begin{equation}\label{eq:recurrent_strong_pruning_full}
\E\abs{X^h(T)-X(T)}^2
\le
C_X\,N_{\mathrm r}\,\|\mathcal M^{\mathrm{rec}}(T)\|_\infty
+
C_{\max}^2\sum_{i=1}^{N_{\mathrm r}}\Pp\big(B_i^{\mathrm{mis}}(h,T)\big).
\end{equation}
Combining \eqref{eq:recurrent_strong_pruning_full} with
Theorem~\ref{thm:recurrent_goodset_explicit} and \eqref{eq:mis_count_rate_recurrent} yields a
fully explicit recurrent strong bound in terms of
$(W,\Lambda^{\mathrm{rec}},L^\ast,K_{\mathrm{eff}},\rho_{\max},r^{\max},R^{(2)})$.  In
particular, the matched term already includes the contribution of small crossing speeds through the
boundary-flux averaging, so there is no separate tangency mass to balance.

\paragraph{How $L^\ast$, $K_{\mathrm{eff}}(T)$, and synaptic weights enter the matched term.}
To make the dependence transparent, consider crude uniform bounds
$\|W\|_\infty\le \mathcal W$, $\Lambda_j^{\mathrm{rec}}(T)\le \Lambda^{\mathrm{rec}}(T)$,
$C_i^{\mathrm{loc}}(T)\le C_{\mathrm{loc}}(T)$, and
$C_i^{\mathrm{dir}}(T)C_i^{\log}(T)\le C_{\mathrm{dir}}(T)$.  Then
\[
\|b(T,h)\|_\infty
\lesssim
C_{\mathrm{dir}}(T)\,C_{\mathrm{loc}}(T)\,h\bigl(1+\abs{\log h}\bigr),
\]
and
\[
\|H(T,h)\|_\infty
\lesssim
8\tauc^2\,C_{\mathrm{dir}}(T)\,\mathcal W^2\,\Lambda^{\mathrm{rec}}(T)\,
\bigl(1+\abs{\log h}\bigr).
\]
Hence
\[
\|\mathcal M^{\mathrm{rec}}(T)\|_\infty
\lesssim
h\bigl(1+\abs{\log h}\bigr)\,
\Big\|\sum_{k=0}^{K_{\mathrm{eff}}(T)}H(T,h)^k\Big\|_\infty.
\]
If $\|H(T,h)\|_\infty<1$ (a loop-stable regime), the Neumann series is uniformly bounded and the
matched MSE scales like $h$ up to the same mild logarithmic factor.  In contrast, if
$\|H(T,h)\|_\infty\gtrsim 1$, a crude bound gives
$\big\|\sum_{k=0}^{K_{\mathrm{eff}}}H^k\big\|_\infty\lesssim \|H\|_\infty^{K_{\mathrm{eff}}(T)}$,
which highlights the roles of (i) the effective event depth $K_{\mathrm{eff}}(T)$, (ii) the loop
length $L^\ast$ via the block bound \eqref{eq:loop_block_bound}, and (iii) the synaptic weight
scale $\mathcal W$ and spike-load $\Lambda^{\mathrm{rec}}(T)$.

\begin{remark}[Interpreting $L^\ast$ through weak dependence and density control]\label{rem:Lstar_interpretation}
The recurrent network can be
viewed as feedforward along a time-ordered synaptic influence chain until the chain closes.
When $L^\ast$ is large, self-feedback requires many hops and the resulting synaptic input
$\mu_j(t)$ is often approximately exogenous to neuron $j$ (in asynchronous/mixing regimes), so the
boundary-density constants $\rho_{\max}$ and $\rho^V_{\max}$ are better controlled and the proxy
bound \eqref{eq:rho_bounds_linear_surrogate} is more plausible.  When $L^\ast$ is small, however,
this decoupling is not guaranteed and \eqref{eq:rho_bounds_linear_surrogate} may fail in strongly
correlated or synchronized regimes.
\end{remark}

\begin{remark}[A crude deep-cascade growth rule]\label{rem:recurrent_pruning_exponent}
In the crude regime
$\big\|\sum_{k=0}^{K_{\mathrm{eff}}}H^k\big\|_\infty\lesssim \|H\|_\infty^{K_{\mathrm{eff}}}$, the
matched MSE inherits the factor $\|H(T,h)\|_\infty^{K_{\mathrm{eff}}(T)}$.   For fixed $T$ this does not change the leading $h$
scaling of the matched term, but it can produce very large constants when $K_{\mathrm{eff}}(T)$ is
large.
\end{remark}

\begin{remark}[Why bursty synchrony is disastrous for recurrent strong bounds]
In bursty regimes, the causal depth $K_{\mathrm{eff}}(T)$ can be large (deep cascades), and the
loop-block growth \eqref{eq:loop_block_bound} can produce large amplification even on the matched
set.  Moreover, the density bounds $\rho_{\max,i}$ and $\rho^V_{\max,i}$ can become extremely large
in small-noise synchronized regimes (and may approach singularities as noise $\downarrow 0$), while
the rate and burst constants $r_i^{\max}(T)$ and $R_i^{(2)}(T)$ can deteriorate.  Both effects
worsen the strong bound and can force the estimate to become vacuous in strongly synchronized
regimes.
\end{remark}

\begin{remark}[Spikes and strong order in recurrent networks]\label{rem:spikes_vs_half_recurrent}
The recurrent strong bounds have the same structural origin as in the feedforward matched analysis
of Section~\ref{sec:strong}: subthreshold Euler error is $O(h)$ in mean square, while threshold
events introduce only polylogarithmic corrections through the boundary-flux integral.  Ignoring
the terminal mismatch layer, the squared strong error is again $h$ times polylogarithmic factors
in $h$, so the effective strong order remains close to the classical Euler--Maruyama $1/2$ rate.

What changes in recurrence is not the local $h$-scaling but the \emph{amplification mechanism}:
feedback loops can recycle spike-time perturbations, and the matched recursion is controlled by the
shortest directed cycle length together with weight and spike-load constants
(Theorem~\ref{thm:recurrent_goodset_explicit}).  The extra contribution still comes only from the
horizon mismatch probabilities $\Pp(B_i^{\mathrm{mis}}(h,T))$.
\end{remark}

\section{Recurrent extensions: averaged weak error under rate and density bounds}\label{sec:recurrent_weak}

Weak error in recurrent networks has the same basic structure as in Section~\ref{sec:weak}, but the
one-step bookkeeping must now be done with total network spike counts.
Within-step cascades are no longer excluded by acyclicity, so the natural control parameters are the
first and second factorial moments of the total exact spike count in one step, together with the same
backward transmission and numerical strip assumptions used in the feedforward theorem.

\subsection{Spike maps and a weight norm for recurrent jumps}
When neuron $p$ spikes, it resets $v_p$ and adds the synaptic column $W_{:p}$ to the current vector.
The recurrent spike map $\mathcal{R}_p$ therefore shifts the \emph{entire} current vector.
The relevant weight norm for the one-step timing/decay defect is the column-sum norm
\[
\|W\|_{1}:=\max_{1\le p\le N_{\mathrm r}}\sum_{i=1}^{N_{\mathrm r}} |W_{ip}|.
\]
Because recurrent spike maps add fixed current columns and reset only their own voltage coordinates,
they commute pairwise in the present current-based model.

\subsection{Averaged one-step recurrent control package}
Write again $X_m^h=X^h(t_m)$.
For each step $m$, let $\widetilde X_{\mathrm{rec}}^{(m)}(t)$, $t\in[t_m,t_{m+1}]$, denote the exact
recurrent process started from $X_m^h$ at time $t_m$ and driven by fresh Brownian increments.
Let
\[
\Delta \widetilde N_{\mathrm{tot}}^{(m)}
:=
\sum_{j=1}^{N_{\mathrm r}}
\Bigl(\widetilde N_j^{(m)}(t_{m+1})-\widetilde N_j^{(m)}(t_m)\Bigr)
\]
be the exact total number of recurrent spikes in the step started from the numerical state.

For each neuron $i$, define the exact-vs.-grid decision discrepancy event
\[
D^{\mathrm{rec}}_{i,m}
:=
\Bigl\{
\mathbf 1_{\{\widetilde N_i^{(m)}(t_{m+1})-\widetilde N_i^{(m)}(t_m)\ge 1\}}
\neq
\mathbf 1_{\{v_i^{\mathrm{pre}}(X_m^h)\ge \Vth\}}
\Bigr\},
\]
and set
\[
a_i(x):=-\tfrac{1}{\tauv}(v_i-\Vr)+I_i,
\qquad
a_i^+(x):=\max\{a_i(x),0\}.
\]

\begin{assumption}[Averaged recurrent one-step rate, strip, and factorial-moment bounds]\label{ass:weak_rec_step_pkg}
There exist finite constants $R_{\mathrm{net}}(T)$, $R_{\mathrm{net},2}(T)$,
$\rho^V_{\max,i}(T)$, $A_{2,i}(T)$, and $C_{\eta,i}^{\mathrm{rec}}(T)$ such that for every step $m$
with $t_{m+1}\le T$:
\begin{enumerate}[label=(\roman*),leftmargin=2.2em]
\item
\begin{equation}\label{eq:rec_rate_first}
\E[\Delta \widetilde N_{\mathrm{tot}}^{(m)}]\le R_{\mathrm{net}}(T)\,h.
\end{equation}
\item
\begin{equation}\label{eq:rec_rate_second}
\E\!\Big[\Delta \widetilde N_{\mathrm{tot}}^{(m)}
\bigl(\Delta \widetilde N_{\mathrm{tot}}^{(m)}-1\bigr)\Big]
\le R_{\mathrm{net},2}(T)\,h^2.
\end{equation}
\item For each neuron $i$, the numerical marginal law of $V_i^h(t_m)$ admits a density
$\varrho^h_{i,m}(v)$ with
\begin{equation}\label{eq:rec_num_strip}
\sup_{m:\,t_m\le T}\sup_{y\in[0,1]}\varrho_{i,m}^h(\Vth-y)\le \rho^V_{\max,i}(T),
\qquad
\sup_{m:\,t_m\le T}\E\big[(a_i^+(X_m^h))^2\big]\le A_{2,i}(T).
\end{equation}
\item There exist nonnegative random variables $\eta^{\mathrm{rec}}_{i,m}$ such that on the event
$\{\Delta \widetilde N_{\mathrm{tot}}^{(m)}\le 1\}$,
\begin{equation}\label{eq:rec_discrepancy_strip}
D^{\mathrm{rec}}_{i,m}
\subset
\Bigl\{
0<\Vth-V_i^h(t_m)\le h\,a_i^+(X_m^h)+\eta^{\mathrm{rec}}_{i,m}
\Bigr\},
\end{equation}
and
\begin{equation}\label{eq:rec_eta_bound}
\E\big[(\eta^{\mathrm{rec}}_{i,m})^2\big]\le C_{\eta,i}^{\mathrm{rec}}(T)\,h^3.
\end{equation}
\end{enumerate}
\end{assumption}

\subsection{One-step weak defect: explicit dependence on $R_{\mathrm{net}}$ and $R_{\mathrm{net},2}$}
Let $\Phi\in C_b^4(\mathsf{D})$ and let $u(t,x)=\E[\Phi(X(T))\mid X(t)=x]$ be the backward value
function.  Retain the derivative bounds
\[
M_v(T):=\sup_{t\in[0,T)}\max_i\|\partial_{v_i}u(t,\cdot)\|_\infty,
\qquad
M_I(T):=\sup_{t\in[0,T)}\|\nabla_I u(t,\cdot)\|_\infty.
\]

\begin{lemma}[Recurrent averaged one-step weak defect]\label{lem:recurrent_onestep_weak}
Fix a step $m$ with $t_{m+1}<T$ and set $f=u(t_{m+1},\cdot)$.
Assume the subthreshold affine diffusion EM truncation satisfies
\[
\sup_{x\in\mathsf D}
\big|
(Q_h^{\mathrm{diff}}-P_{t_m,t_{m+1}}^{\mathrm{diff}})f(x)
\big|
\le
C_{\mathrm{diff}}(T)\,h^2.
\]
Then
\begin{align}
&\big|\E[u(t_{m+1},X_{m+1}^h)-u(t_m,X_m^h)]\big|
\nonumber\\
&\le
h^2\Bigg(
C_{\mathrm{diff}}(T)
+
\frac{M_I(T)+M_v(T)}{\tauc}\,R_{\mathrm{net}}(T)\,\|W\|_{1}
+
M_v(T)\sum_{i=1}^{N_{\mathrm r}}\rho^V_{\max,i}(T)\tilde{A}_{2,i}(T)
+
2\|u\|_\infty\,R_{\mathrm{net},2}(T)
\Bigg).
\label{eq:recurrent_onestep_weak}
\end{align}
\end{lemma}

\begin{proof}
The decomposition is the recurrent analogue of Lemma~\ref{lem:weak_onestep}.

\emph{(i) Subthreshold affine-diffusion truncation.}
This contributes $C_{\mathrm{diff}}(T)h^2$.

\emph{(ii) Timing/decay of a unique exact spike.}
On the event $\{\Delta \widetilde N_{\mathrm{tot}}^{(m)}=1\}$, shifting the unique exact recurrent
spike from its true time to $t_{m+1}$ changes the current vector at the endpoint by at most
$(h/\tauc)\|W\|_1$, and the induced voltage change is smaller by one more factor $h$.
Using the derivative bounds and \eqref{eq:rec_rate_first} yields the term
$\frac{M_I(T)+M_v(T)}{\tauc}R_{\mathrm{net}}(T)\|W\|_1h^2$.

\emph{(iii) Exact-vs.-grid spike-map decision on the single-spike event.}
The same argument as in Lemma~\ref{lem:weak_missed}, with
\eqref{eq:rec_num_strip}--\eqref{eq:rec_eta_bound} in place of
Assumption~\ref{ass:weak_step_pkg}, gives
\[
M_v(T)\sum_{i=1}^{N_{\mathrm r}}\rho^V_{\max,i}(T)\tilde{A}_{2,i}(T)\,h^2 .
\]

\emph{(iv) Two or more exact spikes in the step.}
On $\{\Delta \widetilde N_{\mathrm{tot}}^{(m)}\ge 2\}$ we use $|u|\le \|u\|_\infty$ and
\eqref{eq:rec_rate_second}, which gives the final
$2\|u\|_\infty R_{\mathrm{net},2}(T)h^2$ contribution.

Adding the four terms yields \eqref{eq:recurrent_onestep_weak}.
\end{proof}

\subsection{Global weak order $1$ with explicit rate dependence}
\begin{theorem}[Recurrent global averaged weak order $1$ with explicit $R_{\mathrm{net}},R_{\mathrm{net},2}$ scaling]\label{thm:recurrent_weak_global}
Let $T=Mh$ and $\Phi\in C_b^4(\mathsf{D})$.
Assume Assumption~\ref{ass:weak_backward_reg}, Assumption~\ref{ass:weak_rec_step_pkg}, and that
Lemma~\ref{lem:recurrent_onestep_weak} holds for $m=0,\dots,M-2$.
Then
\[
\big|\E[\Phi(X(T))]-\E[\Phi(X^h(T))]\big|
\le
T\,h\,C_{\mathrm{weak}}^{\mathrm{rec}}(T,W)
+
C_{\mathrm{term}}(T)\,h,
\]
where
\[
C_{\mathrm{weak}}^{\mathrm{rec}}(T,W)
=
C_{\mathrm{diff}}(T)
+
\frac{M_I(T)+M_v(T)}{\tauc}R_{\mathrm{net}}(T)\|W\|_{1}
+
M_v(T)\sum_{i=1}^{N_{\mathrm r}}\rho^V_{\max,i}(T)\tilde{A}_{2,i}(T)
+
2\|u\|_\infty R_{\mathrm{net},2}(T),
\]
and $C_{\mathrm{term}}(T)$ controls the final step $m=M-1$ (where transmission at $t=T$ need not hold).
If $\Phi$ is spike-map compatible at $T$, one may take $C_{\mathrm{term}}(T)=0$.
\end{theorem}

\begin{proof}
Use the telescoping identity
\[
\E[\Phi(X(T))]-\E[\Phi(X^h(T))]
=
\sum_{m=0}^{M-1}\E[u(t_{m+1},X_{m+1}^h)-u(t_m,X_m^h)].
\]
Apply Lemma~\ref{lem:recurrent_onestep_weak} for $m\le M-2$ and bound the final step by
$C_{\mathrm{term}}(T)h$.
\end{proof}

\begin{remark}[Effect of bursty synchrony]
The bursty synchronous regime corresponds to large $R_{\mathrm{net}}(T)$ and
$R_{\mathrm{net},2}(T)$.
The weak \emph{order} remains $1$ under these assumptions, but the constants can become large.
In addition, synchrony can inflate the numerical strip-density constants $\rho^V_{\max,i}(T)$, so the averaged weak theorem is only useful in
regimes where the recurrent one-step control package remains moderate.
\end{remark}


\section{Discussion}\label{sec:discussion}
This work provides a finite-horizon numerical theory for time-driven simulation of current-based
leaky integrate-and-fire (LIF) networks with resets.
Its main practical message is that the relevant notion of accuracy depends on the quantity of
interest.
If one studies single-trial spike trains, synchrony, causal spike order, or timing-based plasticity,
then strong error and terminal-time mismatch are the central quantities.
If one studies smooth readouts, filtered firing rates, or loss functionals built from them, then
weak accuracy is the more appropriate target.
These are different accuracy notions because the threshold/reset rule is discontinuous: a small
subthreshold state error can translate into an $O(1)$ change in a spike count or in a terminal
discontinuous observable.

Our analysis makes this separation explicit.
On matched spike histories, the strong argument does \emph{not} remove slow crossings into a separate
tangency bad set: the single-spike impact is averaged directly against the boundary flux of crossing
speeds, and only the horizon spike-count mismatch event is treated separately.
This yields matched mean-square strong error of order $h$ up to polylogarithmic factors.
The weak theorem has a different logical form.
Because a pointwise operator defect is not small near threshold, the weak proof is carried out as an
\emph{averaged} one-step estimate along the numerical chain.
Transmission makes the jump discrepancy vanish on the threshold surface, so after averaging against the
step-start voltage law the boundary strip contributes only $O(h^2)$ per step.
Under the backward transmission assumption and the one-step rate/strip/factorial-moment controls of
Sections~\ref{sec:weak} and \ref{sec:recurrent_weak}, this yields weak order $1$ with explicit
dependence on weights, stepwise rate moments, and strip-density constants.

For actual neuroscience models, our results mean that a time step adequate for rate-based summaries may
still distort synchrony statistics or spike-timing-dependent updates. This is especially relevant for cortical regimes in which irregular spiking and weakly driven threshold crossings are part of the modeled phenomenon rather than a nuisance \cite{GerstnerEtAl2014,BretteReview2007}. 

The same conclusions matter for neuromorphic computation and spike-based AI.
Many spiking architectures rely on precise event timing, synaptic filtering, and recurrent
amplification to process temporal streams or to support online learning
\cite{Roy2019,NeftciMostafaZenke2019,Eshraghian2023}.
Our bounds suggest a practical distinction: if the task is driven by smooth readouts, the averaged
order-one weak theory may already be adequate, whereas tasks that depend on single-trial timing, causal credit
assignment, or spike-triggered updates require control of strong error and of terminal mismatch probabilities.
The feedforward result that depth need not change the strong $h$-exponent, provided no new
small-speed events are created downstream.


Below we mention several important extensions.
\paragraph{More general modeling setups.} On the modeling side, conductance-based synapses, adaptation currents, heterogeneous
delays, and plasticity rules would move the theory closer to contemporary neuroscience
models and to hardware spiking systems
\cite{GerstnerEtAl2014,IndiveriEtAl2011,Eshraghian2023}.
On the numerical side, the present estimates point directly toward adaptive or
event-corrected schemes that refine only near threshold, namely in the regime where the
transversality variable is small and time-driven Euler--Maruyama is intrinsically most
fragile.

\paragraph{Density control.} At the analytical level, the key quantities in our bounds are the threshold flux and the voltage
mass in a strip below threshold ($\gamma^\ell$ and $\rho^V_{\max}$). $\gamma^\ell$ controls the density of near-tangential crossings in strong errors, and $\rho^V_{\max}$
controls weak defects. These quantities can be connected to population dynamics and statistical-mechanics descriptions of spiking networks.

In asynchronous regimes these quantities are often regular enough to be estimated from
mean-field or Fokker--Planck descriptions, whereas in synchronized or strongly
self-excitatory regimes the same descriptions can develop sharp boundary layers and even
collective firing singularities
\cite{Brunel2000,CaceresCarrilloPerthame2011,DelarueEtAl2015,HuangLiuXuZhou2021}.
Bringing such PDE estimates into the good/bad decomposition (i.e., discarding the probability of extreme synchrony) could turn qualitative
rate/density assumptions into quantitative functions of noise level, coupling strength,
and E--I balance, and would connect the present numerical analysis more directly to the
statistical mechanics of spiking populations
\cite{Bressloff2010,TouboulErmentrout2011,Touboul2014}.

\paragraph{Long time, large depth, and recurrence.}
The main limitation of the present strong theory is that it fixes $T$ and $L$ and uses a
deliberately conservative recursion:
once an upstream spike-time error appears, its maximal downstream effect is passed to all later
layers.
This is robust, but it ignores spike causality.
A perturbation in an early layer can influence a deep layer only through a later chain of
descendant spikes.
A sharper theory should therefore replace the current layerwise max recursion by a causal
path-sum or Volterra representation for the hybrid variational equation, so that error propagates
only inside an effective causal cone.
In feedforward networks this could replace the raw depth $L$ by a causal depth determined by
weights, synaptic filters, and descendant spike paths.
In recurrent networks the analogous object is a loop expansion over directed cycles, weighted by
path length and loop gain.
Under average contraction or subcritical integrated gain, one may then hope for strong bounds that
are uniform in $T$ and sublinear, or even eventually uniform, in $L$.

Weak error should also be reformulated in the long-time and large-depth regime.
The present $O(Th)$ estimate is obtained by summing one-step defects over time, which is natural on
finite horizons but not optimal as $T\to\infty$.
A more natural route is to combine the hybrid Markov semigroup with a Poisson equation under an
invariant law, potentially yielding a uniform-in-time $O(h)$ bias after mixing.
An analogous transfer-operator viewpoint in the layer index may control readout error in deep
balanced feedforward architectures without accumulating defects in $L$.
In large-population limits, one may hope to express the resulting constants through stationary
McKean--Vlasov or Fokker--Planck quantities, thereby linking numerical error for finite-network
simulation of the macroscopic order parameters studied in statistical mechanics
\cite{Bressloff2010,TouboulErmentrout2011,Touboul2014}.

Finally, the possibility of contraction/mixing dynamics discussed above naturally leads to an important direction for long-time estimate.
The current good/bad argument is intentionally one-sided: once an exact and numerical path mismatch.
For long times this is likely too pessimistic.
The Lyapunov analysis suggests that, in contracting regimes, paths may re-synchronize
after a transient STE, especially when the STE does not trigger a persistent difference in downstream spike cascades.
Constructing good/bad decompositions allowing such recovery would preserve the interpretability of the present theory while potentially yielding sharper long-time
strong bounds.
In future work, this appears to be the right framework to connect numerical analysis, dynamical
stability, and statistical mechanics in spiking neural networks.


\section{Appendix. Numerical illustrations for the strong and weak errors.}
\label{app:numerics_setup}
This appendix records the simulation protocol used for the numerical illustration for strong and weak errors.
All notation is consistent with Section~\ref{sec:model}.
In both cases, shallower networks are obtained by restricting a single simulation of the full
feedforward network to its first few layers; the deeper-layer errors therefore inherit the same
upstream spike-time perturbations as the shallow-layer errors.

\subsection{Common architecture and discretization}

All illustration runs use the current-based feedforward LIF network defined by
\eqref{eq:dv}--\eqref{eq:reset} and the time-driven EM scheme
\eqref{eq:EM_I}--\eqref{eq:EM_v}.
The numerical parameters shared by the strong and weak studies are as follows:
\begin{itemize}
    \item $\Vr=0$, $\Vth=1$, and $\tauv=1$, so that $\Ith=1$;
    \item full network depth $L=9$ and constant layer width $n_\ell=24$ for all $\ell$;
    \item In each layer, 19 neurons are set excitatory, and the rest are inhibitory. The feedforward connection probability is fixed as $p=0.25$;
    \item independent Brownian motions $\{B_{\ell,j}\}$ with noise amplitude
    $\sigma_{\ell,j}\equiv 0.25$;
    \item layerwise deterministic drive of the form $b_{\ell,j}(t)\equiv b_\ell$, where the constants
    $b_\ell$ are chosen by pilot calibration;
    \item fine reference step size
    \[
      h_{\mathrm{ref}}=2^{-10},
    \]
    and coarse step sizes
    \[
      h\in\{2^{-5},2^{-6},2^{-7},2^{-8},2^{-9}\}.
    \]
\end{itemize}

The feedforward weights are E--I balanced and normalized by $1/\sqrt{np}$.
More precisely, if neuron $i$ in layer $\ell-1$ is excitatory, then
\[
W^{\ell-1,\ell}_{ji}=\frac{c_W}{\sqrt{np}}\,\xi^{\ell-1,\ell}_{ji},
\qquad
\xi^{\ell-1,\ell}_{ji}\sim \mathrm{Bernoulli}(p),
\]
and if neuron $i$ is inhibitory, then
\[
W^{\ell-1,\ell}_{ji}=-\frac{n_E}{n_I}\frac{c_W}{\sqrt{np}}\,\xi^{\ell-1,\ell}_{ji},
\]
so that the expected signed input into each postsynaptic neuron is close to zero when excitatory
and inhibitory populations fire at comparable rates.
The constant $c_W$ depends on the scenario and is specified below.

For every Monte Carlo sample we first generate one Brownian path on the fine grid
$\{mh_{\mathrm{ref}}\}$.
For each coarse step size $h$, the corresponding coarse Brownian increments are obtained by block
summation of this same fine-grid path.
Hence all displayed values of $h$ are coupled by the same underlying random input.
Likewise, all displayed depths are extracted from one simulation of the full depth-$L$ network by
restricting the output to the first $L'\le L$ layers.
This shared-prefix/shared-noise design is essential for both Figure~\ref{fig:strong_illustration}
and Figure~\ref{fig:weak_illustration}.

\begin{table}[htbp]
\centering
\caption{Scenario-dependent parameters used in the illustration figures.}
\label{tab:illustration_scenarios}
\begin{tabular}{lcc}
\hline
 & strong-error figure & weak-error figure \\
\hline
scenario name & strong coupling & base \\
$\tauc$ & $0.35$ & $0.20$ \\
weight prefactor $c_W$ & $1.2\times 0.18$ & $0.18$ \\
\hline
\end{tabular}
\end{table}

\subsection{Strong-error illustration (Figure~\ref{fig:strong_illustration})}

Figure~\ref{fig:strong_illustration} is designed to reveal the feedforward propagation of spike-time
error across depth.
For this reason we use the strong-coupling regime from Table~\ref{tab:illustration_scenarios}, namely
$\tauc=0.35$ and weight prefactor $c_W=1.2\times 0.18$.
The reported curves correspond to prefix depths
\[
L'\in\{1,2,3,6,9\},
\]
all extracted from the same simulations of the full depth-$9$ network.

Before the production run, we perform cheap pilot simulations to choose the constants
$\{b_\ell\}_{\ell=1}^L$ and the simulation horizon $T$ so that each neuron contributes at least about
$200$ spikes on average.
This requirement is important for the strong study: the $\log h$ factors in the strong bound comes from small crossing speed $a$. The illustration is not informative unless the sample pool contains enough spikes with
crossing speed on that scale

The final Monte Carlo pool is enlarged until the following conditions hold simultaneously for every
coarse step size $h$ and every displayed prefix depth $L'$:
\begin{enumerate}
    \item each monitored neuron contributes at least $150$ reference spikes on average;
    \item each monitored neuron contributes at least $10$ reference spikes whose crossing speeds satisfy
    \[
      A^k_{\ell,j}\in [0.5,2]\sqrt{h};
    \]
    \item there are at least $400$ trajectories on the good set.
\end{enumerate}
The good set is defined on $[0, T]$ requiring that the fine and coarse spike histories match up to time $T$.
The bad set is the complement at the full horizon.

For each sample, the strong error is the terminal mean-square state discrepancy between the coarse
path and the fine reference path on the prefix network of depth $L'$.
Writing $X^{h,L'}(T)$ for the terminal state vector $(v^{h}(T),I^{h}(T))$ restricted to layers
$1,\dots,L'$, the reported quantity is the empirical average of
\[
\|X^{h,L'}(T)-X^{h_{\mathrm{ref}},L'}(T)\|_2^2
\]
up to the normalization used in the code.
Because all depths are evaluated on the same full-network trajectories, the growth of the curves
with $L'$ reflects genuine feedforward inheritance of upstream spike-time error.

\subsection{Weak-error illustration (Figure~\ref{fig:weak_illustration})}

Figure~\ref{fig:weak_illustration} uses the base regime from
Table~\ref{tab:illustration_scenarios}, namely $\tauc=0.20$ and weight prefactor $c_W=0.18$.
The reported curves correspond to prefix depths
\[
L'\in\{1,3,9\}.
\]
As in the strong study, these depths are not simulated separately; they are extracted from the same
full depth-$9$ trajectories.

The bias constants $\{b_\ell\}$ are tuned so that the empirical firing rate in each layer remains in
an intermediate regime, approximately $10$--$50$ spikes per unit time (equivalently $10$--$50$ Hz if
one unit of time is identified with one second), over a simulation lasting from a few seconds to a
much longer horizon.
The weak illustration then uses one longest simulation on $[0,T_{\max}]$ and obtains all shorter time
windows by checkpointing the same trajectories.
In the reported run,
\[
T_{\max}=111.40,
\qquad
T\in\{27.85,55.70,111.40\}.
\]
Thus the three time windows in Figure~\ref{fig:weak_illustration} are consecutive prefixes of the
same simulation and not three separate Monte Carlo experiments.

The weak observable is a smooth bounded terminal functional of the last layer of the current prefix
network:
\[
\Phi_{L'}(X(T))
=\tanh\!\bigl(c_v\,\overline v_{L'}(T)+c_I\,\overline I_{L'}(T)+c_r\,r_{L'}(T)\bigr),
\]
where $\overline v_{L'}(T)$ and $\overline I_{L'}(T)$ are the empirical means of $v_{L',j}(T)$ and
$I_{L',j}(T)$ over $j=1,\dots,n_{L'}$, and $r_{L'}(T)$ is an exponentially filtered spike readout
of the same layer with time constant $0.35$.
In the simulations we use
\[
c_v=0.8,
\qquad
c_I=0.5,
\qquad
c_r=1.
\]
For each $h$, each time window $T$, and each displayed depth $L'$, the weak error is estimated by
the paired difference
\[
\E\Bigl[\Phi_{L'}\bigl(X^{h,L'}(T)\bigr)-\Phi_{L'}\bigl(X^{h_{\mathrm{ref}},L'}(T)\bigr)\Bigr].
\]
The Monte Carlo size is selected from a pilot variance estimate, with at least $500$ paired samples
and at most $5000$.


\bibliographystyle{siam}
\bibliography{spiking_numerics_neuro_ai_refs}

\end{document}